\documentclass[onefignum,onetabnum]{siamart190516}

\usepackage{lipsum}
\usepackage{amsfonts,amsmath}
\usepackage{amsmath,amsfonts,amssymb,stackengine,graphicx}
\usepackage{graphicx}
\usepackage{epstopdf}
\usepackage{bm}
\usepackage{bbm}
\usepackage{autonum}
\usepackage{algorithm}
\usepackage{algpseudocode}
\newtheorem{assumption}{Assumption}
\usepackage{mathrsfs}
\usepackage[misc]{ifsym}

\usepackage{enumitem}

\usepackage{lineno}
\usepackage{setspace}
\usepackage{stfloats} 
\usepackage{array, multirow}
\usepackage{pgfplots}
\pgfplotsset{compat=1.15,
	height=0.8\columnwidth,
	width=\columnwidth
}
\usepgfplotslibrary{external} 
\usepackage[subrefformat=parens,labelformat=parens,caption=false]{subfig}
\usepackage{hyperref}

\definecolor{burgundy}{rgb}{0.5, 0.0, 0.13}

\usepackage[normalem]{ulem}

\newcommand{\tol}{\mathsf{tol}}
\newcommand{\R}{\mathbb{R}}

\newcommand{\yy}{y}

\newcommand{\lv}{\left \vert}
\newcommand{\rv}{\right \vert}
\newcommand{\lno}{\left \Vert}
\newcommand{\rno}{\right \Vert}

\newcommand{\codecomm}[1]{\color{gray}{\# #1}\color{black}}

\newcommand{\E}{\mathbb{E}}
\newcommand{\eff}{{\mathcal{F}}}
\newcommand{\V}{\mathbb{V}}
\newcommand{\te}{\theta}

\newcommand{\pot}{\Phi(\te;\yy)}
\newcommand{\poth}{\Phi_\ell(\te;\yy)}
\newcommand{\potl}{{\Phi_{\ell-1}(\te;\yy)}}

\newcommand{\potm}{{\Phi_{\ell}(\te;\yy)}}

\newcommand{\Q}{\mathsf{QoI}}

\newcommand{\enu}{\mathbb{E}_{\nu_{\ell}}}

\newcommand{\Pisl}{{\bm{P}_{\ell}}}
\newcommand{\Psml}{{\bm{P}^*_{\ell}}}

\newcommand{\plj}{{p_{\ell,j}}}
\newcommand{\Plj}{{P_{\ell,j}}}
\newcommand{\pisl}{{\bm{p}_{\ell}}}

\newcommand{\nuop}{\widehat{\nu}}
\newcommand{\py}{\pi^y}
\newcommand{\pyl}{\pi^y_\ell}
\newcommand{\pyll}{\pi^y_{\ell-1}}

\newcommand{\my}{\mu^y}
\newcommand{\myl}{\mu^y_\ell}
\newcommand{\myll}{\mu^y_{\ell-1}}

\newcommand{\nul}{\nu_{\ell}}
\newcommand{\ql}{Q_\ell}
\newcommand{\qoi}{\mathsf{QoI}}

\newcommand{\lev}{\mathsf{L}}
\newcommand{\pr}{\mathsf{P}}

\usepackage{afterpage}

\usepackage[final]{pdfpages}

\newcommand{\diff}{\mathrm{d}}

\newcommand{\mlQ}{\widehat{\Q}_{\mathsf{L},\{N_{\ell}\}_{\ell=0}^\mathsf{L}}}

\newcommand{\tel}{\te_{\ell,\ell}}
\newcommand{\tell}{\te_{\ell,\ell-1}}
\newcommand{\telj}{\te_{\ell,j}}
\newcommand{\teln}{\te^n_{\ell,\ell}}
\newcommand{\telln}{\te^n_{\ell,\ell-1}}
\newcommand{\tebl}{\bm{\te}_{\ell}}
\newcommand{\tebln}{\bm{\te}^n_{\ell}}

\newcommand{\calx}{\mathsf{X}}
\newcommand{\Te}{\mathsf{X}}
\newcommand{\calb}{\mathcal{B}}

\usepackage[utf8]{inputenc}
\usepackage{pgfplots}
\usepgfplotslibrary{groupplots,dateplot}
\usetikzlibrary{patterns,shapes.arrows}
\pgfplotsset{compat=newest}

\newcommand{\X}{\mathsf{X}}
\newcommand{\Y}{\mathsf{Y}}

\def\one{\mbox{1\hspace{-4.25pt}\fontsize{12}{14.4}\selectfont\textrm{1}}} 
\newcommand{\car}[1]{\one_{\left\{ #1 \right\}}}

\newcommand{\VL}{V_{\ell-1,\ell}}

\newcommand{\mup}{{\mu_\mathrm{pr}}}
\newcommand{\muy}{\mu^y}

\newcommand{\etat}{\eta^\mathsf{T}}

\ifpdf
  \DeclareGraphicsExtensions{.eps,.pdf,.png,.jpg}
\else
  \DeclareGraphicsExtensions{.eps}
\fi

\newsiamremark{remark}{Remark}
\newsiamremark{hypothesis}{Hypothesis}
\crefname{hypothesis}{Hypothesis}{Hypotheses}
\newsiamthm{claim}{Claim}
\newsiamthm{ex}{Example}
\newtheorem{example}{Example}
\headers{Analysis of a class of MLMCMC algorithms based on IMH}{Madrigal-Cianci, Nobile,  Tempone}

\title{Analysis of a class of Multi-Level Markov Chain Monte Carlo algorithms based on Independent Metropolis-Hastings \thanks{Submitted to the editors DATE.
\funding{This project is partly founded through the King Abdullah University of Science and Technology Grant OSR-2015- CRG4-2584-01: ``Advanced Multi-Level sampling techniques for Bayesian Inverse Problems with applications to subsurface''. JPMC and FN also acknowledge support from the {Swiss Data Science Center (SDSC) Grant P18-09}. RT  acknowledges support from the Alexander von Humboldt foundation.}}}

\author{ Juan P. Madrigal-Cianci\thanks{SB-MATH-CSQI, \'{E}cole Polytechnique F\'{e}d\'{e}rale de Lausanne, Lausanne Switzerland
		(\email{juan.madrigalcianci@epfl.ch, fabio.nobile@epfl.ch}).}
	\and Fabio Nobile\footnotemark[2]
	\and Ra\'ul Tempone\footnotemark[3] \thanks{Computer, Electrical and  Mathematical Sciences and Engineering, KAUST, Thuwal, Saudi Arabia, and  Alexander von Humboldt professor in Mathematics of Uncertainty Quantification, RWTH Aachen University, Aachen, Germany (\email{raul.tempone@kaust.edu.sa,tempone@uq.rwth-aachen.de}). }}

\usepackage{amsopn}

\ifpdf
\hypersetup{
  pdftitle={Analysis of a class of MLMCMC algorithms},
  pdfauthor={Juan P. Madrigal-Cianci}
}
\fi

\date{\today}
\begin{document}

\maketitle


\begin{abstract}
 In this work we present, analyze, and implement a class of Multi-Level Markov chain Monte Carlo (ML-MCMC) algorithms based on independent Metropolis-Hastings proposals for  Bayesian inverse problems. In this context, the likelihood function involves solving a complex differential model, which is then approximated on a sequence of increasingly accurate discretizations. The key point of this algorithm is to construct highly coupled Markov chains together with the standard Multi-level Monte Carlo argument to obtain a better cost-tolerance complexity than a single level MCMC algorithm. Our method extends the ideas of Dodwell, et al. ``A hierarchical multilevel Markov chain Monte Carlo algorithm with applications to uncertainty quantification in subsurface flow,'' \textit{SIAM/ASA Journal on Uncertainty Quantification 3.1 (2015): 1075-1108,} to a wider range of proposal distributions.  We present a thorough convergence analysis of the  ML-MCMC method proposed, and show, in particular, that (i) under some mild conditions on the (independent) proposals and the family of posteriors,  there exists a unique invariant probability measure for the coupled chains generated by our method, and (ii) that  such coupled chains are uniformly ergodic. We also generalize the cost-tolerance theorem of Dodwell et al., to our  wider class of ML-MCMC algorithms. Finally, we propose a self-tuning continuation-type ML-MCMC algorithm (C-ML-MCMC). The presented method is tested on an array of academic examples, where some of our theoretical results are numerically verified. These numerical experiments evidence how our extended ML-MCMC  method is  robust  when targeting some \emph{pathological} posteriors, for which some of the previously proposed ML-MCMC algorithms fail. 
\end{abstract}

\section{Introduction}\label{sec:intro}
By now, Multi-Level Monte Carlo (MLMC) methods are  well-established computational techniques \cite{giles2008multilevel} to compute expectations that arise in stochastic simulations in cases in which the stochastic model can not be simulated exactly, rather approximated, at different levels of accuracy and, as such, at different computational costs. Despite their wide-spread applicability,  extending these MLMC ideas to  Multi-Level Markov Chain Monte Carlo (ML-MCMC) methods to compute expectations with respect to (w.r.t) a complex target distribution from which independent (whether exact or approximate) sampling is not accessible, has only recently been attempted, with only a handful of works dedicated to this task. This situation arises, for instance, in Bayesian inverse problems where one wishes to compute the expectation $\E_{\my}[\Q]$ of some output quantity of interest $\Q$ with respect to the posterior measure $\my$ of some parameters $\te\in \Te$ given some indirect noise measurements $y=\eff(\te)+\eta$, where $\eta$ is the additive noise and $\eff $ is the forward operator, which may involve the solution of a differential equation. At their core, ML-MCMC methods for BIP introduce a hierarchy of discretization levels $\ell=0,1,\dots, \lev$ of the underlying forward operator, which in turn induce a family of posterior  probability  measures $\muy_\ell$, approximating $\muy$ with increasing levels of accuracy  as $\ell\to\infty$. Given some $\muy$-integrable quantity of interest $\Q,$ one can approximate the expectation of $\Q$ over $\muy$ by the usual telescoping sum argument of MLMC, namely \begin{align}
\E_{\mu^y}[\Q]\simeq \E_{\mu^y_\lev}[\Q_\lev]&=\E_{\mu^y_0}[\Q_0]+\sum_{\ell=1}^\lev\left(\E_{\mu^y_\ell}[\Q_\ell]-\E_{\mu^y_{\ell-1}}[\Q_{\ell-1}]\right)\\
&=\sum_{\ell=0}^\lev\Delta \mathsf{E}_\ell,\label{eq:tel}
\end{align}
with $\Delta\mathsf{E}_\ell:=\E_{\mu^y_\ell}[\Q_\ell]-\E_{\mu^y_{\ell-1}}[\Q_{\ell-1}]$, $\Delta\mathsf{E}_{0}=\E_{\mu^y_0}[\Q_0]$ and where, for $\ell=0,1,\dots,\lev,$ $\Q_\ell$ is a  $\muy_{\ell}$-integrable, level-$\ell$ approximation of the quantity of interest $\Q$. This telescoping sum presents the basis for various types of multi-level techniques for Bayesian inverse problems. The work \cite{hoang2013complexity}, for example, approximates the expectation \eqref{eq:tel} by splitting each $\Delta \mathsf{E}_\ell$ into three different terms, which are then computed using a mixture of importance-sampling and MCMC techniques. A multi-index generalization of such method is presented in \cite{jasra2018multi}. In addition, similar multi-level ideas have also been attempted in the context of Multi-Level Sequential Monte Carlo (MLSMC) in the works \cite{beskos2017multilevel,jasra2017multilevel,latz2017multilevel}. 

 In this work, we rather follow the approach proposed in \cite{dodwell2015hierarchical}, which is probably the first proposition of multi-level ideas for BIP and which consists in approximating $\E_{\mu^y_\lev}[\Q_\lev]$  by the following ergodic estimator:
\begin{align}
 \E_{\mu^y_\lev}[\Q_\lev]&\approx \frac{1}{N_0}\sum_{n=0}^{N_\ell}\Q_0(\te_{0,0}^{(n)})+\sum_{\ell=1}^\lev\frac{1}{N_\ell}\sum_{n=0}^{N_\ell} \underbrace{\Q_\ell(\te_{\ell,\ell}^{n})-\Q_{\ell-1}(\te_{\ell,\ell-1}^{n})}_\text{$:=Y^{n}_\ell$},\label{eq:MLMCMC_est}
\end{align}
where $\{\te^{n}_{\cdot,\ell}\}_{n=0}^{N_\ell}$ is an ergodic Markov chain with invariant distribution $\myl$. The key idea here is to couple the chains $\{\te_{\ell,\ell-1}^{n},\tel^{n}\}_{n=0}^{N_\ell}$ so that they are highly correlated and the variance of the ergodic estimator $\V[N_\ell^{-1}\sum_n Y_\ell^{n}]$ is smaller and smaller as $\ell$ increases.   By carefully choosing $N_\ell,$  this method can achieve a much better sampling complexity (in terms of cost versus tolerance) than its single level counterparts (see \cite{dodwell2015hierarchical}). 

Most of the existing literature on ML-MCMC has focused on the construction of these types of couplings \cite{cui2019multilevel,dodwell2015hierarchical}. In \cite{dodwell2015hierarchical}, the authors use (an approximation of) the posterior distribution at the previous discretization level $\ell-1$ as a proposal for level $\ell$. This is practically implemented by   sub-sampling from the chain $\{\te^{n}_{\ell-1,\ell-1}\}_{n=0}^{N_{\ell-1}}$.

%
 
 Such an idea has been recently expanded in \cite{cui2019multilevel}, where the sub-sampling idea is combined with the so-called Dimension Independent Likelihood Informed (DILI) MCMC method of \cite{cui2016dimension} to generate proposed samples at level 0 in their ML-MCMC algorithm. Some further work  combining multi-level Monte Carlo ideas with Bayesian inference has been presented in \cite{jasra2018markov}, where the authors use rejection-free Markov transitions kernels, such as Gibbs sampler, in order to couple the multi-level MCMC chains at two consecutive levels. 

 However, investigating more theoretical aspects of ML-MCMC algorithms, such as the existence of an invariant measure for the coupled chains and the type of convergence to such measure (provided it exists), have been widely overlooked, and one of the aims of this paper is to fill this gap. 
 
  This work presents several novel contributions. First, we present a ML-MCMC algorithm where chains are coupled using independent Metropolis Hastings-type proposals as in \cite{dodwell2015hierarchical},  however, allowing  for a wider class of admissible proposals. In particular, we show that the sub-sampling approach in \cite{dodwell2015hierarchical} can be replaced by a properly chosen independent Metropolis proposal (IMH) (that is, a proposal for which the proposed state is independent of the current state of the chain), which proposes the same state to the two chains $\{\te_{\ell,\ell-1}^{n},\te^{n}_{\ell,\ell}\}_{n=0}^{N_\ell}$ targeting $\myll,\myl$ respectively, which is then accepted by the usual Metropolis-Hastings criterion. This ensures the coupling of the chains. Such proposal can be, for example, the prior, a Laplace approximation, or even a kernel density approximation of the posterior at the previous level. Obviously, the choice of proposal has a direct impact on the joint invariant distribution $\nul$ of the coupled chain $\{\te_{\ell,\ell-1}^{n},\te^{n}_{\ell,\ell}\}_{n=0}^{N_\ell}$ provided it exists, hence on the variance of the ergodic estimator $N_\ell^{-1}\sum_n Y_\ell^n$.
  
  The main contribution of this work is  an in-depth convergence analysis of our extended ML-MCMC method. More precisely,  we provide sufficient conditions on the (marginal) level $\ell$ posterior and on the proposal probability measure $Q_\ell$ so that there exists a unique joint invariant probability measure for the coupled chain. Such a contribution is presented in Theorem \ref{thm:main_existence_ergo}, where it is shown that, under some mild conditions on $Q_\ell,\myl,\myll$, the ML-MCMC algorithm presented here-in (i) has a unique, invariant probability measure for the joint chain at level $\ell$ and (ii) is uniformly ergodic. Following \cite{rudolf2011explicit}, we provide also computable,  quantitative, non-asymptotic error estimators for the ergodic estimator \eqref{eq:tel}. This allows us on the one hand, to generalize the cost-tolerance result of \cite{dodwell2015hierarchical} to our extended MLMCMC method and, on the other hand, to propose an adaptive ML-MCMC algorithm in which the number of levels $\lev$ and chain lengths $N_\ell$ are determined on the fly, in the spirit of the continuation MLMC method of \cite{collier2015continuation}. 

The rest of the paper is organized as follows. In Section \ref{sec:background} we introduce some notation and standard results in the theory of Markov chains and Bayesian inverse problems.   We present our ML-MCMC method in Section \ref{sec:ml_mcmc}, and then proceed to analyze its convergence in Section \ref{sec:analysis}. Section \ref{sec:cost_analysis} is dedicated to the generalization to our case of the cost-tolerance analysis result of \cite{dodwell2015hierarchical}. In Section \ref{sec:cMLMCMC} we discuss the continuation-type algorithm and implementation details.  In Section \ref{sec:Numerical_experiments} we numerically verify some of our results in two academic examples, and, lastly, we present some conclusions and finalizing remarks in Section \ref{sec:conclussions}.

\section{Background}\label{sec:background}

\subsection{Background on Markov chains}
We recall some definitions and standard results from the theory of Markov chains. Let $(\calx, \lno \cdot\rno_\calx)$ be a separable Banach space with Borel  $\sigma$-algebra $\mathcal{B}(\calx),$ and denote by $\mathcal{M}(\calx)$ the set of  probability measures on $(\calx,\mathcal{B}(\calx))$. For some $\mu\in \mathcal{M}(\calx)$ and any $q\in[1,\infty],$ we define the Banach spaces
\begin{align}
L_q(\calx,\mu)&:=
\left\{ f:\calx\to \R \text{ $\mu$-measurable s.t. } \int_{\calx}  |f|^q(\te)\mu(\diff\te)<\infty   \right\},&\quad\text{if } q<\infty,\\
L_\infty(\calx,\mu)&:=\left\{ f:\calx\to \R \text{ $\mu$-measurable  s.t. } \underset{\te\in \calx}{\mathrm{ess} \sup} |f(\te)|<\infty   \right\},&\quad\text{if } q=\infty, 
\end{align}
endowed with the norms
\begin{align}
\lno f\rno_{ L_q(\calx,\mu)}:=&\left(\int_{\calx}|f|^q(\te)\mu(\diff\te)\right)^{1/q},\quad&\text{if } q<\infty,\\
\lno f\rno_{ L_\infty(\calx,\mu)}:=&\underset{\te\in \calx}{\mathrm{ess} \sup} |f(\te)|,\quad&\text{if } q=\infty.
\end{align}
In the particular case where $q=2$,  $L_2(\calx,\mu)$ is a Hilbert space equipped with the inner product
\begin{align}
\langle f,g\rangle_\mu:=&\int_{\calx} f(\te) g(\te)\mu(\diff \te), \ f,g\in L_2(\calx,\mu). \\
\end{align}
For any $q\in[1,\infty],$ we define the subspace $ L_q^0(\calx,\mu)$  of $ L_q(\calx,\mu)$  as 
\begin{align}
L^0_q(\calx,\mu):=\left\{ f\in  L_q(\calx,\mu) \text{ s.t. } \int_{\calx}  f(\te)\mu(\diff\te)=0   \right\}.
\end{align}
Recall that a \emph{Markov transition kernel} is a function $p:\calx\times\mathcal{B}(\calx)\rightarrow[0,1]$ such that \begin{enumerate}
	\item For each $A \in \mathcal{B}(\calx)$, the mapping $ \calx\ni\te\to p(\te,A)$, is a $\mathcal{B}(\calx)$-measurable real-valued function. 
	\item For each $\te$ in $\calx$, the mapping $\mathcal{B}(\calx)\ni A\to p(\te,A)$, is a probability measure on $(\calx,\mathcal{B}(\calx))$. 	
\end{enumerate}
Given a Markov transition kernel $p$, we denote by $P$ its associated \emph{Markov transition operator} which acts to the  left on measures, $\mu\to\mu P\in\mathcal{M}(\calx),$ and to the right on functions, $f\to Pf,$ measurable on $(\calx,\mathcal{B}(\calx)),$ such that 
\begin{align}
(\mu P) (A)&=\int_\calx p(\te,A)\mu(\mathrm{d}\te), \quad  \forall A\in  \calb(\calx),\\
(Pf)(\te)&=\int_\calx f(\vartheta)p(\te,\mathrm{d}\vartheta), \quad \forall \te \in \calx .
\end{align}
We say that a Markov operator $P$ is $\mu$-\emph{invariant} if $\mu P=\mu$. We say that a Markov operator $P$ is $\mu$-\emph{reversible} if \begin{align}
\int_B p(\te,A)\mu(\diff \te)=\int_A p(\te,B)\mu(\diff \te), \quad A,B\in \mathcal{B}(\calx).
\end{align}
 It is a well-known fact that if $P$ is $\mu$-reversible, then $P$ is $\mu$-invariant. The reverse is in general not true. The main idea behind using Markov
chain Monte Carlo methods to sample a measure of interest $\mu$ on $(\calx,\mathcal{B}(\calx)),$ is to create a Markov chain with initial state $\te^0\sim \lambda^0$, for some $\lambda^0\in\mathcal{M}(\calx),$ using a $\mu$-invariant Markov transition operator $P$. The Markov chain $\{ \te^n \}_{n\in\mathbb{N}}$ is then generated by sampling $\te^n\sim p(\te^{n-1},\cdot), \forall n\in \mathbb{N}$. Under some suitable conditions, it can be shown that $\lambda^0P^n\to \mu$ as $n\to \infty$ where the convergence is with respect to a suitable distance between probability measures.

%
%

\subsection{Bayesian inverse problems and Markov chain Monte Carlo}\label{ss:bip}
 Let $(\X,\lno\cdot\rno_\X)$ and $(\mathsf{Y},\lno\cdot\rno_\mathsf{Y})$ be separable Banach spaces with associated Borel $\sigma$-algebras $\mathcal{B}(\X),\ \mathcal{B}(\mathsf{Y})$, and let us define the forward operator $\eff:\X\rightarrow \mathsf{Y}$.   In inverse problems, we use some data $\yy\in \mathsf{Y}$, usually polluted by some random noise $\varepsilon\sim\mu_\textrm{noise},\ \varepsilon\in \mathsf{Y}$, to determine a possible state $\te\in \X$ that may have generated the data. Assuming an additive noise model, the relationship between $\te$ and $y$ is given by:
\begin{align} \label{Eq:InvProb}
{y}={\eff}(\te)+{\varepsilon}, \quad \varepsilon\sim \mu_\textrm{noise}.
\end{align}
 Here, $\te$ can be a set of parameters of a possibly non-linear Partial Differential Equation (PDE) and the data correspond to some observable (output) quantities related to the solution of the PDE.  The forward operator $\eff$ describes then the parameter-to-outputs map and its evaluation at any $\te\in \Te$ implies the solution of a PDE. As such, it is, in general,  not possible to evaluate $\eff$ exactly and one has rather to use numerical methods. We denote by $\eff_\lev$ the numerical approximation of the forward map at an accuracy level $\lev$. In a Bayesian setting, we consider the parameter $\te$ to be uncertain and model it as a random variable  with a given \emph{prior}  probability measure $\mup$ on $(\X, \mathcal{B}(\X))$. Such a prior models the knowledge we have on the uncertainty in $\te$, before observing the data $y$. For simplicity, we will assume that $\mathsf{Y}=\R^n$ for some $n\geq1,$ and that $\mu_\textrm{noise}$ has a  density $\pi_\textrm{noise},$  with respect to the Lebesgue measure, and that such a measure is strictly positive and bounded, i.e., $\exists \mathsf{A}>0$ such that $0<\pi_\textrm{noise}(z)\leq \mathsf{A},$ $\forall z\in \mathsf{Y}$. Furthermore, if we  assume that the noise $\varepsilon$ and $\te$ are statistically independent (when seen as random variables on their respective spaces), then, we have that $\mathbb{P}(y-\eff_\lev(\te)\in \cdot |\te)=\mathbb{P}(\varepsilon\in \cdot)$, i.e., $y-\eff_\lev(\te)$ conditioned on $\te$ has the same distribution as $\varepsilon$. This allows us to  define the \emph{likelihood} function \begin{equation}\pi_\lev(\yy|\te):=\pi_\text{noise}(\yy-{\eff_\lev}(\te)),\end{equation}
and  write it in terms of a  \emph{potential} function $\Phi_\lev(\te;y):\X\times \mathsf{Y}\to [0,\infty)$:
\begin{align}\label{eq:potential}
\Phi_\lev(\te;y)=-\log \left[\pi_\textrm{noise}(y-\mathcal{F}_\lev(\te))/\mathsf{A}\right].
\end{align}
The scaling factor $\mathsf{A}^{-1}$ in the definition of the potential is introduced to make it non-negative.

The function $\Phi_\lev(\te;y)$ is a measure of the misfit between the recorded data $y$ and the predicted value (at accuracy level $\lev$) $\mathcal{F}_\lev(\te)$, and often depends only on $\lno y-\mathcal{F}_\lev(\te)\rno_\mathsf{Y}$. From Bayes theorem (see e.g., \cite[Section 10.2]{dudley2018real}), it follows that the Bayesian Inverse Problem (BIP) consists in determining (or approximating)  the conditional probability measure $\my_\lev=\mathbb{P}(\te\in \cdot | y)$ of $\te\in\Te$ given the observed data $y\in\Y$, which we write in terms of its Radon-Nikodym derivative with respect to $\mup$: \begin{align}\label{Eq:BT} \py_\lev(\te)&=\frac{\diff \my_\lev}{\diff \mup}=\frac{\pi_\text{noise}(\yy-{\eff}_\lev(\te))}{\int \pi_\textrm{noise}(y-\eff_\lev(\te))\mu_\mathrm{pr}(\diff \te)}=\frac{1}{Z_\lev}\exp\left(-\Phi_\lev(\te;y)\right),\\
\text{with }Z_\lev&:= \int_\X \exp(-\Phi_\lev(\te;y))\mup(\mathrm{d}\te).
\end{align}   Notice that $\mu^y_\lev$ (resp. $\pi^y_\lev$) are a numerical approximation of a usually unattainable posterior measure (resp. density) $\mu^y$ (resp. $\pi^y$).  The goal of Bayesian analysis is often to explore the posterior measure or compute posterior expectations of some $\my$-integrable output quantity of interest $\Q:\Te\to \R$ by drawing samples from the posterior measure.   A common method for performing such a task is to use Markov chains Monte Carlo (MCMC) methods, the most celebrated of which is, perhaps, the Metropolis-Hastings algorithm \cite{metropolis1953equation,hastings1970monte}, which we describe briefly. Let $Q:\Te\times \Te\to\R_+$ be a continuous and strictly positive function such that, for any $\te\in\Te$, $\int_\Te Q(\te,z)\mup(\diff z)=1$. Thus, for any fixed $\te\in \Te,$ $Q(\te,\cdot)$ induces a probability measure $\hat{Q}(\te,\cdot)$  having Radon-Nikodym derivative with respect to the prior given by $\frac{\diff \hat Q(\te,\cdot)}{\diff\mup}(z)=Q(\te,z).$   Notice that, since $Q(\te,z)$ is assumed to be strictly positive for any $\te,z\in \Te$, $\mup$ and $\hat Q$ are equivalent probability measures. Furthermore, define the following three probability measures on $(\Te\times\Te,\mathcal{B}(\Te\times\Te))$:
\begin{align}
\eta_\lev(\diff \te,\diff z)&:=\my_\lev(\diff\te)\hat Q(\te,\diff z),\\
\etat_\lev(\diff \te,\diff z)&:=\my_\lev(\diff z)\hat Q( z,\diff \te),\\
\eta_\mathrm{pr}(\diff \te,\diff z)&:=\mup(\diff\te)\mup(\diff z).
\end{align}
 It is known (see e.g., \cite{tierney1998note}) that whenever {$\etat_\lev$} is absolutely continuous  w.r.t $\eta_\lev$ ($\etat_\lev\ll\eta_\lev$), one can construct a well-defined Metropolis-Hastings algorithm targeting $\my_\lev$. Clearly, since both $\py_\lev$ and $Q$ are strictly positive,  both $\eta_\lev$ and $\etat_\lev$ are equivalent with respect to $\eta_\mathrm{pr}$, with Radon-Nikodym derivatives give by 
\begin{align}
\frac{\diff \eta_\lev}{\diff \eta_\mathrm{pr}}(\te,z)&=\frac{\diff \my_\lev}{\diff \mup}(\te) \frac{\diff \hat Q(\te,\cdot)}{\diff \mup}(z)=\py_\lev(\te)Q(\te,z),\\
\frac{\diff \etat_\lev}{\diff \eta_\mathrm{pr}}(\te,z)&=\frac{\diff \my_\lev}{\diff \mup}(z) \frac{\diff \hat Q( z,\cdot)}{\diff \mup}(\te)=\py_\lev(z)Q(z,\te).
\end{align} As such, one can define the  so-called (level $\lev$) Metropolis acceptance ratio
\begin{align}
\frac{\diff \etat_\lev}{\diff \eta_\lev}(\te,z)=\left(\frac{\diff \etat_\lev}{\diff \eta_\mathrm{pr}}\frac{\diff \eta_\mathrm{pr}}{\diff \eta_\lev}\right)(\te,z)=\frac{\py_\lev(z) Q(z,\te)}{\py_\lev(\te)Q(\te,z)}, \label{eq:rn_metro}
\end{align}
and, similarly, one can define the (level $\lev$) Metropolis acceptance probability $\alpha_\lev: \Te \times \Te \to [0,1]$  by 
\begin{align}
\alpha_\lev(\te,z)=\min \left\{1,\frac{\diff \etat_\lev}{\diff \eta_\lev}(\te,z) \right\}.
\end{align}
The Metropolis-Hastings algorithm constructs a Markov chain $\{\te^n\}_{n=1}^{N}$ by iteratively proposing at each step $n$ a candidate state $z\sim Q(\te^n,\cdot)$, and setting $\te^{n+1}=z$ with probability $\alpha_\lev(\te^n,z)$, and otherwise setting $\te^{n+1}=\te^n.$ This procedure produces a Markov chain $\{\te^n\}_{n=1}^{N}$ whose  samples are asymptotically distributed according to $\my_\lev$, provided the chain is ergodic. The  procedure is depicted in Algorithm \ref{algo:MH}. 

Once samples $\{\te^n\}_{n=1}^N$, drawn approximately from $\mu^y_\lev$ have been obtained by some MCMC algorithm, the posterior expectation $\E_{\mu^y}[\qoi]$  can be approximated by the following ergodic estimator 
\begin{align}\label{eq:mc_estimator}
\E_{\mu^y}[\Q]\approx \E_{\mu^y_\lev}[\Q_\lev]\approx \widehat{\Q}_\lev:=\frac{1}{N} \sum_{n=1}^N \Q_\lev(\te^{n}),
\end{align}
where $\Q_\lev:\Te\to\R$ is a $\mu^y_\lev$-integrable numerical approximation of $\Q$.

\begin{algorithm}[H]
	\caption{Metropolis-Hastings}\label{algo:MH}
	\begin{algorithmic}[1]
		\Procedure{Metropolis-Hastings}{$\pi^y_\lev,Q$,$ N$,$\lambda^0$}.
		\State Sample $\te^0\sim\lambda^0$
		\For{$n=0,\dots,N-1$}
		\State Sample ${z}\sim  Q(\te^n,\cdot) $.
		\State Set $\te^{n+1}=z$ with acceptance probability 
		$$\alpha_\lev(\te^n,z)=\min \left[1,\frac{\pi^y_\lev(z)Q(z,\te^n)}{\pi^y_\lev(\te^n)Q(\te^n,z)}\right].$$
		\State Set $\te^{n+1}=\te^n$ otherwise.
		\EndFor

		\State Output $\{\te^n\}_{n=0}^{N}$ 
		\EndProcedure
	\end{algorithmic}
\end{algorithm}

The Metropolis-Hastings algorithm induces the following Markov transition kernel $p:\X\times\mathcal{B}(\X)\to[0,1]$ 
\begin{align}
p(\te,A)= \int_A\alpha_\lev(\te,z) Q(\te,z)\mup(\mathrm{d}z) + \car{\te\in A}\int_\X(1-\alpha_\lev(\te,z)) Q(\te,z)\mup(\mathrm{d}z),
\end{align}
for every  $\te\in \X$ and $A\in \mathcal{B}(\X)$. Here, $\car{\te\in A}$ is the characteristic function of the set $A$, i.e., $\car{\te\in A}=1$ if $\te\in A$ and $\car{\te\in A}=0$ otherwise. Thus, such an algorithm can be understood as iteratively applying the Markov transition operator $P$ associated to $p$ to the initial probability measure $\lambda^0$.

\section{Multi-level Markov Chain Monte Carlo}\label{sec:ml_mcmc}
In general, the Metropolis-Hastings algorithm  requires a large number of samples to converge \textemdash it is not uncommon for $N$ to be of the order of tens of thousands.  Furthermore, as it can be seen from Algorithm \ref{algo:MH}, every time $\alpha_\lev(\te^n,z)$ is evaluated, one needs to evaluate the posterior density at each new proposed state.   In PDE-based inverse problems, where evaluating $\py_\lev(z)$ implies solving a possibly non-linear and time-dependent PDE on a sufficiently fine mesh (i.e., with high accuracy), the cost associated to the MH algorithm can rapidly become prohibitive.  One way  to alleviate this issue is by introducing multi-level techniques. To that end, let $\{M_\ell\}_{\ell=0}^\lev$
be a hierarchy of discretization parameters of the underlying mathematical model $\eff(\cdot)$ in \eqref{Eq:InvProb}, which could, for instance, represent the number of degrees of freedom used in the discretization of the underlying PDE. In what follows, we consider only geometric sequences for $\{M_\ell\}_{\ell=0}^\lev$ with $M_\ell=sM_{\ell-1}$ for some $M_0>0$ and $s>1$. We denote the corresponding discretized forward models by $\eff_\ell(\cdot)$ and the corresponding approximate quantity of interest by $\Q_\ell$. We assume that the accuracy of the discretization, as well as the cost of evaluating the discretized model, increases as $\ell$ (and hence $M_\ell$) increases.  This hierarchy of discretizations, in turn, induces a hierarchy of  posterior probability measures $\{\mu^y_\ell\}_{\ell=0}^\lev$ approximating $\mu^y$ with increasing accuracy and cost. Notice that we can write the posterior expectation $\E_{\mu^y}[\Q],$ approximated on the finest available discretization level $\lev,$  in terms of the following telescoping sum: 
\begin{align}
\E_{\mu^y}[\Q]&\approx \E_{\mu^y_\lev}[\Q_\lev]=\E_{\mu^y_0}[\Q_0]+\sum_{\ell=1}^\lev\left(\E_{\mu^y_\ell}[\Q_\ell]-\E_{\mu^y_{\ell-1}}[\Q_{\ell-1}]\right).
\end{align}
This motivates introducing the following MLMCMC version of the ergodic estimator  \eqref{eq:mc_estimator}
\begin{align}
\widehat \Q_{\lev,\{N_\ell\}_{\ell=0}^\lev}:= \frac{1}{N_0}\sum_{n=0}^{N_0}[\Q_0(\te_{0,0}^n)]+\sum_{\ell=1}^\lev\frac{1}{N_\ell}\sum_{n=0}^{N_\ell}\underbrace{\left(\Q_\ell(\te^n_{\ell,\ell})-\Q_{\ell-1}(\te^n_{\ell,\ell-1})\right)}_\text{$:=Y_\ell^n$}.\label{eq:ML-MLMC_estimator}
\end{align}
 Here we have introduced a notation where, for every $\te_{\ell,j}\in \X$, the first sub-index $\ell$, $\ell=1,\dots,\lev$, represents the current level in the telescoping sum (\ref{eq:ML-MLMC_estimator}), whilst the second sub-index $j=\ell-1,\ell,$ represents the discretization  level of the posterior measure associated to $\te_{\ell,j}$, i.e., $\te_{\ell,j}\sim \mu^y_{j}$ asymptotically. Notice that the terms $Y^n_\ell$ are small, in general,  if $(\te^n_{\ell,\ell-1},\te^n_{\ell,\ell})$ are \emph{close}.  The key of the method is then to design a coupled Markov chain $\{\te^n_{\ell,\ell-1},\tel^n\}$ for which $\tell^n,\tel^n$ stay highly correlated and  close to each other for every $n$, while, at the same time, keeping the right marginal invariant distributions $\myll,\myl$, respectively. This is necessary for the terms in \eqref{eq:ML-MLMC_estimator} to telescope in the mean. 
 Constructing a coupled Markov chain (with marginal target measures $\myll,\myl$) for which $(\te_{\ell,\ell-1}^n,\te^n_{\ell,\ell})\to 0$ in a suitable sense, as $\ell\to\infty$, will in turn result in $\V_{\nul}[Y_\ell]\to 0$ as $\ell\to \infty$, where $\nul\in\mathcal{M}(\Te^2)$ is the invariant measure of the coupled Markov chain (provided it exists).   Hence, by using an adequate proposal distribution and  properly choosing $\lev$ and $\{N_{\ell}\}_{\ell=0}^{\lev}$ one can obtain a significantly better complexity than that of a single-level MCMC estimator ( see \cite{dodwell2015hierarchical} for a general complexity result of the ML-MCMC approach).  To achieve this, following \cite{dodwell2015hierarchical}, we will use what we call an \emph{Independent Metropolis-Hastings coupling} (IMH-coupling) of $\tell,\tel$.   The main idea of such a  coupling is to create two simultaneous Markov Chains $\{\te^n_{\ell,\ell-1},\te^{n}_{\ell,\ell}\}_{n\in\mathbb{N}}$  at two adjacent discretization levels, using as a proposal a probability measure $\tilde Q_\ell$ (where $\my_j\ll \tilde Q_\ell$  $j=\ell-1,\ell$), having a (strictly positive) $\mup$-density $Q_\ell$, in such a way that (i) $\tilde Q_\ell$ generates proposed states $z\in \X$ independently of the current state of either chain,  and (ii)  at every iteration, the same candidate state $z$ is proposed as the new state of both chains, which then  accept or reject it using the standard MH accept-reject step with the same uniform random variable $u\sim \mathcal{U}(0,1)$. This will in turn guarantee that, marginally $\te_{\ell,j}\sim\mu^y_j$, asymptotically for both $j=\ell-1,\ell$ (that is, the marginal chains follow the right distribution), and that the pair $(\tell^n,\tel^n)$  is highly correlated for any ${n\in\mathbb{N}}$, provided the acceptance rate is sufficiently high.  A depiction of one step of  such a  coupling procedure is shown in Algorithm \ref{algo:2-IS_coupling}.  The full multi-level MCMC procedure is presented in Algorithm \ref{algo:MLMH}. Notice that at each level $\ell=1,2,\dots,\lev$, the coupled chains $\{\tell^n,\tel^n\}_{n=0}^{N_\ell}$ in Algorithm \ref{algo:MLMH} start from the same state $\tell^0=\tel^0$ (the \textit{diagonal} of the set $\X^2$).
\begin{algorithm}[]
	\caption{one-step IMH coupling}\label{algo:2-IS_coupling}
	\begin{algorithmic}[1]
		\Procedure{IMH\_Coupling}{$\{\pi^y_{\ell-1},\pi^y_{\ell}\},\{\tell^n,\tel^n\}, Q_\ell$}
		\State Sample ${z}\sim   Q_{\ell} .$
		\State Sample $u\sim\mathcal{U}(0,1).$
		\For {$j=\ell-1,\ell$}
		\State Set $\theta_{\ell,j}^{n+1}=z$ if $u<\alpha_j(\theta_{\ell,j}^n,z)$, where
		$$\alpha_j(\theta_{\ell,j}^n,z):=\min \left[1,\frac{\pi^y_j(z) Q_\ell(\theta_{\ell,j}^n)}{\pi^y_j(\theta_{\ell,j}^n) Q_\ell(z)}\right].$$
		\State Set $\theta_{\ell,j}^{n+1}=\theta_{\ell,j}^n$ otherwise.
		\EndFor
		\State Output $\{\te^{n+1}_{\ell,\ell-1}, \te^{n+1}_{\ell,\ell} \}.$
		\EndProcedure
	\end{algorithmic}
\end{algorithm}

\begin{algorithm}[]
	\caption{Multi-level MCMC}\label{algo:MLMH}
	\begin{algorithmic}[1]
		\Procedure{ML-MCMC}{$\{\pi^y_\ell\}_{\ell=0}^\lev,  Q,\{N_\ell\}_{\ell=0}^\lev,\lambda^0$}
		\If {$\ell=0$}
		\State $\{\te^n_{0,0}\}_{n=0}^{N_0}=$\texttt{Metropolis-Hastings}($\pi^y_0, Q,N_0,\lambda^0$)
		\State Set $\chi_{0,0}=\{\te_{0,0}\}_{n=0}^{N_0}$.
		\EndIf
		\For{$\ell=1,\dots,\lev$}
		\State ``Construct'' $Q_\ell$.
		\State Sample $\tell^0\sim \lambda^0$, and set $\tel^{0}=\tell^{0}$ (e.g., from $\chi_{\ell-1,\ell-1}$.)
		\For{$n=0,\dots,N_\ell-1$}
		\State \color{gray}\# Create a coupled chain using IMH coupling \color{black}
		\State $\{\tell^{n+1},\tel^{n+1} \}=\text{\texttt{IMH\_Coupling}}(\{\pi^y_{\ell-1},\pi^y_{\ell}\},\{\tell^n,\tel^n\}, Q_\ell)$
		\EndFor
		\State Set $\chi_{\ell,j}=\{\te^n_{\ell,j}\}_{n=0}^{N_\ell}, $ $j=\ell-1,\ell$.
		\EndFor
		\State Output $\chi_{0,0}\cup\{ \chi_{\ell,\ell-1},\chi_{\ell,\ell}  \}_{\ell=1}^{\lev}$ and $\mlQ$.
		\EndProcedure
	\end{algorithmic}
\end{algorithm}
Algorithm \ref{algo:2-IS_coupling} is, effectively, a type of independent sampler Metropolis \cite{asmussen2007stochastic} on the marginal chains. As such, the sampling efficiency of such an algorithm will critically depend on how well the proposal $\tilde Q_\ell$  \textit{approximates} $\mu^y_\ell$ and $\myll$; choosing a proposal $\tilde Q_\ell$ that closely resembles $\mu^y_\ell$ or $\mu^y_{\ell-1}$ will reduce the amount of rejection steps, hence enhancing the mixing of the chains (see \cite{asmussen2007stochastic,brooks2011handbook} for a more in-depth discussion on this). In principle, $\tilde Q_\ell$ can be chosen to be, e.g., the prior, an empirical version of the posterior based on the sample $\{\te^n_{\ell-1,\ell-1}\}_{n=0}^{N_\ell}$ collected at the previous level, as originally proposed in \cite{dodwell2015hierarchical}, or any reasonable  approximation of $\mu^y_\ell, \ \mu^y_{\ell-1}$ such as, e.g., a Laplace approximation or a kernel density estimator (KDE), again based on the sample $\{\te^n_{\ell-1,\ell-1}\}_{n=0}^{N_\ell}$ collected at the previous level. 

  Notice that each step of Algorithm \ref{algo:2-IS_coupling} produces one out of 4 possible configurations $S_1,S_2,S_3,S_4,$ namely;\begin{align}
&S_1:(\tell^{n+1},\tel^{n+1})=(z,z) \quad&\text{(both chains accept the proposed state)},\\
&S_2:(\tell^{n+1},\tel^{n+1})=(z,\teln) \quad&\text{(chain al level $\ell-1$ accepts and chain at level $\ell$ rejects)},\\
&S_3:(\tell^{n+1},\tel^{n+1})=(\telln,z) \quad&\text{(chain al level $\ell-1$ rejects and chain at level $\ell$ accepts)},\\
&S_4:(\tell^{n+1},\tel^{n+1})=(\telln,\teln) \quad&\text{(both chains reject the proposed state)}.
\end{align} 

 This is illustrated in Figure \ref{fig:drawing}. More formally, set $\Te^2\ni\tebln:=(\tell^n,\tel^n).$
Then, for any $A\in\calb(\Te^2)$, Algorithm \ref{algo:2-IS_coupling} induces the \emph{multilevel Markov transition kernel}  $\pisl:\Te^2\times\calb(\Te^2)\mapsto[0,1]$ given by 
\begin{align}\label{eq:joint_kernel}
&\pisl(\tebln,A):=\int_{\Te}\min\{\alpha_{\ell-1}(\telln,z),\alpha_\ell(\teln,z) \}\ql(z)\car{(z,z)\in A}\mup(\diff z)\\
&+\int_\Te (\alpha_{\ell-1}(\telln,z)-\alpha_{\ell}(\teln,z))^+\ql(z)\car{(z,\teln)\in A}\mup(\diff z)\\
&+\int_\Te (\alpha_\ell(\teln,z)-\alpha_{\ell-1}(\telln,z))^+\ql(z)\car{(\telln,z)\in A}\mup(\diff z) \\
&+\car{(\telln,\teln)\in A}\left(1-\int_\Te \max\{\alpha_{\ell-1}(\telln,z),\alpha_\ell(\teln,z) \}\ql(z)\mup(\diff z)\right),
\end{align}
 where $(x)^+:=\frac{x+|x|}{2},\ x \in \R.$ Notice that each line on the right hand side of \eqref{eq:joint_kernel} corresponds to the transition kernel proposing to move from the state $\tebln$ to one of the 4 possible configurations $S_i,\  i=1,2,3,4$.   Although it is clear that $\pisl$ targets the right marginals, the properties related to the convergence of the chain generated by $\pisl$, such as irreducibility, existence of an invariant (joint) measure $\nul$, or geometric ergodicity,  are not obvious. We will investigate these convergence properties on the following section. 
\begin{figure}[htpb]
	\centering
	\includegraphics[width=0.5\linewidth]{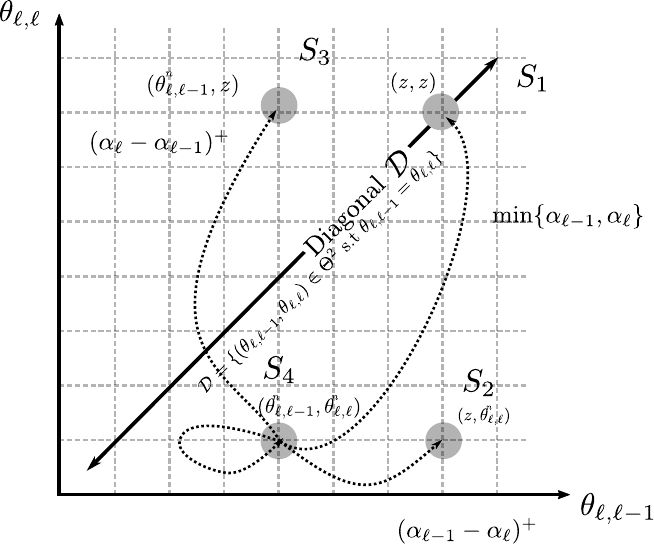}
	\caption{Schematic of the possible configurations $S_1,S_2,S_3,S_4$. Note that the sampler moves to the diagonal $\Delta=\{(\tell,\tel)\in\Te^2\text{ s.t }\tell=\tel\}$ whenever both chains accept (regardless of their current state) or when both chains reject, assuming that they were at the diagonal. }
	\label{fig:drawing}
\end{figure}

\section{Convergence analysis of the ML-MCMC algorithm}\label{sec:analysis}

We  now proceed to analyze the convergence of the level-wise coupled chains generated by  Algorithm \ref{algo:2-IS_coupling}. The main result of this section is stated in Theorem \ref{thm:main_existence_ergo}. Loosely speaking, this theorem (i) gives conditions for  the existence and uniqueness of a joint invariant measure of the multi-level Markov transition kernel \eqref{eq:joint_kernel}, and (ii) shows that, under certain conditions, such a kernel generates a uniformly ergodic chain, i.e., a chain that converges exponentially fast to its invariant distribution with a constant that does not depend on the initial state of the chain.

 Notice that, at each level $\ell$,  Algorithm \ref{algo:2-IS_coupling} creates two coupled chains using the same proposal $Q_\ell$. This in turn induces two Markov transition kernels, each generating a marginal chain. We formalize this in the following definition. 

\begin{definition}[Marginal kernel]
For a given level $\ell$, $\ell=1,2\dots,\lev$ and proposal $Q_\ell$, we define the $\my_j$-invariant \emph{Marginal Markov Transition Kernel} $p_{\ell,j}:\Te\times\mathcal{B}(\Te)\rightarrow[0,1],$ with $j=\ell-1,\ell$,  as 
\begin{align}\label{eq:marginal_kernel}
p_{\ell,j}(\te_{\ell,j},A)&:= \int_A\alpha_j(\telj,z)Q_\ell(z)\mup(\mathrm{d}z) \\&
+ \car{\te_{\ell,j}\in
	 A}\int_\X\left(1-\alpha_j(\telj,z)\right)Q_\ell(z)\mup(\mathrm{d}z),
\end{align}
for any $\te_{\ell,j}\in \Te$, and $A\in \mathcal{B}(\Te).$ Similarly, we denote by $P_{\ell,j}$ its corresponding \emph{marginal Markov Transition Operator}.  
\end{definition}

Clearly, the marginal chain $\{\te^n_{\ell,j}\}_{n=0}^{N_\ell},\ j=\ell-1,\ell$ generated by \eqref{eq:marginal_kernel} is indeed a Markov chain. Furthermore, notice that, by construction, $\Plj$ is $\mu^y_j$-invariant, i.e., $\mu^y_j \Plj=\mu^y_j$.

We make the following  assumptions on the proposal and the (marginal) posterior densities.
\begin{assumption}[Assumptions on positivity and continuity]\label{Ass:positivity}
	The following  conditions hold for all $\ell=1,\dots,\lev$:
	\begin{enumerate}[label=\emph{{1.\arabic*}}.,ref=1.\arabic*]
		\item\label{Ass:prop_is_positive} The proposal density is continuous and positive, i.e., $\ql(z)>0,\ \forall z\in\Te$.
		\item\label{Ass:post_is_positive} The posterior density is continuous and positive, i.e.,  $\pyl(z)>0,\ \forall z\in\Te$.
		\item \label{Ass:compact_level_sets} For any $\ell=1,2,\dots,\lev$, $j=\ell-1,\ell$, the quotient $Q_\ell(\te)/\py_j(\te)$ has compact level sets, i.e., for any $c_r>0,$ the set $A\in\calb(\X)$ defined as $A:=\left\{\te\in \Te: \frac{Q_\ell(\te)}{\py_j(\te)}\leq c_r\right\}$ is compact.  
		\item \label{Ass:essinf} There exists a positive constant $c\in(0,1),$ independent of $\ell$, such that  $$\mathrm{ess}\inf_{z\in\Te} \left\{\ql(z)/\py_j(z)\right\}\geq c>0,\quad j=\ell-1,\ell.$$
		\item \label{Ass:p_moments} There exist positive constants $r>1,$ and $C_r$, independent of $\ell$, such that $\int_\Te Q_\ell^r(\te)\mup(\diff\te)\leq C_r,$ for any $\ell$.
	\end{enumerate}
\end{assumption}

We remark that Assumptions \ref{Ass:prop_is_positive} and \ref{Ass:post_is_positive} are relatively mild and straightforward to satisfy. In particular, under the assumption that $\Phi_\ell(\te;y)$ is positive everywhere, \ref{Ass:post_is_positive} is ``inherited'' from the forward operator $\eff$ and $\mu_\mathrm{noise}$, while \ref{Ass:prop_is_positive} is user defined. 
Assumption \ref{Ass:compact_level_sets} is common in the study of converge of Markov chains when using coupling arguments (see, e.g., \cite{roberts2004general}). It is satisfied, for instance, if $\Te$ is compact or if it is finite-dimensional and $\lim_{|\te|\to\infty} \ql(\te)/\pyl(\te)=+\infty$.  Assumption \ref{Ass:essinf} implies, in particular, that the tails of the proposal $\ql$ must decay more slowly than those of $\myl$, $\myll$ at infinity, i.e., $Q_\ell$ has heavier tails than $\my_j,\ j=\ell-1,\ell$. In practice, this is indeed a more restrictive assumption than \ref{Ass:prop_is_positive} and \ref{Ass:post_is_positive}. It is, however, crucial for the convergence of both the marginal IMH and the multi-level algorithm. Lastly, Assumption \ref{Ass:p_moments} is an integrability condition on $Q_\ell$ with respect to the prior. This assumption is quite mild and will become useful in the next section (c.f. Lemma \ref{lemma:prob_bnd}).

\begin{remark}
It is a consequence of Assumptions \ref{Ass:prop_is_positive}-\ref{Ass:compact_level_sets} that $\mathrm{ess}\inf_{z\in\Te} \left\{\ql(z)/\py_j(z)\right\}\geq c_\ell>0$ with a possibly level-dependent constant. Assumption \ref{Ass:essinf} requires that such  constants $c_\ell$ are uniformly lower-bounded by a strictly positive constant $c$ when $\ell\to\infty$. 
\end{remark}

\subsection{Convergence of the level-wise coupled chain}

We say that a Markov kernel $p:\calx\times\mathcal{B}(\calx)\rightarrow[0,1]$ is \emph{$\psi$-irreducible} if there exists a non-zero, $\sigma$-finite measure $\psi$ on $(\calx,\calb(\calx))$ such that for all measurable sets $A\in\calb(\calx)$ with $\psi(A)>0$, and for all $\te\in\calx,$ there exists a positive integer $n$, possibly depending on $\te$ and $A,$ such that $p^n(\te,A)>0$. We say that a set $S\in\calb(\calx)$ is \emph{small} (or, more precisely, \emph{$\nu_m$-small}) if there exists $m\in\mathbb{N}$ and a non-trivial, positive measure $\nu$ on $ (\Te,\mathcal{B}(\calx))$ such that \begin{align}
p^m(\te,A)\geq\nu(A), \quad \forall \te\in S,\ A\in\calb({\calx}).
\end{align}

In most Markov chain Monte Carlo methods, one typically designs a Markov chain with a given invariant probability measure. This automatically ensures the existence of (at least) one invariant probability measure. However, this is unfortunately not the case for Multi-level MCMC algorithms (including the one presented here), and as such, a significant part of the paper will be devoted to show that such an invariant measure uniquely exists. We recall here a standard result in the theory of Markov chains (see e.g., \cite[Theorem 15.0.1]{meyn2012markov}) which guarantees the existence of a unique invariant probability measure, and which Theorem \ref{thm:main_existence_ergo} relies upon. 
\begin{theorem}[Existence and uniqueness of an invariant probability measure]\label{thm:existence_uniquenes_bax}
	Consider a (discrete-time) Markov chain on a general state space $(\calx,\mathcal{B}(\calx))$ with a Markov transition kernel $p:\calx\times \calb(\calx)\to[0,1].$ Suppose the following conditions hold: 
	\begin{enumerate}[label=\emph{{A\arabic*}}.,ref=A\arabic*]
		\item\label{Ass:Small_exs} The chain is $\psi$-irreducible and aperiodic.
		\item\label{Ass:Strong_drift_exs} There exists a measurable function $V:\calx\to[1,\infty)$, called \emph{Lyapunov function}, a small set $S\in\mathcal{B}(\calx),$ and  positive constants $\lambda<1$, $\kappa<\infty$ satisfying
		\begin{align}
		(PV)(\te)\leq\lambda V(\te)+\kappa \emph{\one}_{\{\te\in S\}}. \label{eq:drift_cond}
		\end{align}
	\end{enumerate}
	Then, $p$ has a unique invariant probability measure $\mu\in\mathcal{M}(\calx,\mathcal{B}(\calx))$. In addition,  there exist positive constants $M<\infty$ and $r\in(0,1)$ such that  
	\begin{align}
	\underset{|f|\leq V}{\sup}\left| (P^nf)(\te)-\int_{\calx}f(\te)\mu(\diff \te)\right|\leq MV(\te)r^n,  \quad \forall \te\in	\calx,\ n\in\mathbb{N}, \label{eq:v_geom_erg}
	\end{align}
	where the supremum is taken over all $\mu$-integrable functions $f$ satisfying $|f|\leq V$.
\end{theorem}
Markov chains that satisfy \eqref{eq:v_geom_erg} are called \emph{geometrically ergodic}. Markov chains for which the bound on the right hand side of \eqref{eq:v_geom_erg} does not depend on $\te$ are called \emph{uniformly ergodic}.   The following result shows the reverse implication
\begin{theorem}[Bounded Lyapunov function, \cite{meyn2012markov}] \label{thm:bdd_lyap}
	A  Markov chain on a general state space $\Te$ is uniformly ergodic if and only if it satisfies a drift condition of the form \eqref{eq:drift_cond} with a bounded Lyapunov function $V:\Te\to[1,V_\mathrm{max}], \ V_\mathrm{max}\in\R_+$, and a small set $S_\ell$. 
	\begin{proof}
		See \cite[Theorem 16.0.2, implication viii]{meyn2012markov}. 
	\end{proof}
\end{theorem}
 Thus, from Theorem \ref{thm:bdd_lyap}, we have that for each level $\ell=0,1,\dots,\lev,$ the marginal kernels $\plj$, satisfy a drift condition of the form \eqref{eq:drift_cond} with  a bounded Lyapunov function $V_j$ and a small set $S_j$, $j=\ell-1,\ell$. Chains that are geometrically ergodic and reversible have a positive $L_2(\calx,\mu)$-spectral gap \cite{roberts1997geometric} (c.f. Equation \eqref{eq:spec_gap}).

Lastly, we recall a result regarding the convergence of probability measures on general Banach spaces. From the definition of Markov transition kernel, we know that the Markov transition kernel $p(\cdot,\cdot)$  induces a probability measure $p(\te,\cdot)$ on $(\Te,\mathcal{B}(\Te))$ for any $\te\in \Te$. We say that such a measure $p(\te,\cdot)$ is  \emph{tight} if for each $\epsilon >0$ there exists a compact set $K\in \mathcal{B}(\Te)$ such that $p(\te,K)>1-\epsilon.$  It is known \cite{billing} that a $p(\te,\cdot)$ is tight if and only if for each $A\in \mathcal{B}(\Te)$, $p(\te,A)$ is the supremum of $p(\te,K)$ over the compact subsets $K$ of $A$. In addition, we have the following result, which will become useful for the proof of Theorems \ref{thm:main_existence_ergo} and \ref{Thm:abs_cont}.

\begin{theorem}\label{thm:billingsley}
	If $\Te$ is a separable Banach space, then, every probability measure on $(\Te,\mathcal{B}(\Te))$ is tight. 
	\begin{proof}
		See \cite[Theorem 1.3]{billing}
	\end{proof}
\end{theorem}

We present the main result of this subsection.

\begin{theorem}\emph{(Uniform ergodicity of the coupled chain)}\label{thm:main_existence_ergo} Suppose that Assumption  \emph{\ref{Ass:positivity}}  holds. Then, for any level $\ell=1,2,\dots,\lev$,  \begin{enumerate}
		\item The joint Markov chain generated by $\Pisl$ in \eqref{eq:joint_kernel} has a unique invariant probability measure $\nul$ on $(\Te^2,\calb(\Te^2))$.
		\item The joint Markov chain generated by $\Pisl$ is uniformly ergodic. That is, there exists a bounded, $\nul$-integrable, bi-variate Lyapunov function $\VL:\Te^2\to[1,V_{\mathrm{max},\ell}],$ an $r\in(0,1)$ and $M\in\R^+$, such that
		\begin{align}
		\underset{|f|\leq V_{\ell-1,\ell}}{\sup}\left| (\bm P_\mathrm{\ell}^nf)(\tebl)-\int_{\Te^2}f(\tebl)\nul(\diff \tebl)\right|\leq MV_{\mathrm{max},\ell}r^n,  \quad \forall \tebl\in	\Te^2,\ n\in\mathbb{N},
		\end{align}
		where the supremum is taken over all $\nul$-measurable functions $f:\Te^2\to\R$ satisfying $|f(\tebl)|\leq \VL(\tebl)$, $\forall \tebl \in \Te^2.$  
	\end{enumerate}
\end{theorem}	

We postpone the proof of Theorem \ref{thm:main_existence_ergo} to the end of the this subsection. The main idea behind the proof is to show that Assumption \ref{Ass:positivity}  implies that the joint kernel $\pisl$ satisfies Assumptions \ref{Ass:Small_exs} and \ref{Ass:Strong_drift_exs} of Theorem \ref{thm:existence_uniquenes_bax}. We begin by recalling a well-known result in the theory of Markov chains.

\begin{theorem}[Uniform ergodicity of IMH]\label{thm:unif_erg} For any $\ell=1,2,\dots,\lev$ and $j=\ell-1,\ell$, let $p_{\ell,j}:\Te\times \mathcal{B}(\Te)\to[0,1]$ denote the $\my_j$-reversible Markov transition kernel associated to an Independent Metropolis Hastings algorithm with proposal $\ql$. If $\ql$ and $\my_j$ are such that $\mathrm{ess}\inf_{z\in\Te} \left\{\ql(z)/\py_j(z)\right\}>0$, then, $p_{\ell,j}$ is uniformly ergodic. Conversely, if $\mathrm{ess}\inf_{z\in\Te} \left\{\ql(z)/\py_j(z)\right\}=0,$ then, $p_{\ell,j}$ fails to be ergodic in the sense of \eqref{eq:v_geom_erg}.  
	\begin{proof}
		See \cite[Theorem 2.1]{mengersen1996rates}.
	\end{proof}
\end{theorem}
Thus, from Theorem \ref{thm:unif_erg}, it can be seen that Assumption \ref{Ass:essinf} implies uniform ergodicity of the marginal chains of the multi-level MCMC algorithm.

We proceed to  investigate the convergence of the coupled Markov chain generated by Algorithm \ref{algo:2-IS_coupling} and start by studying its irreducibility. Loosely speaking, Assumption \ref{Ass:positivity} implies that  \eqref{eq:joint_kernel} is positive $\forall \tebl^n\in\Te$ and $A\in\mathcal{B}(\Te^2)$. We proceed to show this more formally.
\begin{lemma}\emph{(Irreducibility of the joint chain)}\label{lemma:irreducibility}
Suppose Assumption \emph{\ref{Ass:positivity}} holds. Then, for any $\ell=1,2,\dots,\lev$, the joint Markov transition kernel $\pisl$ defined in \eqref{eq:joint_kernel} is $\psi$-irreducible. 
\begin{proof}
 Let $A\in \calb(\Te^2)$ be a measurable set, and let $K\in\mathcal{B}(\Te)$ be any compact set with non-zero $\mup$ mass, i.e., $\int_K \mup(\diff \te)>0$. Lastly, define the (also compact) set $K^2:=K\times K,$ $K^2\in\mathcal{B}(\Te^2)$, and  $A_K:=A\cap K^2$. 
 For a given state $\tebl\in\Te^2$ we have that  \begin{align}
\pisl(\tebl,A)\geq\pisl(\tebl,A_K)\geq\int_{\Te}\min\{\alpha_{\ell-1}(\tell,z),\alpha_\ell(\tel,z) \}\ql(z)\car{(z,z)\in A_K}\mup(\diff z).
\end{align}
In addition, let $\Te_{\ell-1;\tebl}:=\{z\in\Te\text{ s.t. }\alpha_{\ell-1}(\tell,z)\leq\alpha_{\ell}(\tel,z)\},$ and, similarly, 	$\Te_{\ell;\tebl}:=\{z\in\Te\text{ s.t. }\alpha_{\ell-1}(\tell,z)>\alpha_{\ell}(\tel,z)\}.$ Then we have that 
\begin{align}
&\int_{\Te}\min\{\alpha_{\ell-1}(\tell,z),\alpha_\ell(\tel,z) \}\ql(z)\car{(z,z)\in A_K}\mup(\diff z)=\\
&\int_{\Te_{\ell-1;\tebl}}\alpha_{\ell-1}(\tell,z)\ql(z)\car{(z,z)\in A_K}\mup(\diff z)\\
&+\int_{\Te_{\ell;\tebl}}\alpha_\ell(\tel,z)\ql(z)\car{(z,z)\in A_K}\mup(\diff z). \label{eq:integrals_irred}
\end{align}
Since  $\ql(z),\pyll(z),\pyl(z)$ are continuous and positive $\forall z \in \Te$, it follows that there exist  $\epsilon,\epsilon',\epsilon''>0$ such that $\epsilon\leq\ql(z),$  $\epsilon'\leq\pyll(z),$ and $\epsilon''\leq\pyl(z),\  \forall z\in K,$ and hence $\forall (z,z)\in A_K$. Thus, we have that 
\begin{align}
&\int_{\Te_{\ell-1;\tebl}}\alpha_{\ell-1}(\tell,z)\ql(z)\car{(z,z)\in A_K}\mup(\diff z)\\ &\geq  \int_{\Te_{\ell-1;\tebl}}\min\left\{\epsilon,\frac{\pyll(z)}{\pyll(\tell)}\ql(\tell) \right\}\car{(z,z)\in A_K}\mup(\diff z)\\
&\geq f_{\ell-1}(\tell,\epsilon')\int_{\Te_{\ell-1;\tebl}}\car{(z,z)\in A_K}\mup(\diff z),
\end{align}
with $f_{\ell-1}(\tell,\epsilon'):=\min\left\{\epsilon,\frac{\epsilon'\ql(\tell)}{\pyll(\tell)} \right\}>0$. A similar bound holds for the second integral in \eqref{eq:integrals_irred},  namely \begin{align}
\int_{\Te_{\ell;\tebl}}\alpha_\ell(\tel,z)\ql(z)\car{(z,z)\in A_K}\mup(\diff z)\geq f_\ell(\tel,\epsilon'')\int_{\Te_{\ell;\tebl}}\car{(z,z)\in A_K}\mup(\diff z),
\end{align}
with $  f_\ell(\tel,\epsilon''):=\min\left\{\epsilon,\frac{\epsilon''\ql(\tel)}{\pyl(\tel)} \right\}>0$. Setting $\hat f(\tebl,\epsilon,\epsilon',\epsilon''):=\min\{f_{\ell-1},f_\ell\}$ gives the bound 
\begin{align}
\pisl(\tebl,A)\geq f(\tebl,\epsilon,\epsilon',\epsilon'')\int_\Te\car{(z,z)\in A_K}\mup(\diff z).
\end{align}
Choosing $\psi(A)=\int_\Te\car{(z,z)\in A\cap K^2}\mup(\diff z),$ gives the desired result.  
\end{proof}
\end{lemma}

%
%

We now proceed to show the existence of small sets for the joint kernel.  In words,  the next Lemma \ref{lemma:small_set} states that compact subsets of $\Te^2$ are small sets, provided Assumption \ref{Ass:positivity} holds true. 

\begin{lemma}\emph{(Existence of small sets)}\label{lemma:small_set} Let $\hat S_\ell\in \calb(\Te^2)$ be a compact subset and suppose \emph{Assumption \ref{Ass:positivity}} holds. Then, $\hat S_\ell$ is a small set for the Markov kernel $\pisl$.
\begin{proof} We follow an argument similar to that in \cite[Section 3]{roberts2004general}. Let $\tebl=(\tell,\tel)\in \hat S_\ell,$ and let $B\in\calb(\Te^2)$ be some measurable set. Since $\hat S_\ell$ is compact, Assumption \ref{Ass:positivity} ensures that there exist positive constants $\epsilon_{\ell}, \epsilon_{\ell}'\in \R^+$ such that $0<\epsilon_{\ell}\leq \ql(\telj)$ and $\pi^y_j(\telj)\leq \epsilon'_{\ell}$, $j=\ell-1,\ell$. Then, we have that: 
	\begin{align}
	&\pisl(\tebl,B)\geq \int_\Te\min\left\{\alpha_{\ell-1}(\tell,z),\alpha_{\ell}(\tel,z)\right\}\ql(z)\car{(z,z)\in B}\mup(\diff z)\\
	&\geq \int_\Te\min\left\{\min\left\{\ql(z),\pyll(z)\frac{\epsilon_{\ell}}{\epsilon'_{\ell}}\right\},\min\left\{\ql(z),\pyl(z)\frac{\epsilon_{\ell}}{\epsilon'_{\ell}}\right\}\right\}\car{(z,z)\in B}\mup(\diff z)\\
	&=\nu_\mathrm{small}(B),
	\end{align}
	where we have denoted 
	\begin{align}
	&\nu_\mathrm{small}(B):=\\
	&\int_\Te\min\left\{\min\left\{\ql(z),\frac{\pyll(z)\epsilon_{\ell}}{\epsilon'_{\ell}}\right\},\min\left\{\ql(z),\frac{\pyl(z)\epsilon_{\ell}}{\epsilon'_{\ell}}\right\}\right\}\car{(z,z)\in B}\mup(\diff z).
	\end{align}
\end{proof}	
\end{lemma}	

By definition of strong aperiodicity \cite[Section 5.4]{meyn2012markov}, we have from Lemmata \ref{lemma:irreducibility} and \ref{lemma:small_set}  that the chain generated by $\pisl$ is strongly aperiodic, and hence, aperiodic \cite[Section 5.4]{meyn2012markov}.

 We now proceed to show an auxiliary result that we will use to prove ergodicity of our multi-level sampler. 
\begin{lemma}\label{lemma:drift_condition_aux} Suppose Assumption \emph{\ref{Ass:positivity}} holds, let $V_j:\Te\to[1,\infty]$, $j=\ell-1,\ell$, be a $\mup$-integrable function and suppose $V(\tebl)=V_j(\telj)$.  Then, 
	\begin{align}
	(\Pisl V)(\tebl)=(P_{\ell,j}V_j)(\telj).
	\end{align}
	\begin{proof} Assume, without loss of generality, that $j=\ell.$ From \eqref{eq:joint_kernel} we have that:
	\begin{align}
	(\Pisl V)(\tebl)=&\int_{\Te}\min\{\alpha_{\ell-1}(\tell,z),\alpha_\ell(\tel,z) \}V_\ell(z)\ql(z)\mup(\diff z)&\mathrm{(I)}\\
	+&\int_\Te (\alpha_{\ell-1}(\tell,z)-\alpha_{\ell}(\tel,z))^+V_\ell(\tel)\ql(z)\mup(\diff z) &\mathrm{(II)}\\
	+&\int_\Te (\alpha_{\ell}(\tel,z)-\alpha_{\ell-1}(\tell,z))^+V_\ell(z)\ql(z)\mup(\diff z) &\mathrm{(III)}\\
	+&V_\ell(\tel)\left(1-\int_\Te \max\{\alpha_{\ell-1}(\tell,z),\alpha_\ell(\tel,z) \}\ql(z)\mup(\diff z)\right).&\mathrm{(IV)}
	\end{align}
	 On an attempt to simplify notation, we introduce $\mathsf{D}_{\ell-1}(\tebl,z):=\alpha_{\ell-1}(\tell,z)-\alpha_{\ell}(\tel,z)$ and  $\mathsf{D}_{\ell}(\tebl):=-\mathsf{D}_{\ell-1}(\tebl).$ We analyze each of the four terms (I), (II), (III), (IV) separately.  We begin with term (I). Since $\min\{a,b\}=\frac{a+b-|a-b|}{2}\ \forall a,b\in \R,$ we have that
	\begin{align}
	\mathrm{(I)}&=\frac{1}{2}\left(\int_\Te\alpha_{\ell-1}(\tell,z)\ql(z)V_\ell(z)\mup(\diff z)+\int_\Te\alpha_\ell(\tel,z)\ql(z)V_\ell(z)\mup(\diff z)\right. \\&\left.-\int_\Te |\mathsf{D}_{\ell-1}(\tebl,z)|V_\ell(z)\ql(z)\mup(\diff z)\right). \label{eq:I}
	\end{align}
	Moreover, since $(a)^+=\frac{a+|a|}{2}, \forall a\in \R,$  we can write (II) and (III) as:
		\begin{align}
		\mathrm{(II)}=&V_\ell(\tel) \int_\Te\frac{\mathsf{D}_{\ell-1}(\tebl,z)+|\mathsf{D}_{\ell-1}(\tebl,z)|}{2}\ql(z)\mup(\diff z),\label{eq:II}\\
		\mathrm{(III)}=&\int_\Te\frac{\mathsf{D}_{\ell}(\tebl,z)+|\mathsf{D}_{\ell}(\tebl,z)|}{2}V_\ell(z)\ql(z)\mup(\diff z)\label{eq:III}.
		\end{align}	
	Adding Equations \eqref{eq:I} and \eqref{eq:III} gives 
	\begin{align}
	\text{(I)+(III)}=\int_\Te\alpha_\ell(\tel,z)\ql(z)V_\ell(z)\mup(\diff z) \label{I+III}.
	\end{align}
Furthermore,  using the fact that $\max\{a,b\}=\frac{a+b+|a-b|}{2},\ \forall a,b\in \R,$  we can write (IV) as 
	\begin{align}\label{eq:IV}
	\mathrm{(IV)}&=V_\ell(\tel)-V_\ell(\tel)\int_\Te\frac{\alpha_{\ell-1}(\tell,z)+\alpha_{\ell}(\tel,z)+|\mathsf{D}_\ell(\tebl,z)|}{2}\ql(z)\mup(\diff z).
	\end{align}
Adding Equations \eqref{eq:II} and \eqref{eq:IV} gives 
	\begin{align}
	\text{(II)+(IV)}=V_\ell(\tel)\left(1-\int_\Te\alpha_{\ell}(\tel,z)\ql(z)\mup(\diff z)\right)\label{eq:II+IV}.
	\end{align}
	Lastly, adding \eqref{I+III} and \eqref{eq:II+IV} gives 
	\begin{align}
	(\Pisl V)(\tebl)=&\int_\Te\alpha_\ell(\tel,z)\ql(z)V(z)\mup(\diff z)+V_\ell(\tel)\left(1-\int_\Te\alpha_{\ell}(\tel,z)\ql(z)\mup(\diff z)\right)\\=&(P_{\ell,\ell}V_\ell)(\tel),
	\end{align}
	as desired.
\end{proof}

\end{lemma}
The previous result is quite intuitive, in the sense that we expect the Markov transition kernel of the multi-level MCMC algorithm \eqref{eq:joint_kernel} to ``respect'' the right marginals. 
The next Lemma shows that under Assumption \ref{Ass:positivity}, the families of measures $\{ \plj(\te,\cdot), \te \in \Te\}$, $j=\ell,\ell-1$ are uniformly tight. 

\begin{lemma}[Uniform tightness of the marginal kernel]
Suppose Assumption \ref{Ass:positivity} holds. Then, for any level $\ell=0,1,\dots,\lev$, the Markov kernel $p_{\ell,j}$, $j=\ell-1,\ell$  is uniformly tight, i.e., given some fixed, positive $\epsilon$, there exists a compact set $K_\epsilon\in\mathcal{B}(\Te)$ such that \begin{align}
p_{\ell,j}(\te,K_\epsilon)>1-\epsilon, \quad \forall \te\in \Te.
\end{align}
\begin{proof}
Exploiting the fact that $\mup$, $\my_j,$ $j=\ell-1,\ell$, and $Q_\ell$ are tight measures on $(\Te,\calb(\Te))$, there exists a compact set $K_\epsilon\in\mathcal{B}(\Te)$ such that  $\int_{K_\epsilon}Q_\ell(z)\mup(\diff z)>\epsilon$,  $\int_{K_\epsilon}\py_j(z)\mup(\diff z)>\epsilon$, and $\int_{K_\epsilon}\mup(\diff z)>\epsilon$. Take $r=\underset{z\in K_\epsilon}{\max} \  Q_\ell(z)/\py_j(z),$ and define the set $A_r=\{ z\in \Te \text{ s.t }  Q_\ell(z)/\py_j(z) \leq r \}$.  Clearly, $A_r\supset K_\epsilon$, and 
\begin{align}
\int_{A_r}Q_\ell(z)\mup(\diff z)>1-\epsilon, \ \int_{A_r}\py_j(z)\mup(\diff z)>1-\epsilon, \ \int_{A_r}\mup(\diff z)>1-\epsilon.
\end{align} Furthermore, from Assumption \ref{Ass:compact_level_sets}, $A_r$ is compact. Then, we have that 
\begin{align}
p_{\ell,j}(\te,A^c_r)&=\underbrace{\int_{A^c_r}\alpha_j(\te,z)Q_\ell(z)\mup(\diff z)}_\text{I}\\&+\underbrace{\car{\te\in A^c_r}\left(1-\int_\Te \alpha_j(\te,z)Q_\ell(z)\mup(\diff z)\right)}_\text{II}.
\end{align}
Clearly, we have that $$I\leq \int_{A^c_r}Q(z)\mup(\diff z)\leq \epsilon,\quad \forall \te\in\Te.$$ Moreover, we have that $II=0$ if $\te\in A_r$. For $\te\in A^c_r$, we have
 \begin{align}
\text{II}&=1-\int_{A_r}\alpha_j(\te,z)Q_\ell(z)\mup(\diff z)\\&-\int_{A^c_r}\alpha_j(\te,z)Q_\ell(z)\mup(\diff z)\\
&\leq 1-\int_{A_r}\min\left\{1, \frac{Q_\ell(\te)\py_j(z)}{Q_\ell(z)\py_j(\te)}\right\}Q_\ell(z)\mup(\diff z)=\text{III}
\end{align}
Notice that, by Assumption \ref{Ass:compact_level_sets}, we have that $\frac{Q_\ell(\te)}{\py_j(\te)}>\frac{Q_\ell(z)}{\py_j(z)}$ $\forall \te\in A^c_r$, $\forall z\in A_r. $ Then, \begin{align}
\text{III}=1-\int_{A_r} Q_\ell(z)\mup(\diff z)\leq \epsilon
\end{align} Thus, adding everything up we get  \begin{align}
p_{\ell,j}(\te,A^c_r)\leq 2\epsilon\quad \forall \te\in\Te,
\end{align}
which implies uniform tightness of the marginal kernel. 
\end{proof}
\end{lemma}

\begin{lemma}\emph{(Drift condition for the joint chain)}\label{lemma:drift_condition} Suppose that Assumption \emph{\ref{Ass:positivity}} holds. Then, for any level $\ell=1,2,\dots,\lev$, the multi-level Markov transition kernel $\pisl$ satisfies a drift condition of the form \eqref{eq:drift_cond}.  
\begin{proof}
We follow an argument similar to that of \cite[Theorem 5.2]{roberts1999bounds}, which has been previously  used in the analysis of regeneration times for Markov chains. We begin by constructing a particular Lyapunov function and a small set such that a drift condition holds for the marginal Markov transition kernels. From the uniform tightness of the marginal kernels, we have that for any $\ell=1,\dots,\lev$, $j=\ell-1,\ell,$ and some fixed $\epsilon\in(0,\frac{1}{4})$, there exist compact sets $K_{2\epsilon}, \ \tilde K_\epsilon\ \in \mathcal{B}(\Te)$  with
 \begin{align}
\plj(\te,\tilde K_\epsilon)>1-\epsilon \quad \forall \te\in \Te,\\
\plj(\te,K_{2\epsilon})>1-2\epsilon \quad \forall \te\in \Te.
\end{align}
Take then $K_\epsilon:=\tilde K_\epsilon\cup K_{2\epsilon}$ so that $\plj(\te,K_\epsilon)>\plj(\te,\tilde{K}_\epsilon)>1-\epsilon$ and $K_{2\epsilon}\subset K_\epsilon$. We can now construct a bounded Lyapunov function $V_{j}$ such that $V_{j}=1$ on $K_{2\epsilon}$ and $V_{j}=\Gamma$ on $K_\epsilon^c$, for some $\Gamma>1$ large enough such that \begin{align}\label{eq:beta_of_gamma}
\Lambda:=\frac{1+(\Gamma-1)2\epsilon}{\Gamma}+\frac{(1-\frac{1+(\Gamma-1)2\epsilon}{\Gamma}+(\Gamma-1)2\epsilon)}{\Gamma+1}<1.\end{align}

 Then, for $\te\in K_{\epsilon}$ and $\lambda$ arbitrary we have 
\begin{align}
(\Plj V_{j})(\te)&=\int_\Te \plj(\te,\diff z)V_{j}(z)=\int_\Te \plj(\te,\diff z)(V_{j}(z)\pm 1)\\
&=1+\int_{K_{2\epsilon}^c}\plj(\te,\diff z)(V_{j}(z)-1)\leq 1+(\Gamma-1)2\epsilon\\
&\leq \lambda V_j(\te)+(1-\lambda+(\Gamma-1)2\epsilon).
\end{align}
Conversely, for $\te\in K^c_\epsilon$, 
\begin{align}
(\Plj V_{j})(\te)\leq \frac{1+(\Gamma-1)2\epsilon}{\Gamma}\Gamma= \frac{1+(\Gamma-1)2\epsilon}{\Gamma}V_{j}(\te)=\lambda V_{j}(\te),
\end{align}
with $\lambda=\frac{1+(\Gamma-1)2\epsilon}{\Gamma}$, which satisfies $\lambda<1$ since $\epsilon<1/4$. Thus, we have that $\Plj$ satisfies the following drift condition for any $\te\in \Te$ 
\begin{align}
(\Plj V_{j})(\te)\leq\lambda V_{j}(\te)+\underbrace{(1-\lambda+(\Gamma-1)2\epsilon)}_\text{$:=\kappa$}\car{\te \in K_\epsilon}.
\end{align}
We can now define the bi-variate Lyapunov function $\VL:\Te\times \Te\to [1,\infty)$ as $\VL(\tebl):=\frac{1}{2}(V_\ell(\tel)+V_{\ell-1}(\tell))$ and the small set $\widehat{S}_\ell:=K_\epsilon\times K_\epsilon$. We begin by analyzing the Lyapunov condition for $\tebl\notin \widehat{S}_\ell$.  Notice that the set $\widehat{S}^c_\ell$ can be decomposed into three non-overlapping sets, namely $\mathrm{R}_1=K_\epsilon^c\times K_\epsilon,$ $\mathrm{R}_2=K_\epsilon\times K_\epsilon^c$ and $\mathrm{R}_3=\widehat{S}_\ell^c\backslash(\mathrm{R}_1\cup \mathrm{R}_2)$, as depicted in Figure \ref{fig:smallset}. Thus, we analyze the case $\tebl\notin\widehat{S}_\ell$ by considering each region $\mathrm{R}_i, i=1,2,3$ individually.  
		\begin{figure}
		\centering
		\includegraphics[width=0.5\linewidth]{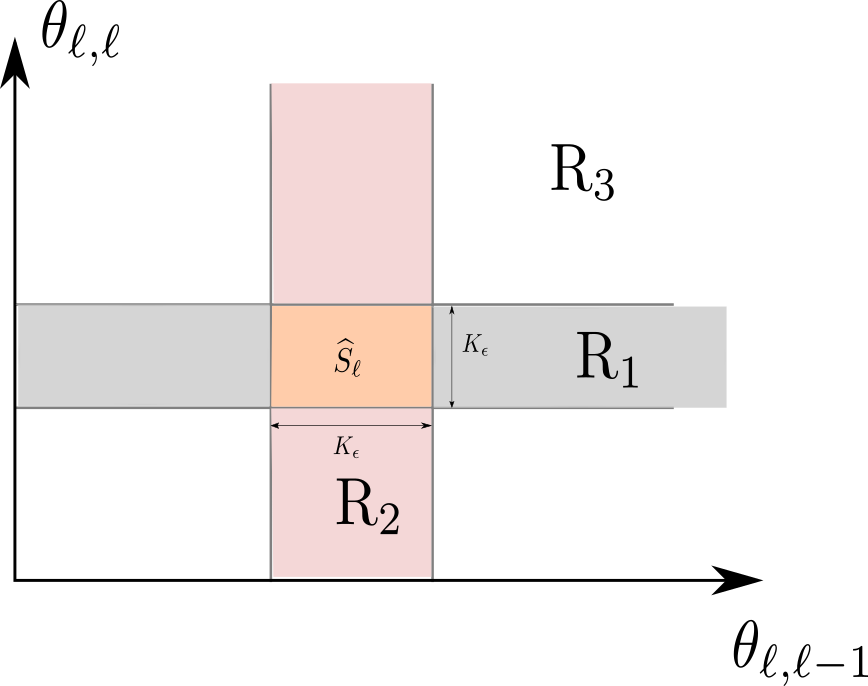}
		\caption{Schematic of the small set $\widehat{S}_\ell=K_\epsilon\times K_\epsilon$, and it's three related regions, $\mathrm{R}_1=K_\epsilon^c\times K_\epsilon,$ $\mathrm{R}_2=K_\epsilon\times K_\epsilon^c$ and $\mathrm{R}_3=\widehat{S}_\ell^c\backslash(\mathrm{R}_1\cup \mathrm{R}_2)$.}
		\label{fig:smallset}
	\end{figure}	
	
	 For $\tebl \in \mathrm{R}_3$, we simply have 
	\begin{align}
(\Pisl \VL) (\tebl)\leq	\lambda V_{\ell-1,\ell}(\tebl).
\end{align}	
Furthermore, $\forall\tebl\in \mathrm{R}_1$ we have:
	\begin{align}
(\Pisl \VL) (\tebl)\leq	\lambda V_{\ell-1,\ell}(\tebl)+\frac{\kappa}{2}&\leq \lambda V_{\ell-1,\ell}(\tebl)+\frac{\kappa}{1+\Gamma}V_{\ell-1,\ell}(\tebl),  \label{eq:first_ineq_drift}\\
&\leq \underbrace{\left(\lambda +\frac{\kappa}{1+\Gamma}\right)}_\text{$=\Lambda<1$ by \eqref{eq:beta_of_gamma}} V_{\ell-1,\ell}(\tebl)
	\end{align}
And similarly for $\tebl \in \mathrm{R}_2$. Thus, given that $\widehat{S}_\ell^c=\mathrm{R}_1\cup \mathrm{R}_2\cup\mathrm{R}_3$, with $\mathrm{R}_1, \mathrm{R}_2,\mathrm{R}_3$ non-overlapping sets, we have that $\forall \tebl\in \widehat{S}_\ell^c$ 
\begin{align}
(\Pisl \VL) (\tebl)\leq \Lambda \VL(\tebl) \label{eq:lambda_hat}.
\end{align}
	On the other hand, when $\tebl\in \widehat{S}_\ell$, we have that 
	\begin{align}
	\Pisl \VL (\tebl)&=\frac{1}{2} P_{\ell,\ell-1}V_{\ell-1}(\tell)+\frac{1}{2}P_{\ell,\ell}V_\ell(\tel)\\
	&\leq\frac{1}{2}\left(\lambda V_{\ell-1}(\tell)+\kappa\right)+\frac{1}{2}\left(\lambda V_{\ell}(\tel)+\kappa\right)\\
	&\leq\Lambda \VL(\tebl)+\kappa -\frac{\kappa}{1+\Gamma}\VL(\tebl)\\
	&\leq\Lambda \VL(\tebl) +\tilde \kappa,
	\end{align}
	where \begin{align}\tilde{\kappa}:=\frac{\Gamma}{1+\Gamma}\kappa.\label{eq:kappa_hat}\end{align} Thus, the joint Markov operator $\Pisl$ satisfies a drift condition of the form \ref{Ass:Strong_drift_exs} with a small set $\widehat{S}_\ell$, and constants given by \eqref{eq:lambda_hat} and \eqref{eq:kappa_hat}. 
\end{proof}

\end{lemma}

	We now have all the required results to prove Theorem \ref{thm:main_existence_ergo}. 
		\begin{proof}[Proof of Theorem \ref{thm:main_existence_ergo}]
 	The first claim of the theorem (the one pertaining to the existence and uniqueness of the invariant measure) follows immediately from Theorem \ref{thm:existence_uniquenes_bax} and Lemmata \ref{lemma:irreducibility}, \ref{lemma:small_set}, and \ref{lemma:drift_condition}. The second claim (the one pertaining to uniform ergodicity) follows from Theorems \ref{thm:unif_erg}, \ref{thm:bdd_lyap}, and the fact that the bi-variate Lyapunov function $V_{\ell-1,\ell}$ associated to $\pisl$, is bounded.	
 \end{proof}

\begin{remark}Notice that, contrary to many other MCMC methods where one typically constructs the Markov transition kernel in such a way that the algorithm generates samples from a target invariant probability measure (whose Radon-Nikodym derivative with respect to some reference probability measure is accessible up to a normalization constant), the level $\ell$ joint invariant measure $\nu_\ell$ is not known and it is induced by the joint Markov transition kernel. In particular, it  depends on the proposal $Q_\ell$. We aim at further investigating the influence of $Q_\ell$ on the invariant measure in future works. It is worth saying that, by construction, $\nu_\ell$ will in any case have the desired marginals $\muy_\ell,\muy_{\ell-1}$.    
\end{remark}

We have shown that the joint chain generated by the multi-level algorithm with independent proposals (i) has an invariant measure  and (ii) is uniformly ergodic. 

The choice of $Q_\ell$ is delicate for the ML-MCMC algorithm to work. For instance, consider the case $\lev=1$, $\mu^y_0=\mathcal{N}(1,1)$ and $\mu^y_1=\mathcal{N}(\frac{1}{2},1)$. What might look at first sight as a good proposal for the coupled chain at level $(\ell-1,\ell)=(0,1)$ is to take $Q_1=\mu^y_0$, i.e., the (exact) posterior at the previous level. However, this choice of proposal (which is unfeasible in practice, as direct sampling from $\mu^y_{\ell-1}$ is inaccessible) will not lead to a geometrically ergodic chain in view of Theorem \ref{thm:main_existence_ergo}, since $Q_1(z)/\muy_1(z)$ has essential infimum 0. The idea of proposing from the previous level is somehow what is advocated in \cite{dodwell2015hierarchical}, which could work only if $\exists c_1,c_2\in \R_+$ such that $c_1\leq \py_{\ell-1}(z)/\pyl(z)\leq c_2$,   $\forall \te\in \X$ and $\forall\ell$.

Lastly, notice that, by construction, the ML-MCMC Algorithm \ref{algo:MLMH} starts from  a measure $\widehat\lambda^0(A):=\lambda^0(A_\Delta)$, $\lambda^0\ll\mup$, where, for any set $A\in\mathcal{B}(\Te^2)$, we define $A_{\Delta}:=\{z\in \Te:\ (z,z)\in A\}$. We now show that, for any level $\ell=1,2,\dots,\lev$,  $\widehat \lambda^0\ll\nul$.

\begin{theorem}[Absolute continuity of initial measure]\label{Thm:abs_cont} Under the same assumptions as in Theorem \ref{thm:main_existence_ergo}, we have that, for any level $\ell=1,2,\dots,\lev$,  $\widehat \lambda ^0\ll\nul$. 
\begin{proof}
Let $A\in \mathcal{B}(\Te^2)$ be a compact set such that $\nul(A)=0$ (the non-compact set will be shown later). Furthermore, from the tightness of $\nul$, we have that, given some $\epsilon>0$, there exists a compact $K_\epsilon\in\mathcal{B}(\Te^2)$ such that $\nul(K_\epsilon)\geq 1-\epsilon$. We then have that 
\begin{align}
0=\nul(A)&=\int_{\Te^2}\pisl(\tebl,A)\nul(\diff\tebl)\\&\geq \int_{\Te^2}\int_{A_\Delta}\min_j\left\{\min\left\{Q_\ell(z),\frac{\py_j(z)Q_\ell(\te_{\ell,j})}{\py_j(\te_{\ell,j})} \right\}   \right\}\mup(\diff z)\nul(\diff \tebl)\\
&\geq \int_{K_\epsilon}\int_{A_\Delta}\min_j\left\{\min\left\{Q_\ell(z),\frac{\py_j(z)Q_\ell(\te_{\ell,j})}{\py_j(\te_{\ell,j})} \right\}   \right\}\mup(\diff z)\nul(\diff \tebl).\label{eq:bndnula}
\end{align}
By Assumption \ref{Ass:positivity} and compactness of $K_\epsilon$ and $A$, we have that there exists a $c'>0$ such that $c'\leq\min_j\left\{\min\left\{Q_\ell(z),\frac{\py_j(z)Q_\ell(\te_{\ell,j})}{\py_j(\te_{\ell,j})} \right\}   \right\}$, $\forall \tebl \in K_\epsilon$, $\forall z\in A_\Delta$. Then, we get that
\begin{align}
\eqref{eq:bndnula}\geq c'(1-\epsilon) \mup(A_\Delta),
\end{align}
which implies that $\mup(A_\Delta)=0$.  Moreover, since ${\lambda}^0\ll\mup$, we have $ \widehat\lambda^0(A)=\lambda^0(A_\Delta)=0$, and as such, $\widehat{\lambda}^0\ll \nul$. 
Now, suppose $A$ is not compact. 
Since $\hat{\lambda}^0$  is a tight probability measure (c.f. Theorem \ref{thm:billingsley}), then it follows that,  $$\hat{\lambda}^0(A)=\underset{\underset{K\text{ compact}}{K\subset A}}{\sup}\hat{\lambda}^0(K)=0,$$ and we can conclude as in the previous case.
	\end{proof}
\end{theorem}


\subsection{Non-asymptotic bounds on the level-wise ergodic estimator}

Recall that if the Markov operator $P$ defines a map $P:L_q(\calx,\mu)\to L_q(\calx,\mu)$, for some $q\in[1,\infty],$ we can define the norm of $P$ as 
\begin{align}
\lno P\rno_{L_q(\calx,\mu)\to L_q(\calx,\mu)}:=\underset{\underset{\lno f \rno_{L_q(\calx,\mu)}=1}{f\in L_q(\calx,\mu)}}{\sup}\lno P f \rno_{L_q(\calx,\mu)}.
\end{align}
It is well-known that $P$ is always a weak contraction in $L_q(\calx,\mu)$, i.e., $\lno P\rno_{L_q(\calx,\mu)\to L_q(\calx,\mu)}\leq 1$. 
We define the $L_q(\calx,\mu)$-\emph{spectral gap} of $ P: L_q(\calx,\mu)\to  L_q(\calx,\mu)$ as 
\begin{align}\label{eq:spec_gap}
\gamma_q[P]:=1-\lno P\rno_{L_q^0\to L_q^0}.
\end{align}
Whenever $\gamma_q[P]>0$, it can be shown that $\nu^0P^n$ converges to $\mu$ for any $\nu^0\in\mathcal{M}(\X)$ in some appropriate distance for probability measures (see, e.g., \cite{latz2020generalized,rudolf2011explicit}).

We now turn to analyze the level-wise contribution to the ML-MCMC ergodic estimator \eqref{eq:ML-MLMC_estimator}, which we write hereafter in more general terms, including a burn-in phase. 

For $\ell=1,2,\dots,\lev,$ let  $\Q_\ell-\Q_{\ell-1}=:Y_\ell:\Te^2\to \R$ be a $\nul$ square-integrable function, and $\nu^0$ be a probability measure on $(\Te^2,\mathcal{B}(\Te^2))$ such that $\nu^0\ll \nul$. In addition, denote by $	\mathsf{E}_{\nu^0, \bm{P}_\ell}[Y_\ell]$ the expectation of $Y_\ell$ over the Markov chain generated by $\Pisl$, starting from an initial probability measure $\nu^0$, and consider the following ergodic estimator:
\begin{align}\label{eq:erg}
	\hat{Y}_{\ell,N_\ell,n_{b,\ell}}:=\frac{1}{N_\ell}\sum_{n=1}^{N_\ell}Y_\ell(\tebl^{n+n_{b,\ell}}),\quad \tebln\sim \pisl(\tebl^{n-1},\cdot),
\end{align}
where $n_{b,\ell}\in \mathbb{N}$ is the usual \emph{burn-in period}.  The aim of this section is to provide error bounds on the \emph{non-asymptotic statistical Mean Square Error} of \eqref{eq:erg}, namely 
\begin{align}
	\mathrm{MSE}(\hat{Y}_{\ell,N_\ell,n_{b,\ell}};\nu^0):=\mathsf{E}_{\nu^0, \bm{P}_\ell}\left[\left(\hat{Y}_{\ell,N_\ell,n_{b,\ell}}-\nul(Y_\ell)\right)^2\right].
\end{align} 
In particular, we aim at obtaining a bound of the form
\begin{align}\label{eq:mse_bound_can}
	\mathrm{MSE}(\hat{Y}_{\ell,N_\ell,n_{b,\ell}};\nu^0)\leq C_{\mathrm{mse},\ell}\frac{\V_{\nul}[Y_\ell]}{N_\ell},
\end{align}
 for some level-dependent, positive constant $C_{\mathrm{mse},\ell}.$  Such a bound is presented in Theorem \ref{thm:namse}, the main result of this subsection. A bound of the form \eqref{eq:mse_bound_can} will be required for the cost analysis of Section \ref{sec:cost_analysis}. Bounds such as \eqref{eq:mse_bound_can} are known to exist for geometrically ergodic and reversible Markov transition kernels \cite{rudolf2011explicit}. However, it is easy to see that the chain generated by $\Pisl$ is not $\nul$-reversible.
As a consequence, we can not directly apply the non-asymptotic bounds presented in \cite{rudolf2011explicit}.  Instead, inspired by the error analysis of \cite{rudolf2011explicit} and the so-called \emph{pseudo-spectral} approach of \cite{paulin2015concentration}, we construct a bound of the form \eqref{eq:mse_bound_can} for general (i.e., not necessarily multi-level) non-reversible, discrete-time Markov chains. To the best of the authors' knowledge, this result is new.

%
%

For any $q,q'\in[1,\infty]$ with  $\frac{1}{q}+\frac{1}{q'}=1,$ we denote the \emph{adjoint operator} of $\Pisl:L_q(\Te^2,\nul)\to L_q(\Te^2,\nul)$ by $\Psml:L_{q'}(\Te^2,\nul)\to L_{q'}(\Te^2,\nul).$ It can be shown \cite{rudolf2011explicit} that a Markov operator acting on $L_2(\Te^2,\nul)$ is self-adjoint if and only if it is $\nul$-reversible. Notice that, even though $\Pisl$ is not reversible, the multiplicative reversibilization $\bm R_\ell=(\Psml\Pisl)$ is.

Moreover, we define the so-called \emph{pseudo-spectral gap} of the Markov operator $\Pisl: L_2(\Te^2,\nul)\to  L_2(\Te^2,\nul)$ as
\begin{align}\label{eq:psg}
	\gamma_\mathrm{ps}[\Pisl]:=\underset{k\geq 1}{\max}\left \{ \gamma_2[ (\Pisl^*)^k\Pisl^k ]/k \right\},\ k \in \mathbb{N}.
\end{align}

As mentioned in \cite{paulin2015concentration}, the pseudo-spectral gap can be understood as a generalization of the $L_2$-spectral gap of the  so-called \emph{multiplicative reversibilization} $\bm R_\ell=(\Psml\Pisl).$ It is shown in \cite[Proposition 3.4]{paulin2015concentration} that for a uniformly ergodic chain with Markov kernel,  $\bm P_\ell$, $\exists \tau_\ell<\infty$ such that $\gamma_\mathrm{ps}[\bm P_\ell]\geq \frac{1}{2\tau_\ell}>0$, where $\tau_\ell$ is the so-called \textit{mixing time}.

\begin{theorem}[Non-asymptotic bound on the mean square error]\label{thm:namse} 
	Suppose Assumption  \emph{\ref{Ass:positivity}} holds. Furthermore, for any $\ell=1,2,\dots,\lev,$ let $Y_\ell\in  L_2(\Te^2,\nul)$, and write $g_\ell(\tebl)=Y_\ell(\tebl)-\int_{\Te^2}Y_\ell(\tebl)\nul(\diff \tebl)$, and assume the Markov chain generated by $\Pisl$ is started from a measure $\nu^0$ with $\nu^0\ll\nul$, and  $\frac{\diff \nu^0}{\diff \nul}\in L_\infty(\Te^2,\nul).$    Then,
	\begin{align}
		\mathrm{MSE}(\hat{Y}_{\ell,N_\ell,n_{b,\ell}};\nu^0):=\mathsf{E}_{\nu^0, \bm{P}_\ell}\left|\frac{1}{N_\ell}\sum_{n=1}^{N_\ell}g_\ell(\tebl^{n+n_{b,\ell}})\right|^2\leq C_{\mathrm{mse},\ell}\frac{\V_{\nul}[Y_\ell]}{N_\ell},\label{eq:bound_namse}
	\end{align}	
 	where $C_{\mathrm{mse},\ell}=C_{\mathrm{inv},\ell }+C_{\mathrm{ns},\ell },$ with $$C_{\mathrm{inv},\ell }=\left(1+\frac{4}{\gamma_\mathrm{ps}[\bm{P}_\ell]}\right),\quad C_{\mathrm{ns},\ell }=\left(2\lno\frac{\diff \nu^0}{\diff \nul}-1 \rno_{L_\infty}\left(1+\frac{4}{\gamma_{\mathrm{ps}}[\Pisl]}\right)\right),$$ and where $\gamma_{\mathrm{ps}}[ \Pisl]$  is the pseudo-spectral gap of $\Pisl,$ defined in \eqref{eq:psg}.  
\end{theorem}
The proof of this theorem is technical and presented in the appendix for a general uniformly ergodic (non-reversible) Markov chain. 
Notice that the Assumption $\nu^0\ll\nul$ holds in our setting by Theorem \ref{Thm:abs_cont} for $\nu^0(A)=\lambda^0(A_\Delta)$.

Notice also that, even though the constants $C_{\mathrm{inv},\ell}$, $C_{\mathrm{ns},\ell}$  depend on the level $\ell$, we do not expect them to degenerate as $\ell\to\infty$. In particular, the dependency on the level is given by  two terms, namely  $\gamma_{\mathrm{ps}}[\bm{P}_\ell]$ and  $\lno \frac{\diff \nu^0}{\diff \nul}-1\rno_{L_\infty}.$ For the first term, we expect $\gamma_{\mathrm{ps}}[\Pisl]$ to become smaller and smaller as $\ell\to\infty$  and for it to converge to a limit value $\gamma_{\mathrm{ps}}[\bm{P}_\infty]>0$ (see also the discussion of sychronization of the coupled chains in Section \ref{sec:cost_analysis}).  For the second term $\lno \frac{\diff \nu^0}{\diff \nul}-1\rno_{L_\infty},$ notice that $\nul$ converges to a measure that has all of its mass in the diagonal set of $\mathsf{X}^2$. Since $\nu^0$ is a finite measure on such a diagonal, we also expect that this term remains bounded  as $\ell\to\infty$. We are, however, not able to prove these claims at the moment, reason why we formulate the following assumption.  

\begin{assumption}\label{ass:cmse_ind}
There exist a level independent constant $C_\mathrm{mse}$ such that, for any $\ell=0,1,\dots,$ it holds that $C_{\mathrm{mse},\ell}<C_\mathrm{mse}$.
\end{assumption}
The fact that $C_{\mathrm{mse},\ell}$  does not blow-up as $\ell\to\infty$ is an important requirement on the asymptotic analysis of ML-(MC)MC methods.

We remark that the bound \eqref{eq:bound_namse} should be compared to the one presented in \cite[Theorem 3.34]{rudolf2011explicit}. In particular, that work presents a sharper bound than \eqref{eq:bound_namse},  however, such a bound necessitates more restrictive Assumptions which we list in the next theorem for completeness, whose proof is an easy adaptation of \cite[Theorem 3.34]{rudolf2011explicit} to our setting and its omitted.

\begin{theorem}\label{thm:rudolf}
Suppose that the Assumptions of Theorem \ref{thm:namse} hold. In addition, assume that for any $\ell=1,2,\dots,\lev$:
	\begin{enumerate}[label=\emph{C\arabic*}.,ref=C\arabic*]
		\item\label{R1} ($L_2$-spectral gap) there exists $b_\ell\in(0,1)$ such that $$\lno \Pisl \rno_{L^0_2(\Te^2,\nul)\to L^0_2(\Te^2,\nul)}<b_\ell,$$
		\item\label{R2} ($L_1$-exponential convergence) there exists $\tilde c_\ell\in \R_+,a_\ell\in(0,1)$ such that 
		\begin{align}
		\lno \nu^0 \Pisl^n-\nu_\ell\rno_{L_1(\Te^2,\nul)}:=\lno \frac{\diff (\nu^0 \Pisl^n)}{\diff \nul}-1\rno_{L_1(\Te^2,\nul)}\leq \tilde c_\ell a_\ell^n,
		\end{align} 
	\end{enumerate}	

Then, the non-asymptotic MSE is given by 
\begin{align}\label{eq:rudolf}
\mathsf{E}_{\nu^0, \bm{P}_\ell}\left|\frac{1}{N_\ell}\sum_{n=1}^{N_\ell}g_\ell(\tebl^{n+n_{b,\ell}})\right|^2\leq\frac{\V_{\nul}[Y_\ell]}{N_\ell}\left(\frac{2}{(1-b_\ell)}+ \frac{2 \tilde{c}_\ell \lno \frac{\diff\nu^0}{\diff\nul}-1\rno_{L_\infty} a_\ell^{n_b,\ell}}{N_\ell(1-a_\ell)^2}\right),\label{eq:bound_namse_r}
\end{align}	
where the first term in the parenthesis is associated to the variance contribution to the MSE, while the second term corresponds to the statistical squared bias and is of higher order in $N_\ell$. 
\end{theorem}
In general, the stronger  Assumptions \ref{R1} and \ref{R2} are known to hold for Markov chains which are both reversible and geometrically ergodic. However, due to its construction, the Markov transition kernel $\Pisl$ of the ML-MCMC algorithm is not reversible. Nevertheless,  we believe that the algorithm presented herein satisfies Assumptions \ref{R1} and \ref{R2} and as such, a bound on the MSE of the form \eqref{eq:rudolf}, should hold. However, we are currently unable to verify this claim either, and we will restrict to the bound of theorem \ref{thm:namse} and the less restrictive Assumption \ref{ass:cmse_ind}. Further examination of the convergence properties of ML-MCMC algorithms will be the topic of future work.

\section{Cost analysis of the ML-MCMC algorithm}\label{sec:cost_analysis}
 For $\ell=0,1,\dots,\lev,$ let $\qoi_\ell:\Te\mapsto \R$ be a $\myl$-integrable quantity of interest. The \emph{Total Mean Square Error} of the multi-level estimator \eqref{eq:ML-MLMC_estimator} is given by
 \begin{align}
\hat{\mathrm{e}}_\mathrm{ML}(\mlQ)&:=\E\left[\left(\mlQ -\E_{\my}[\qoi]\right)^2\right]\label{eq:MSE},
\end{align}
where we have denoted expectation  over the whole multi-level Markov chain by $\E$, i.e., without a subscript. We remark that the estimator $\mlQ$ will also depend on $\{ \Pisl\}_{\ell=1}^\lev,$ as well as the burn-in and initial measure for each level; however, for the sake of readability, we opted  not to explicitly write these dependencies throughout this section.  Notice that the previous term can be upper bounded by 

 \begin{align}
&\hat{\mathrm{e}}_\mathrm{ML}(\mlQ)= \V[\mlQ]+\left[  \E\left[\mlQ\right] -\E_{\my}[\qoi]\right]^2\\
&\leq \underbrace{\V[\mlQ]}_\text{Variance contr.}+2\underbrace{\left[\E\left[\mlQ\right] -\E_{\my_\lev}[\qoi]\right]^2}_\text{MCMC bias contr.}+2\underbrace{\left[\E_{\my_\lev}[\qoi]-\E_{\my}[\qoi]\right]^2}_\text{Discretization contr.}.\label{eq:1st_bound}
\end{align}
where we have denoted by $\V$ the variance  over the whole multi-level Markov chain. Furthermore, from Cauchy-Schwarz, we have that \begin{align}
\V[\mlQ]&\leq 2(\lev+1)\sum_{\ell=0}^{\lev}\V_{\nul}[\hat Y_\ell],
\end{align} and \begin{align}
2\left[\E\left(\mlQ\right) -\E_{\myl}[\qoi]\right]^2 &\leq 2(\lev +1)\sum_{\ell=0}^\lev \left(\E[\hat{Y}_\ell]-\E_{\my_\ell}[{Y}_\ell]\right)^2.
\end{align}
Thus, from \eqref{eq:bound_namse} and \eqref{eq:1st_bound} we obtain 
\begin{align}
&\hat{\mathrm{e}}_\mathrm{ML}(\mlQ)	\leq\\ &\underbrace{2(\lev+1)\sum_{\ell=0}^{\lev}\mathrm{MSE}(\hat{Y}_\ell)}_\text{Total statistical error}+\underbrace{2\left[\E_{\my_\lev}[\qoi]-\E_{\my}[\qoi]\right]^2}_\text{Discretization error}=:\mathrm{e}_\mathrm{ML}(\mlQ)\label{eq:MSEb}.
\end{align}
For some tolerance $\tol>0,$ we denote the minimal computational cost required to obtain $\mathrm{e}_\mathrm{ML}(\mlQ)\leq \tol^2$ by $\mathcal{C}\left(\mathrm{e}_{\mathrm{ML}}\left(\widehat{\Q}_{\lev,\{N_{\ell}\}_{\ell=0}^\lev}\right),\tol^2\right)$. The focus of this section is to provide upper bounds on this computational cost, while at the same time quantifying the computational advantage of the multi-level MCMC method  over its single level counter part (at level $\lev$). In particular, our result can be though of as an extension of \cite[Theorem 3.4]{dodwell2015hierarchical}.   The main result of this section is presented in Theorem \ref{thm:cost_thm_schichl_1}. To establish a cost-tolerance relation, we first need to make assumptions on the decay of the discretization error, and corresponding increase in computational cost for the evaluation of $\eff_\ell$ introduced in Section \ref{ss:bip} as a function of the discretization parameter $M_\ell=s^\ell M_0$.

\begin{assumption}\label{Assumptions_phi} For any $\ell\geq 0$, the following hold:

	\begin{enumerate}[label={\emph{{\ref{Assumptions_phi}.\arabic*}}}.,ref={\ref{Assumptions_phi}.\arabic*}]

	\item \label{Assumptions_phi1} There exist positive functions $C_\eff, C_\Phi:\Te\to\R_+$ independent of $\ell$, and positive  constants $C_\mathrm{e},\alpha$ independent of $\te$ and $\ell$ such that
\begin{enumerate}
		\item 	\label{eq:conv_u}
		$\lno \eff_\ell(\te) -\eff(\te)\rno_\mathsf{Y}\leq C_{\eff}(\te)s^{-\alpha \ell}, \ \forall \te\in \Te.
	$
		\item \label{Assumptions_phi3} 
		$\label{eq:pot_esta}
		\lv \poth-\pot\rv\leq C_\Phi(\te) \lno \eff_\ell(\te) -\eff(\te)\rno_\mathsf{Y},\ \forall \te\in \Te,
$
    \item $\int_\X \exp(C_\eff(\te)C_\Phi(\te))\mup(\diff \te)\leq C_\mathrm{e}<\infty.$
	\end{enumerate}
\item \label{Assumptions_phi2} Given a $\myl$-integrable quantity of interest $\Q_\ell$, there exits a  function $C_q:\Te\to\R_+$ independent of $\ell$ and positive constants $\tilde C_q,\alpha_q,C_\mathrm{m},$ and $m>2$, independent of $\te$ and $\ell$ such that
\begin{enumerate}
	\item $\lvert \Q_\ell(\te)-\Q(\te) \rvert\leq C_{q}(\te)s^{-\alpha_q \ell},\ \forall \te\in \Te. $
	\item $\int_\X C_q^2(\te)\mup(\diff \te)\leq\tilde C_q^2<\infty.$
	\item \label{Assumptions_phi5} $\left(\int_\X |\qoi_\ell(\te)|^m\mup(\diff\te)\right)^{1/m}\leq C_\mathrm{m}<\infty.$ \label{ass:bdd_2nd}
\end{enumerate}
			
		\item \label{ass:cost}There exist positive constants $\gamma$ and $C_{\mathcal{\gamma}},$ such that, for each discretization level $\ell$, the computational cost of obtaining one sample from a $\myl$-integrable quantity of interest $\qoi_\ell(\tel),\ \tel\sim\myl$, with $\te_{\ell,\ell}$ generated by Algorithm \ref{algo:2-IS_coupling}, denoted by  $\mathcal{C}_\ell(\Q_\ell),$  scales as $$\mathcal{C}_\ell(\Q_\ell)\leq C_{\gamma}s^{\gamma \ell }.$$


	\end{enumerate}
	
\end{assumption}

\remark{Notice that, with a slight abuse of notation, we have used the symbol $\alpha$ to denote the (strong) rate in \ref{Assumptions_phi1}, and $\alpha_\ell(\cdot,\cdot)$ to denote acceptance probability at level $\ell$.  We hope this does not create any confusion. }

For all $\ell=1,2,\dots,\lev$, and for a $\my_j$-integrable quantity of interest $\qoi_j$, $j=\ell-1,\ell$, we write
\begin{align}
Y_\ell(\tebl)&:=\qoi_\ell(\tel)-\qoi_{\ell-1}(\tell), \ \tebl=(\tell,\tel)\in\Te^2,\\
\hat{Y}_{\ell,N_\ell}&:=\frac{1}{N_\ell}\sum_{n=1}^{N_\ell}Y_\ell(\tebl^{n+n_{b,\ell}}),\quad \tebln\sim \pisl(\tebl^{n-1},\cdot),\ n_{b,\ell}\in\mathbb{N}.
\end{align}  We state the main result of this section. 

\begin{theorem}[Decay of errors]\label{thm:cost_thm_schichl_1} For any $\ell=0,1,\dots,\lev$, let $\qoi_\ell$ be an $L_1(\Te,\myl)$-integrable quantity of interest and  suppose Assumptions \emph{\ref{Ass:positivity}}, \emph{\ref{ass:cmse_ind}}, and  \emph{\ref{Assumptions_phi}} hold.  Then, there exist positive constants $C_w,C_v,C_\mathrm{mse}$, independent of $\ell$ such that:
	\begin{enumerate}[label=\emph{{T\arabic*}}.,ref=T\arabic*]
	\item\label{Assumption:for_cost_thm1}\emph{\textrm{(Weak convergence)}} $\lv \E_{\myl}[\Q_{\ell}]-\E_{\my}[\Q]\rv \leq C_w s^{-\alpha_w\ell  },$
	\item\label{Assumption:for_cost_thm2}\emph{\textrm{(Strong convergence)}} $\V_{\nul}[Y_\ell]\leq C_v s^{-\beta \ell }.$
	\item\label{Assumption:for_cost_thm3}\emph{\textrm{(MSE bound)}} $\mathrm{MSE}(\hat{Y}_{\ell,N_\ell})\leq N_\ell^{-1}C_\mathrm{mse}\V_{\nul}[Y_\ell].$
\end{enumerate}	
Here, $\alpha_w=\min\{\alpha_q,\alpha\}$ and $\beta=\min\{2a_q,\alpha(1-2/m)\},$  with $\alpha,\alpha_q,$ and $m$  as in \emph{Assumption \ref{Assumptions_phi}}. 
\end{theorem}

The proof of Theorem \ref{thm:cost_thm_schichl_1} is presented in Section \ref{ss:proof_cost}. It has been shown in \cite[Theorem 3.4]{dodwell2015hierarchical} that, if a ML-MCMC algorithm satisfies conditions \ref{Assumption:for_cost_thm1}-\ref{Assumption:for_cost_thm3}, then it has a complexity (cost-tolerance relation) analogous to a standard MLMC algorithm to compute expectations (when independent sampling from the underlying probability measure is possible) up to logarithmic terms. This result is stated in Theorem \ref{thm:cost_thm_schichl_2} below. The purpose of Theorem \ref{thm:cost_thm_schichl_1} is to show that our class of ML-MCMC algorithms does actually fulfill conditions \ref{Assumption:for_cost_thm1}-\ref{Assumption:for_cost_thm3}. 

\remark{Throughout this work, we have the tacit assumption that the chain at level 0, i.e., the one that does not require an IMH sampler, is geometrically ergodic with respect to $\my_0$.  }

\begin{theorem}\emph{(\cite[Theorem 3.4]{dodwell2015hierarchical})}\label{thm:cost_thm_schichl_2} Under the same assumptions as in Theorem \ref{thm:cost_thm_schichl_1}, with $\alpha_w\geq\frac{1}{2}\min\{\gamma,\beta\}$, for any $\tol>0$ there exist  a number of levels $\lev=\lev(\tol),$ a decreasing sequence of integers $\{N_\ell(\tol)\}_{\ell=0}^\lev$, and a positive constant $C_\mathrm{ML}$ independent of $\tol$, such that the  \emph{MSE} bound of the multilevel estimator, $\mathrm{e}_\mathrm{ML}(\widehat{\Q}_{\lev,\{N_{\ell}\}_{\ell=0}^\lev})$, satisfies
	
	$$\mathrm{e}_\mathrm{ML}\left(\widehat{\Q}_{\lev,\{N_{\ell}\}_{\ell=0}^\lev}\right)\leq\tol^2,$$
	whereas, the corresponding total ML-MCMC cost is bounded by  \begin{align}\label{eq:cost_ml_bound}
	\mathcal{C}\left(\mathrm{e}_{\mathrm{ML}}\left(\widehat{\Q}_{\lev,\{N_{\ell}\}_{\ell=0}^\lev}\right),\tol^2\right)\leq C_\mathrm{ML}\begin{cases}
	\tol^{-2}|\log\tol| & if \ \beta>\gamma,\\
	\tol^{-2}|\log\tol|^3, & if \ \beta=\gamma,\\
	\tol^{-2+(\gamma-\beta)/\alpha_w}|\log\tol|, & if \ \beta<\gamma.\\
	\end{cases}
	\end{align}

\begin{proof}
	Just as in \cite{dodwell2015hierarchical}, the proof of this theorem follows from \eqref{eq:MSEb} and the proof of \cite[Theorem 1]{cliffe2011multilevel}
	\end{proof}

\end{theorem}

We remark that the rates in Theorem \ref{thm:cost_thm_schichl_2} are independent of the dimension of $\X$. It is shown also in \cite{dodwell2015hierarchical} that the cost of obtaining an equivalent single-level (at level $\lev$) mean square error of an  estimator $\widehat{\Q}_{N}$ based on a single-level Markov chain Monte Carlo algorithm (e.g., standard Metropolis-Hastings) (denoted by $\mathrm{e}_\mathrm{SL}),$ generated by a reversible and geometrically ergodic Markov kernel is given by
	\begin{align}
	\mathcal{C}\left(\mathrm{e}_{\mathrm{SL}}\left(\widehat{\Q}_N\right),\tol^2\right)\leq C_\mathrm{SL}\tol^{-2-\gamma/\alpha_w}, \quad C_\mathrm{SL}\in\R_+,
	\end{align}
	where $\alpha_w$ and $\gamma$ are the same constants as in Theorem \ref{thm:cost_thm_schichl_1}, and $C_\mathrm{SL}$ is a positive constant independent of $\tol$. 
\subsection{Proof of Theorem \ref{thm:cost_thm_schichl_1}}\label{ss:proof_cost}
We will decompose the proof of Theorem \ref{thm:cost_thm_schichl_1} in a series of auxiliary results. Notice that \ref{Assumption:for_cost_thm3} is obtained from Theorem \ref{thm:namse} with a level dependent constant and we postulated in Assumption \ref{ass:cmse_ind} that  this constant can be bounded by a finite, level-independent constant $C_\mathrm{mse}$, and as such, we can use it in \ref{Assumption:for_cost_thm3}. Thus, we just need to prove that  \ref{Assumption:for_cost_thm1} and \ref{Assumption:for_cost_thm2} hold. This is done in Lemmata \ref{lemma:verification:1} and \ref{lemma:verification:2}. We begin by proving some auxiliary results needed to prove implication \ref{Assumption:for_cost_thm1}.
%
%

\begin{lemma}\label{lemma:bound_constant}
Suppose \emph{Assumption \ref{Assumptions_phi}} holds. Then, for  $\ell=1,2,\dots,\lev$ it holds\begin{align}
c_I\leq Z_\ell\leq C_e,\end{align}
where $c_I=\int_\X\exp(-\Phi(\te;y)-C_\eff(\te)C_\Phi(\te))\mup(\diff\te)$ and $C_e$ as in \emph{Assumption \ref{Assumptions_phi}}.
\begin{proof}
From Assumptions \ref{Assumptions_phi1} one has that, for all $\ell\geq 0$, and $\te \in\Te$, \begin{align}
\Phi(\te;y)-C_\Phi(\te) C_{\eff}(\te)\leq \Phi_{\ell}(\te;y)\leq \Phi(\te;y)+C_\Phi(\te) C_\eff(\te).
\end{align}
Hence,  
\begin{align}
Z_\ell&=\int_\Te \exp(-\Phi_\ell(\te;y))\mup(\diff\te)\leq \int_\Te \exp\left(-(\Phi(\te;y)-C_\Phi(\te) C_\eff(\te))\right)\mup(\diff\te)\\&\leq \int_\X \exp(C_\Phi(\te) C_\eff(\te))\mup(\diff\te)=C_e,
\end{align}
where the last step follows from the assumption of non-negativity of $\Phi(\theta;y)$ (c.f. \eqref{eq:potential}).
Similarly, $Z_\ell\geq \int_\X\exp(-\Phi(\te;y)-C_\Phi(\te) C_\eff(\te))\mup(\diff \te)=c_I,$ independently of $\ell$. 
\end{proof}
\end{lemma}			

%
%

\begin{lemma}\label{prop:distance_densities}
Suppose \emph{Assumption \ref{Assumptions_phi}} holds. Then, for  any $\ell\geq 1,$ there exist positive functions $C_{\pi,\ell}(\te):\X\to\R_+$, $\tilde C_{\pi,\ell}(\te):\X\to\R_+$, such that
\begin{align}|\pi^y_\ell(\te)	-\pi^y_{\ell-1}(\te)| &\leq C_{\pi,\ell}(\te) s^{-\alpha\ell},\quad \forall \te\in \Te,\label{eq:bound_diff_pi}\\
|\pi^y_\ell(\te)	-\pi^y(\te)| &\leq \tilde C_{\pi,\ell}(\te) s^{-\alpha\ell},\quad \forall \te\in \Te.\label{eq:bound_ass_pi}
\end{align}
Moreover, $C_{\pi,\ell}(\te)=(\pyl(\te)+\pyll(\te))K_{\pi,\ell}(\te)$, $\tilde C_{\pi,\ell}(\te)=(\pyl(\te)+\py(\te))\tilde K_{\pi,\ell}(\te)$, with \begin{align}
	\tilde K_{\pi,\ell}(\te)&=C_\Phi(\te)C_\eff(\te)+c_I^{-1}C_\mathrm{e},\\
K_{\pi,\ell}(\te)&=(1+s^\alpha)\tilde K_{\pi,\ell}(\te).\end{align} Furthermore, we have that for any $p\in[1,+\infty]$,
 \begin{align}
	K_p&:=\left(\int_\X |K_{\pi,\ell}(\te)|^p\mup( \diff \te )\right)^{1/p}<+\infty,\\ 
	\tilde K_p&:=\left(\int_\X |\tilde K_{\pi,\ell}(\te)|^p\mup( \diff \te )\right)^{1/p}<+\infty.
	\end{align}

\end{lemma}

\begin{proof} We begin with the proof of \eqref{eq:bound_diff_pi}. We consider first the case $\potm\leq\Phi_{\ell-1}(\te;y)$. 
	\begin{align}
	&|\pi^y_\ell(\te)	-\pi^y_{\ell-1}(\te)|=\left\lvert \frac{e^{-\Phi_\ell(\te;\yy)}}{Z_\ell}-\frac{e^{-\Phi_{\ell-1}(\te;\yy)}}{Z_{\ell-1}}\right\rvert \nonumber \\
	&\leq
	\underbrace{\left\lvert \frac{e^{-\Phi_\ell(\te;\yy)}}{Z_\ell}-\frac{e^{-\Phi_{\ell-1}(\te;\yy)}}{Z_{\ell}}\right\rvert}_\text{$I$} +
	\underbrace{\left\lvert \frac{e^{-\Phi_{\ell-1}(\te;\yy)}}{Z_\ell}-\frac{e^{-\Phi_{\ell-1}(\te;\yy)}}{Z_{\ell-1}}\right\rvert}_\text{$II$}. \end{align}%

We first focus on $I$.  A straightforward application of the mean value theorem gives 
\begin{align}
\left\lvert {e^{-\Phi_\ell(\te;\yy)}}-{e^{-\Phi_{\ell-1}(\te;\yy)}}\right\rvert\leq e^{-\potm}|\potm-\potl|. \label{eq:mvt}
\end{align} Thus, we have from \eqref{eq:mvt}, together with Assumptions \ref{Assumptions_phi1} that 
\begin{align}
	I&=   {Z_\ell}^{-1}  \left\lvert {e^{-\Phi_\ell(\te;\yy)}}-{e^{-\Phi_{\ell-1}(\te;\yy)}}\right\rvert\leq \pyl(\te)|\potm-\potl|\\
	&\leq \pyl(\te)\left(|\potm-\pot|+|\potl-\pot|\right)\\
	&\leq \pyl(\te)C_{\Phi}(\te)\left(\lno \eff_\ell(\te)-\eff(\te)\rno_\mathsf{Y}+\lno \eff_{\ell-1}(\te)-\eff(\te)\rno_\mathsf{Y}\right)\\
	&\leq \pyl(\te)C_{\Phi}(\te)C_{\eff}(\te)\left(1+s^{\alpha}\right)s^{-\alpha \ell}.  \label{eq:bound_I}
\end{align}
We now shift our attention to $II$. Following a similar procedure as for $I$, we have that   \begin{align}
	II&\leq\frac{\pyll(\te)}{Z_\ell}  \int_\Te \lv e^{-\Phi_\ell(z;\yy)}-e^{-\Phi_{\ell-1}(z;\yy)} \rv \mup(\diff z) \\
	&  \leq \frac{\pyll(\te)}{Z_\ell}  \int_\Te e^{-\min\{\Phi_\ell(z;\yy),\Phi_{\ell-1}(z;\yy)\}}\lv \Phi_\ell(z;\yy)-\Phi_{\ell-1}(z;\yy) \rv \mup(\diff z)\\
&\leq \pyll(\te) \left(1+s^{\alpha}\right)s^{-\alpha \ell}c_I^{-1} \int_\Te C_{\Phi}(z)C_{\eff}(z)e^{-\min\{\Phi_\ell(z;\yy),\Phi_{\ell-1}(z;\yy)\}} \mup(\diff z)\\
&\leq \pyll(\te) \left(1+s^{\alpha}\right)s^{-\alpha \ell}c_I^{-1} \int_\Te C_{\Phi}(z)C_{\eff}(z) \mup(\diff z)\\
&\leq \pyll(\te) \left(1+s^{\alpha}\right)s^{-\alpha \ell}c_I^{-1} C_e \label{eq:bound_II},
\end{align}
where in the last step we have used that $$\int_\Te C_\Phi(\te)C_\eff(\te)\mup(\diff\te)\leq\int_\Te\exp(C_\Phi(\te)C_\eff(\te))\mup(\diff\te)\leq C_\mathrm{e}.$$
Adding \eqref{eq:bound_I} and \eqref{eq:bound_II} gives the desired result with $$C'_{\pi,\ell}(\te)=\left(\pyl(\te)C_{\Phi}(\te)C_{\eff}(\te)+\pyll(\te)c_I^{-1} C_e\right)\left(1+s^{\alpha}\right).$$ The case $\potm>\Phi_{\ell-1}(\te;y)$ can be treated analogously by considering the alternative splitting 
\begin{align}
|\pi^y_\ell(\te)-\pi^y_{\ell-1}(\te)|&\leq\left(\left \lvert\frac{e^{-\potm}}{Z_\ell}-\frac{e^{-\potm}}{Z_{\ell-1}}\right \rvert +\left \lvert\frac{e^{-\potm}}{Z_{\ell-1}}-\frac{e^{-\potl}}{Z_{\ell-1}}\right \rvert\right),
\end{align}
which yields the constant $C''_{\pi,\ell}(\te)=\left(\pyll(\te)C_{\Phi}(\te)C_{\eff}(\te)+\pyl(\te)c_I^{-1} C_e\right)\left(1+s^{\alpha}\right)$. Thus, one can obtain the desired bound $|\pi^y_\ell(\te)	-\pi^y_{\ell-1}(\te)| \leq C_{\pi,\ell}(\te) s^{-\alpha\ell}$ with 

\begin{align}\label{eq:c_pi_ell}
C_{\pi,\ell}(\te)&=(\pyll(\te)+\pyl(\te))K_{\pi,\ell}(\te),\\
K_{\pi,\ell}(\te)&=(C_\Phi(\te)C_\eff(\te)+c_I^{-1}C_\mathrm{e})(1+s^\alpha).
\end{align}
A similar procedure shows that the bound \eqref{eq:bound_ass_pi} holds with
 \begin{align}\label{eq:tilde_c_pi_ell}
 \tilde C_{\pi,\ell}(\te)&=(\py(\te)+\pyl(\te))\tilde{K}_{\pi,\ell}(\te),\\
\tilde K_{\pi,\ell}(\te)&=C_\Phi(\te)C_\eff(\te)+c_I^{-1}C_\mathrm{e}.
\end{align}
We finally remark that 
\begin{align}
&K_p:=\left(\int_\X |K_{\pi,\ell}(\te)|^p\mup(\diff \te)\right)^{1/p}=(1+s^\alpha)\left(\int_\X(C_\Phi(\te)C_\eff(\te)+c_I^{-1}C_\mathrm{e})^p \mup(\diff \te) \right)^{1/p}\\
&\leq (1+s^\alpha)\left(\frac{p}{e}\right)\left( \int_\X \exp\left\{C_\Phi(\te)C_\eff(\te)+c_I^{-1}C_e \right\}\mup(\diff \te)\right)^{1/p} \ \ \text{(using $x^p\leq \left(\frac{p}{e}\right)^pe^x$)}\\
&\leq (1+s^\alpha)\left(\frac{p}{e}\right)\left(C_e\exp\left\{ c_I^{-1}C_e\right\}\right)^{1/p}<+\infty.
\end{align}
A similar calculation for $\tilde K_{\pi,\ell}$ leads to \begin{align}
	\tilde K_p=\left(\int_\X |\tilde K_{\pi,\ell}(\te)|^p\mup(\diff \te)\right)^{1/p}\leq (p/e)\left(C_e\exp\left\{c_I^{-1}C_e\right\}\right)^{1/p}<+\infty. 
\end{align}

\end{proof}


We can now show implication \ref{Assumption:for_cost_thm1}. 
\begin{lemma}\label{lemma:verification:1}
Suppose Assumption \emph{\ref{Assumptions_phi}} holds. Then, for any $\ell=0,1,\dots\lev,$ there exists a positive constant $C_w\in \R_+$, independent of $\ell$, such that: $$|\E_{\pi^y_{\ell}}[\Q_\ell(\te )]-\E_{\pi^y}[\Q(\te)]|\leq C_ws^{-\alpha_w\ell},$$
	with $\alpha_w=\min\{{\alpha_q},\alpha\}$ and $\alpha_q,\alpha$ as in Assumption \emph{\ref{Assumptions_phi}}. 
	\begin{proof}
		We follow an approach similar to that of \cite{dodwell2015hierarchical}. 
		\begin{align}
		\lv\E_{\my_{\ell}}[\Q_\ell(\te )]-\E_{\my}[\Q(\te)]\rv \nonumber &\leq  \lv\E_{\my_{\ell}}[\Q_\ell(\te )]  -\E_{\my_{\ell}}[\Q(\te)] \rv \\	&+	   \lv\E_{\my_{\ell}}[\Q(\te)]  -\E_{\my}[\Q(\te)] \rv \nonumber. 
		\end{align}	
		For the first term, we have that 
		\begin{align}
		&\lv\E_{\my_{\ell}}[\Q_\ell(\te )]  -\E_{\my_{\ell}}[\Q(\te)] \rv\leq  \E_{\myl}[\lv\Q_\ell(\te )-\Q(\te)\rv]\\
		&\leq \left( \int_\X C_q(\te)\myl(\diff\te)\right)s^{-\alpha_q \ell}\leq \frac{s^{-\alpha_q\ell}}{Z_\ell}\int_\X C_q(\te)\mup(\diff \te)\leq c_I^{-1}\tilde C_q s^{-\alpha_q \ell}. \label{eq:bound_eq}
		\end{align}
For the second term, we have that 
		\begin{align}
		&\lv\E_{\myl}[\Q(\te)]  -\E_{\my}[\Q(\te)] \rv= \lv \int_{\Te} \Q(\te)[\pi^y_\ell(\te)-\pi^y(\te)]\mup(\mathrm{d}\te)\rv \\
			&\leq \int_\X \lv\Q(\te)\rv(\pyl(\te)+\py(\te))\tilde K_{\pi,\ell}(\te)\mup(\diff \te) s^{-\alpha \ell}.\label{eq:int_int}\end{align}
			Working on the first term of the previous integral, we obtain  from H\"older's inequality that
			\begin{align}
				&\lv \int_\X \qoi(\te)\pyl(\te)\tilde K_{\pi,\ell}(\te)\mup(\diff \te)\rv\\
				&\leq \left(\int_\X|\qoi(\te)|^m \mup(\diff \te)\right)^{1/m}\left(\int_\X\pyl(\te)|\tilde K_{\pi,\ell}(\te)|^{m'} \mup(\diff \te)\right)^{1/m'}\\
				&\leq C_\mathrm{m} c_I^{-1}\tilde K_{m'},
			\end{align}
		where we have taken $m$ as in Assumption \ref{Assumptions_phi}, $m'=1-1/m$ and $\tilde K_{m'}$  as in Lemma \ref{prop:distance_densities}. A similar bound holds for the second term in \eqref{eq:int_int}, thus leading to
		\begin{align}
	\lv\E_{\myl}[\Q(\te)]  -\E_{\my}[\Q(\te)] \rv	& \leq  2c_I^{-1}C_\mathrm{m}\tilde K_{m'}s^{-\alpha\ell}.\label{eq:bound_e{q'}}
		\end{align}
  The desired result follows from \eqref{eq:bound_eq} and \eqref{eq:bound_e{q'}}, with $\alpha_w=\min\{\alpha_q,\alpha\},$ and  a level independent constant $C_w= c_I^{-1}(2C_\textrm{m}\tilde K_{m'}+\tilde C_q).$ 
	\end{proof}
\end{lemma}

We now shift our attention to implication \ref{Assumption:for_cost_thm2}. We first prove several auxiliary results. 

For any given level $\ell=0,1,\dots, \lev$, we say that the joint chains  created by Algorithm \ref{algo:2-IS_coupling} are \emph{synchronized} at step $n$ if $\te^n_{\ell,\ell}=\te^n_{\ell,\ell-1}$. Conversely, we say they are \emph{unsynchronized} at step $n$ if $\te^n_{\ell,\ell}\neq\te^n_{\ell,\ell-1}$. Notice that if the chains are synchronized at a state $\teln=\telln=\te$, and the new proposed state at the $(n+1)^\mathrm{th}$ iteration of the algorithm in $z\in \Te,$ they de-synchronize at the next step with probability $|\alpha_{\ell}(\te,z)-\alpha_{\ell-1}(\te,z)|$ (c.f. Figure \ref{fig:drawing}). Intuitively, one would expect that such a probability goes to 0 as $\ell\to \infty$. We formalize this intuition below.

\begin{lemma}\label{lemma:desync}
	Suppose  Assumptions \emph{\ref{Assumptions_phi1}}  hold. Then, the following bound holds \begin{align}
	\lv \alpha_\ell(\te,z)-\alpha_{\ell-1}(\te,z)\rv\leq h_\ell(\te,z) s^{-\alpha \ell},\quad \te,z\in\Te,
	\end{align}
	with 
	\begin{align}
	h_\ell(\te,z):=\frac{Q_\ell(\te)}{Q_\ell(z)}\frac{1}{\pyl(\te)\pyll(\te)}\lv \pyl(z)C_{\pi,\ell}(\te) +\pyl(\te)C_{\pi,\ell}(z)\rv
	\end{align}
	and $C_{\pi,\ell}(\cdot)$ as in Lemma \ref{prop:distance_densities}.
	\begin{proof}
		From the definition of $\alpha_\ell$, and the fact that $\psi(x):=\min\{1,x\}$ is Lipschitz continuous with constant 1, it can be seen that 
		\begin{align}
		&\lv \alpha_\ell(\te,z)-\alpha_{\ell-1}(\te,z)\rv
		\leq\lv  \frac{Q_\ell(\te)}{Q_\ell(z)}\frac{\pi^y_\ell(z)}{\pi^y_\ell(\te)}-\frac{Q_\ell(\te)}{Q_\ell(z)}\frac{\pi^y_{\ell-1}(z)}{\pi^y_{\ell-1}(\te)}\rv =\frac{Q_\ell(\te)}{Q_\ell(z)}\lv \frac{\pi^y_\ell(z)}{\pi^y_\ell(\te)} -  \frac{\pi^y_{\ell-1}(z)}{\pi^y_{\ell-1}(\te)} \rv\\
		& =\frac{Q_\ell(\te)}{Q_\ell(z)}\frac{1}{\pyl(\te)\pyll(\te)}\lv \pyl(z)(-\pyl(\te)+\pyll(\te)) +\pyl(\te)(\pyl(z)-\pyll(z))\rv\\
		&\leq\frac{Q_\ell(\te)}{Q_\ell(z)}\frac{1}{\pyl(\te)\pyll(\te)}\left( \pyl(z)C_{\pi,\ell}(\te) +\pyl(\te)C_{\pi,\ell}(z)\right) s^{-\alpha \ell}\label{eq:bound_desync}
		\end{align}	
	\end{proof}
\end{lemma}
\begin{lemma}\label{eq:kern:bound}
	Suppose  Assumptions \emph{\ref{Ass:positivity}} and \emph{\ref{Assumptions_phi}} hold and denote the diagonal set of $\X^2$ as $\Delta:=\{ (\te,z)\in \Te^2 \text{ s.t. } \te=z \}.$ The transition probability to $\Delta^c$ for the coupled chain of Algorithm \ref{algo:2-IS_coupling} is such that
	\begin{align}
	\pisl(\tebl,\Delta^c)\leq R_\ell(\te) s^{-\alpha \ell},&\quad \forall \tebl=(\te,\te)\in\Delta,
	\end{align}
with \begin{align}R_\ell(\te)=\frac{Q_\ell(\te)}{\pyl(\te)\pyll(\te)}\left( C_{\pi,\ell}(\te)+\pyl(\te)K_1\right),
\end{align}and $C_{\pi,\ell}(\cdot)$ and $K_1$ as in Lemma \ref{prop:distance_densities}. Moreover, whenever $\tebl\in\Delta^c,$
\begin{align}
\pisl(\tebl,\Delta^c)\leq 1-c\int_\X\min\{ \pyl(\te),\pyll(\te) \}\mup(\diff \te),
\end{align}
where $c$ is the same constant as in Assumption \ref{Ass:essinf}.  Furthermore, $\exists \delta>0$ independent of $\ell$ such that \begin{align}
\underset{\ell\in\mathbb{N}}{\inf}	\int_\X\min\{ \pyl(\te),\pyll(\te) \}\mup(\diff \te)>\delta>0.
\label{eq:int_bd}	\end{align}

	\begin{proof}
		We begin with the first inequality.	For $\tebl\in\Delta$ and from the definition of $\pisl$ we obtain 
		\begin{align}
		&\pisl(\tebl,\Delta^c)=\int_\Te (\alpha_{\ell-1}(\tell,z)-\alpha_{\ell}(\tel,z))^+\ql(z)\car{(z,\tel)\in \Delta^c}\mup(\diff z)\\
		&+\int_\Te (\alpha_\ell(\tel,z)-\alpha_{\ell-1}(\tell,z))^+\ql(z)\car{(\tell,z)\in \Delta^c}\mup(\diff z),\label{eq:kern_bound_1}
		\end{align}
		where the first and last term in \eqref{eq:joint_kernel} are both 0. Writing $\tel=\tell=\te$, it then follows from Lemma \ref{lemma:desync} that: 
		\begin{align}
		\pisl(\tebl,\Delta^c)&\leq \int_\Te |\alpha_{\ell-1}(\te,z)-\alpha_{\ell}(\te,z)|\ql(z)\mup(\diff z)\\
		&\leq\frac{Q_\ell(\te)  s^{-\alpha \ell}}{\pyl(\te)\pyll(\te)}\int_\Te \lv \pyl(z)C_{\pi,\ell}(\te) +\pyl(\te)C_{\pi,\ell}(z)\rv \mup(\diff z)\\
		&\leq \frac{Q_\ell(\te)  s^{-\alpha \ell}}{\pyl(\te)\pyll(\te)}\left( C_{\pi,\ell}(\te)+\pyl(\te)\int_\X C_{\pi,\ell}(z)\mup(\diff z)\right)\\
		&\leq \frac{Q_\ell(\te)  s^{-\alpha \ell}}{\pyl(\te)\pyll(\te)} \left( C_{\pi,\ell}(\te)+2c_I^{-1}K_1\pyl(\te)\right)
		\end{align}
Thus, one has that $\forall \tebl\in \Delta$, 	\begin{align}
	\pisl(\tebl,\Delta^c)\leq R_\ell(\te) s^{-\alpha \ell},
\end{align}
with \begin{align}R_\ell(\te)=\frac{Q_\ell(\te)}{\pyl(\te)\pyll(\te)}\left( C_{\pi,\ell}(\te)+2\pyl(\te)c_I^{-1}K_1\right).
\end{align}
		We now focus on the second inequality which holds for $\tebl\in\Delta^c$. Thus, from  the fact that $\max\{a,b\}-|a-b|=\min\{a,b\}$ $\forall a,b\in\R$ and using Assumption \ref{Ass:essinf},  we obtain
	\begin{align}
		\pisl(\tebl,\Delta^c)&\leq \int_\Te\left(1- \min_{j=\ell-1,\ell}\{\alpha_{j}(\telj,u)\}\right)\ql(u)\mup(\diff u)\\
		&\leq 1-\int_\Te \min_{j=\ell-1,\ell}\left[\min\left\{1,c\frac{\py_j(u)}{\ql(u)}\right\}\right]\ql(u)\mup(\diff u) \\		
		&=1-\int_\Te\min_{j=\ell-1,\ell}\left[\min\left\{Q_\ell(u),c\py_j(u)\right\}\right]\mup(\diff u)\\
		&=1-\int_\Te\ \min_{j=\ell-1,\ell}\left[\min\left\{\frac{Q_\ell(u)}{\py_j(u)},c\right\} \py_j(u)\right]\mup(\diff u)\\
		&\leq 1-c \int_\Te\min_{j=\ell-1,\ell}\left\{ \py_j(u)\right\}\mup(\diff u).
		\end{align}
		where $c$ is the same constant as in Assumption  \ref{Ass:essinf} (notice that $c<1$). 
		
Finally, we  show that the integral term in the previous equation is lower bounded by a strictly positive constant independent of the $\ell$. First notice that
\begin{align}
\lim_{\ell\to\infty} \int_\Te\min_{j=\ell-1,\ell}\left\{ \py_j(\te)\right\}\mup(\diff \te)&=1-\lim_{\ell\to\infty} \frac{1}{2}\int_\Te |\pyl(\te)-\pyll(\te)|\mup(\diff \te)\\ &\geq \lim_{\ell\to\infty} (1 -K_1s^{-\alpha \ell}) =1,\label{eq:bound_p4}
\end{align}	
and, by definition, \begin{align}
\int_\Te\min_{j=\ell-1,\ell}\left\{ \py_j(\te)\right\}\mup(\diff \te)\leq 1, \quad \forall \ell\in\mathbb{N}	
\end{align}
 Thus, the sequence $\{\int_\Te\min_{j=\ell-1,\ell}\left\{ \py_j(\te)\right\}\mup(\diff \te)\}_{\ell\in\mathbb{N}}$ has 1 as an accumulation point, as $\ell\to \infty$, and there exists $\delta'>0$ and $\ell'\geq0$ such that, for any $\ell\geq\ell'$ $\int_\Te\min_{j=\ell-1,\ell}\left\{ \py_j(\te)\right\}\mup(\diff \te)\}\geq\delta'$. Lastly, recall that by Assumption \ref{Ass:post_is_positive} $\py_\ell$ and $\pyll$ are continuous and strictly positive. Thus, for any compact set $A\subset \X$ with $\mup(A)>0$, and for any $\ell=\{0,1,\dots,\ell'\}$, we have 
 \begin{align}
 	\int_\Te\min_{j=\ell-1,\ell}\left\{ \py_j(\te)\right\}\mup(\diff \te)\geq \int_A\min_{j=\ell-1,\ell}\left\{ \py_j(\te)\right\}\mup(\diff \te)=:\delta_\ell>0.
 \end{align}
 Thus setting $\hat \delta=\min_{0\leq\ell\leq\ell'}\{\hat{\delta}_\ell\}$, and $\delta=\min\{\hat \delta, \delta'\}$ we obtain that, for any $\ell\geq0$
 \begin{align}
	\int_\Te\min_{j=\ell-1,\ell}\left\{ \py_j(\te)\right\}\mup(\diff \te)\geq\delta>0.
\end{align}

	\end{proof}
	
\end{lemma}


\begin{lemma}\label{lemma:prob_bnd}
	 Suppose  Assumptions \emph{\ref{Ass:positivity}} and \emph{\ref{Assumptions_phi}} hold. Then, for all $\ell=1,2,\dots,\lev,$  there exist a positive constant $C_{r,\ell}$ with $C_{r,\ell}\to C_r^*>0 $ as $\ell \to \infty$, such that
	\begin{align}\mathbb{P}_{\nul}(\tel^n\neq \tell^n)\leq C_{r,\ell} s^{-\alpha\ell}, \quad \forall n\in\mathbb{N},\end{align}
 with $c$ as in Assumption \ref{Ass:essinf} and $r$ as in Assumption \ref{Ass:p_moments}.
	\begin{proof}
For notational simplicity, for the remainder of this proof we will write $\pr_n:=\mathbb{P}_{\nul}(\tell^n\neq\tel^n)$, $\telln,\teln\in\Te,\ n\in\mathbb{N}.$ Let $Z_{\Delta,\ell}:=\int_{\Delta}\nul(\diff \tebl)=(1-\pr_n)$. From  Lemma \ref{eq:kern:bound} we obtain, for any $n\in\mathbb{N}$ 
\begin{align}
&\mathbb{P}_{\nul}(\tel^{n+1}\neq \tell^{n+1}|\tebln\in\Delta)=Z_{\Delta,\ell}^{-1}\int_{\Delta}\pisl(\tebl,\Delta^c)\nul(\diff \tebl)\\
&\leq\frac{s^{-\alpha \ell}}{Z_{\Delta,\ell}}\int_\Delta R_\ell(\te)\nul(\diff \tebl)\quad \text{(with $\tebl=(\te,\te)$ on $\Delta$)}\\
&\leq \frac{s^{-\alpha\ell}}{Z_{\Delta,\ell}}\underbrace{\int_\Delta \frac{Q_\ell(\te)}{\pyll(\te)}(K_{\pi,\ell}(\te)+2c_I^{-1}K_1)\nul(\diff \tebl)}_\text{I}+\frac{s^{-\alpha\ell}}{Z_{\Delta,\ell}}\underbrace{\int_\Delta\frac{Q_\ell(\te)}{\pyl(\te)}K_{\pi,\ell}(\te)\nul(\diff \tebl)}_\text{II}
\end{align}
We begin with integral I: 
\begin{align}
\text{I}&=\int_{\X^2}\frac{Q_\ell(\tell)}{\pyll(\tell)}(K_{\pi,\ell}(\tell)+2c_I^{-1}K_1)\car{(\tell,\tel)\in\Delta}\nul(\diff \tebl)\\
&\leq \int_{\X}\frac{Q_\ell(\tell)}{\pyll(\tell)}(K_{\pi,\ell}(\tell)+2c_I^{-1}K_1)\int_\X\nul(\diff \tell, \diff \tel)\\
&= \int_{\X}Q_\ell(\tell)(K_{\pi,\ell}(\tell)+2c_I^{-1}K_1)\mup(\diff \tell)\\
&\leq \left(\int_\X |Q_\ell(\tell)|^r\mup(\diff\tell)\right)^{1/r}\left(2c_I^{-1}K_1+\left(\int_\X |K_{\pi,\ell}(\tell)|^{r'}\mup(\diff \tell)\right)^{1/r'}\right)\\
&=C_r(2c_I^{-1}+1)K_{r'.}
\end{align}
Similarly, for II, we get:
\begin{align}
\textrm{II}&=\int_{\X^2}\frac{Q_\ell(\tel)}{\pyl(\tel)}K_{\pi,\ell}(\tel)\car{(\tell,\tel)\in\Delta}\nul(\diff \tell,\diff \tel)\\
&\leq\int_\X Q_\ell(\tel)K_{\pi,\ell}(\tel)\mup(\diff \tel)\leq C_rK_{r'}.
\end{align}
\noindent Setting $\hat C=2C_rK_{r'}(c_I^{-1}+1)$, one then has
$$\mathbb{P}_{\nul}(\tel^{n+1}\neq \tell^{n+1}|\tel^n=\tell^n)\leq \hat CZ_{\Delta,\ell}^{-1}s^{-\alpha\ell}:=Z_{\Delta,\ell}^{-1}s_\ell,$$
where we have set $s_\ell=\hat c s^{-\alpha \ell}.$
\noindent Similarly, letting  $Z_{\Delta^c,\ell}:=\int_{\Delta^c}\nul(\diff \tebl)$, one has
\begin{align}
&\mathbb{P}_{\nul}(\tel^{n+1}\neq \tell^{n+1}|\tebln\in\Delta^c)=Z_{\Delta^c,\ell}^{-1}\int_{\Delta^c}\pisl(\tebl,\Delta^c)\nul(\diff \tebl)\\
&\leq 1-c\int_\X\min\{\pyl(\te),\pyll(\te)\}\mup(\diff\te)=:\tilde c_\ell \quad \text{(From Lemma \ref{eq:kern:bound})}.
\end{align}
 We can then write the de-synchronization probability at the $(n+1)^\text{th}$ step as 
	\begin{align}
	\pr_{n+1}&=\mathbb{P}_{\nul}(\tel^{n+1}\neq \tell^{n+1})=\mathbb{P}_{\nul}(\tel^{n+1}\neq \tell^{n+1}|\tel^n=\tell^n)\mathbb{P}_{\nul}(\tel^n=\tell^n)\\
	&+\mathbb{P}_{\nul}(\tel^{n+1}\neq\tell^{n+1}|\tel^n\neq\tell^n)\mathbb{P}_{\nul}(\tel^n\neq\tell^n)\\
	&\leq Z_{\Delta,\ell}^{-1}s_\ell\underbrace{(1-\pr_n)}_\text{$=Z_{\Delta,\ell}$}+\tilde c_\ell\pr_n\\
	&\leq s_\ell+\tilde c_\ell\pr_n.
	\label{eq:prob_rec}
	\end{align}
	However, by stationarity, we have that $\pr_{n+1}=\pr_n=:\pr$. Thus, from Equation \eqref{eq:prob_rec} we get that \begin{align}
	\mathbb{P}_{\nul}(\tell^n\neq\tel^n)=\pr\leq \frac{\hat Cs^{-\alpha\ell}}{1-\tilde c_\ell}=\frac{\hat C}{c\int_\X \min\{ \pyl(\te), \pyll(\te)\}\mup(\diff\te)}s^{-\alpha \ell}.
	\end{align}
Notice that by \eqref{eq:int_bd} the integral term in the denominator is lower bounded by a constant  $\delta$ independent of the level.  Furthermore, this integral converges to 1 as $\ell\to\infty$. 
	\end{proof}

\end{lemma}
We are now ready to prove implication \ref{Assumption:for_cost_thm2}.

\begin{lemma}\label{lemma:verification:2} Suppose Assumptions \emph{\ref{Assumptions_phi}} and \emph{\ref{Ass:post_is_positive}} hold. Then, for any $\ell\geq1$,  there exists a positive constant $C_v$ such that
	$$\V_{\nul}[Y_\ell]\leq C_v s^{-\beta\ell},$$
	where $\beta=\min \left\{2\alpha_q,\alpha(1-2/m) \right\},$ and $\alpha,$ $\alpha_q$, $m$ as in Assumption \ref{Assumptions_phi}.
	\begin{proof} We follow a similar argument to that of \cite[Lemma 4.8]{dodwell2015hierarchical}. From Young's inequality we have \begin{align}
		&\V_{\nul}[Y_\ell]\leq\enu\left[( \Q_\ell(\te _{\ell,\ell})-\Q_{\ell-1}(\te_{\ell,\ell-1}))^2\right]\\
		&\leq 2\enu\left[(\Q_{\ell}(\te _{\ell,\ell})-\Q_{\ell}(\te _{\ell,\ell-1}))^2\right]+2\enu\left[(\Q_{\ell}(\te _{\ell,\ell-1})-\Q_{\ell-1}(\te _{\ell,\ell-1}))^2\right].
		\end{align}
		Notice that in the case where $\Q_\ell(\cdot)$ and $\Q_{\ell-1}(\cdot)$ are the same (which could happen when the quantity of interest, seen as a functional, is mesh-independent), the second term vanishes. Otherwise, we have, using Assumption \ref{Assumptions_phi2}, that   \[\enu\left[(\Q_{\ell}(\te _{\ell,\ell-1})-\Q_{\ell-1}(\te _{\ell,\ell-1}))^2\right]\leq 2\tilde C_{q}^2(1+s^{2\alpha_q})s^{-2\alpha_q\ell}.\] As for the first term, notice that it is only non-zero whenever $\te _{\ell,\ell}\neq\te _{\ell,\ell-1}$. Thus, it can be rewritten as 
		\begin{align}
			2\enu\left[(\Q_{\ell}(\te _{\ell,\ell})-\Q_{\ell}(\te _{\ell,\ell-1}))^2\right]=2\enu\left[(\Q_{\ell}(\te _{\ell,\ell})-\Q_{\ell}(\te _{\ell,\ell-1}))^2\mathbbm{1}_{\{ \te _{\ell,\ell}\neq\te _{\ell,\ell-1}\}}\right].
		\end{align}
		In turn, applying H\"older's inequality, we can further write the above expression as 	
		\begin{align}
	2&\enu\left[(\Q_{\ell}(\te _{\ell,\ell})-\Q_{\ell}(\te _{\ell,\ell-1}))^2\mathbbm{1}_{\{ \te _{\ell,\ell}\neq\te _{\ell,\ell-1}\}}\right]\\
		\leq
	 2&\enu\left[\lv\Q_{\ell}(\te _{\ell,\ell})-\Q_{\ell}(\te _{\ell,\ell-1})\rv^{2m/2}\right]^{2/m}\enu[\mathbbm{1}_{\{ \te _{\ell,\ell}\neq\te _{\ell,\ell-1}\}}^{{m'}}]^{1/{m'}} \ \ \text{(with $m'=m/(m-2)$)}\\
		=2&\enu\left[\lv\Q_{\ell}(\te _{\ell,\ell})-\Q_{\ell}(\te _{\ell,\ell-1})\rv^{m}\right]^{2/m}\mathbb{P}_{\nul}(\te _{\ell,\ell}\neq\te _{\ell,\ell-1})^{1/m'}.
		\label{Eq:some_eq}
		\end{align}
	From Assumption \ref{ass:bdd_2nd}, it follows that one can bound the first term in Equation (\ref{Eq:some_eq}) by
		\begin{align}
		&\enu\left[|\Q_{\ell}(\te _{\ell,\ell})-\Q_{\ell}(\te_{\ell,\ell-1}|^{m}\right]^{2/m}\\
	&\leq \left( \E_{\my_{\ell}}[\qoi_\ell(\tel)^{m}]^{\frac{1}{m}}+\E_{\my_{\ell-1}}[\qoi_\ell(\tell)^{m}]^{\frac{1}{m}}\right)^2\\
		&\leq 4c_I^{-2/m}C^2_\textrm{m}.
		\end{align}  Moreover, from Lemma  \ref{lemma:prob_bnd}, we have that $\mathbb{P}_{\nul}(\te _{\ell,\ell}\neq\te _{\ell,\ell-1})\leq C_{r,\ell}s^{-\alpha \ell}$. 
		Thus, 
		\begin{align}
		\V_{\nul}[Y _\ell]\leq C_vs^{-\beta\ell},
		\end{align}
		where $C_v=8 c_I^{-2/m}C_\textrm{m}^2 \underset{\ell\in\mathbb{N}}{\max}\{C_{r,\ell}\}+4\tilde C_q^2(1+s^{2\alpha_q}).$		
\end{proof}	
	\end{lemma}

\section{Implementation}\label{sec:cMLMCMC}
We begin with a discussion on how to choose the optimal number of samples $N_\ell$. 
Recall that,  for $\ell=0,\dots, \lev$, we denote by $\mathcal{C}_\ell$ the total cost of producing one coupled sample of $(\qoi_{\ell-1},\qoi_\ell)$ at level $\ell$ using Algorithm \ref{algo:MLMH}. 
The total cost of the multi-level MCMC estimator is then given by 
\begin{align}\label{eq:total_ml_cost}
\mathcal{C}\left(\widehat{\Q}_{\lev,\{N_{\ell}\}}\right)=\sum_{\ell=0}^\lev \mathcal{C}_\ell N_\ell.
\end{align}	
In order to bound the statistical contribution of the total error bound, we will require, from  \eqref{eq:bound_namse}  and \eqref{eq:MSEb},that 
\begin{align}\label{eq:bound_varc}
	2(\lev+1)\sum_{\ell=0}^{\lev} C_\mathrm{mse}\frac{\V_{\nu_\ell}[Y_\ell]}{N_\ell}\leq \frac{\tol^2}{2},
\end{align}
 where  $\tol$ is some user-prescribed tolerance.  However, it is in general not a simple task to compute or estimate the constant $C_\mathrm{mse}$. We will ignore it hereafter, and aim at bounding the quantity
\begin{align}\label{eq:bound_var}
2(\lev+1)\sum_{\ell=0}^{\lev} \frac{\V_{\nu_\ell}[\hat{Y}_\ell]}{N_\ell}\leq \frac{\tol^2}{2}.
\end{align}
To that end we will use the so-called \emph{batched means estimator}  of $\V_{\nu_{\ell}}[\hat{Y}_\ell]$ denoted by $\hat \sigma^2_\ell$ (see \cite{flegal2011implementing} for further details).  In this case, treating $N_\ell$ as a real number and minimizing \eqref{eq:total_ml_cost} subject to \eqref{eq:bound_var}, gives the optimal samples sizes \begin{align}\label{Eq:nl}
	N_\ell= \left \lceil 2\tol^{-2}\sqrt{ \hat\sigma^2_\ell/ \mathcal{C}_\ell}\left(\sum_{j=0}^{\lev}\sqrt{ \hat\sigma^2_j\mathcal{C}_j}\right)\right \rceil,
\end{align}
where $\lceil \cdot\rceil$ is the ceiling function.  Lastly, we also need to ensure that the second contribution to the total error, i.e., the discretization bias at level $\lev$, is such that 
\begin{align}\label{Eq:cond_we}
	\sqrt{2}\lv \E_{\muy_\lev}[\Q_{\lev}]-\E_{\muy}[\Q]\rv \leq \frac{\tol}{\sqrt{2}}. 
\end{align}
Notice that from  \ref{Assumption:for_cost_thm1} it follows
\begin{align}
	\lv \E_{\muy_\lev}[\Q_{\lev}]-\E_{\muy}[\Q]\rv& =\lv \sum_{j=\lev+1}^{\infty}\E_{\muy_j}[\Q_{j}]-\E_{\muy_{j-1}}[\Q_{j-1}]\rv
\\ &	\approx\frac{\lv [\widehat{\Q}_{\lev}-\widehat{\Q}_{\lev-1}]\rv}{1-s^{-\alpha_w}}.
\label{Eq:check_we}
\end{align}
Thus, to achieve a total (estimated) MSE of the ML-MCMC estimator  less than $\tol^2$, we need to check that \begin{align}
2(\lev+1)\left(\sum_{\ell=0}^{\lev} \frac{\hat \sigma_\ell^2}{N_\ell}\right) + 2\left(\frac{\lv[\widehat{\Q}_{\lev}-\widehat{\Q}_{\lev-1}\rv}{1-s^{-\alpha_w}}\right)^2 \leq \tol^2. \label{eq:total_error} 
\end{align}

In practice, the set of parameters $\mathcal{P}:=\{C_w,\alpha_w,\{\hat\sigma_\ell^2\}_{\ell=0}^\lev,C_\sigma,\beta,C_\gamma,\gamma\}$ need to be estimated with a preliminary run over $\lev_0$ levels, using $\tilde N_\ell,\ \ell=0,1,\dots,\lev_0$ samples per level.  However, the main disadvantage of this procedure is that for computationally expensive problems, this screening phase can be quite inefficient. In particular, if $\lev_0$ is chosen too large, then the screening phase might turn out to be more expensive than the overall ML-MCMC simulation on the optimal hierarchy $\{0,1,\dots,\lev\}$. On the other hand, if $
\lev_0$ (or $\tilde N_\ell$) is chosen too small, the extrapolation (or estimation) of the values of $\mathcal{P}$  might be quite unreliable, particularly at higher levels. In the MLMC literature, one way of overcoming these issues is with the so-called continuation Multi-level Monte Carlo method \cite{collier2015continuation}. We will present a continuation-type ML-MCMC (C-ML-MCMC) algorithm in the following subsection, based on the works \cite{collier2015continuation,pisaroni2017continuation}.
\subsection{A continuation-type ML-MCMC}
 The key idea behind this method is to iteratively implement a ML-MCMC algorithm with a sequence of decreasing tolerances while, at the same time, progressively improving the estimation of the problem dependent parameters $\mathcal{P}$. As presented before, these parameters directly control the number of levels and sample sizes. Following \cite{collier2015continuation}, we  introduce the family of tolerances $\tol_i, \ i=0,1,\dots,$ given by \begin{align}
  \tol_i=\begin{cases}
  r_1^{i_E-i}r_2^{-1}\tol & i<i_E,\\
  r_2^{i_E-i}r_2^{-1}\tol & i\geq i_E,
  \end{cases}
  \end{align}
where $r_1\geq r_2 >1,$ so that $\tol_{i_E-1}\geq\tol>\tol_{i_E}$, with \begin{align}
i_E :=\left \lfloor \frac{-\log(\tol)+\log(r_2)+\log(\tol_0)}{\log(r_1)} \right \rfloor. \label{eq:I_e}
\end{align}
The idea is then to iteratively run the ML-MCMC algorithm for each of the tolerances $\tol_i,$ $i=0,1,\dots$ until the algorithm achieves convergence, based on the criterion defined in the previous subsection. Iterations for which $i<i_E$, are used to obtain increasingly more accurate estimates of $\mathcal{P}$. Notice that when $i=i_E$, the problem is solved with a slightly smaller tolerance $r_2^{-1}\tol$ for some carefully chosen $r_2$. Solving at this slightly smaller tolerance is done in order to prevent any extra unnecessary iterations due to the statistical nature of the estimated quantities.  Furthermore, if the algorithm has not converged at the $i_E^\mathrm{th}$ iteration, it keeps running for even smaller tolerances $\tol_i$, $i>i_E$, to account for cases where the estimates of $\mathcal{P}$ are unstable. Thus,  at the $i^\mathrm{th}$ iteration of the C-ML-MCMC algorithm, we run the Algorithm \ref{algo:MLMH} with an iteration-dependent number of levels $\lev_i$, where $\lev_i$ is obtained by solving the following discrete constrained optimization problem:

\begin{align}
\begin{cases}
&\underset{\lev_{i-1}\leq\lev\leq\lev_\mathrm{max}}{\mathrm{arg}\   \mathrm{min}} \left\{2\tol_i^{-2}2(\lev+1)\left(\sum_{j=0}^{\lev_i}\sqrt{ C_\beta s^{-\beta_j}\mathcal{C}_j}\right)^2 \right\}, \\
&\textrm{s.t.} \quad  C_ws^{-\alpha_w\lev_i}\leq \frac{\tol_i}{\sqrt{2}}. 
\end{cases}\label{eq:L[i]}
\end{align}
Here,  $\lev_{-1}=\lev_\mathrm{0}$ is a given minimum number of levels,  $\lev_\mathrm{\max}$, is chosen as the maximum number of levels given a computational budget (which could be dictated, for example, by the minimum mesh size imposed by memory or computational restrictions). Furthermore, notice that \eqref{eq:L[i]} is easily solved by exhaustive search. 

We now have everything needed to implement the C-ML-MCMC algorithm, which we present in the listing \ref{algo:c_MLMH}.
\begin{algorithm}
	\caption{Continuation ML-MCMC}\label{algo:c_MLMH}
	\begin{algorithmic}[1]
		\Procedure{C-ML-MCMC}{$\{\pi^y_\ell\}_{\ell=0}^\lev,Q,\tilde N,\lev_0,\lev_\mathrm{\max},\{\nu_\ell^0\}_{\ell=0}^\lev,\tol_0,\tol,r_1,r_2$}
		\State \codecomm{Preliminary run}
				\State Compute $i_E$ according to Equation \eqref{eq:I_e}. Set $N_\ell=\tilde N$, $\ell=0,1,\dots,\lev_0$,
		\State $\{\{ \te^n_{\ell,\ell}\}_{n=0}^{\tilde N},\{ \te^n_{\ell,\ell-1}\}_{n=0}^{\tilde N_\ell}\}_{\ell=0}^{\lev_0}$=\texttt{ML-MCMC}$(\{\pi^y_\ell\}_{\ell=0}^{\lev_0},Q,\{ N_\ell\}_{\ell=0}^{\lev_0},\{\nu_\ell^0\}_{\ell=0}^{\lev_0})$.
		\State Compute estimates for the parameters $\mathcal{P}$ using least squares fit

		\State set $i=1$ and $\mathsf{te}=\infty$. 
		\State \codecomm{ Starts continuation algorithm}
		\While{$i<i_E$ \textbf{or} $\mathsf{te}>\tol$}
		\State Update tolerance  $\tol_i=\tol_{i-1}/r_k,$ where $k=1$ if $i< i_E$ and $k=2$ otherwise.
		\State Compute $\lev_i=\lev_i(\lev_{i-1},\lev_\mathrm{max},\tol_i,\mathcal{P})$ using \eqref{eq:L[i]}
		\State Compute $N_\ell=N_\ell(\lev_i,\tol_i,\mathcal{P})$ for $\ell=0,1,\dots,\lev_i$, using \eqref{Eq:nl}. 
		\State \codecomm{ Q can be constructed using samples from previous iterations}
		\State $\{\{ \te^n_{\ell,\ell}\}_{n=0}^{N_\ell},\{ \te^n_{\ell,\ell-1}\}_{n=0}^{N_\ell}\}_{\ell=0}^{\lev_i}$=\texttt{ML-MCMC}{$(\{\pi^y_\ell\}_{\ell=0}^{\mathsf{L}_i},Q,\{N_\ell\}_{\ell=0}^{\mathsf{L}_i},\{\nu^0_\ell\}_{\ell=0}^{\lev_i})$)}
		\State Update estimates for $\mathcal{P}$ using least squares fit 
		\State Update total error  $\mathsf{te}=2(\lev+1)\left(\sum_{\ell=0}^{\lev_i}\hat\sigma^2_\ell/N_\ell\right) +2\left(C_ws^{-\alpha_w \lev_i}\right)^2$
		\State $i=i+1$
		\EndWhile
		\State Return $ \widehat \Q_{\mathsf{L},\{N_\ell\}_{\ell=0}^\lev}$ computed with \eqref{eq:ML-MLMC_estimator}.
		\EndProcedure
	\end{algorithmic}
\end{algorithm}

\section{Numerical experiments}\label{sec:Numerical_experiments}

We present first  two ``sanity check'' experiments aimed at numerically verifying the theory presented in previous sections. 

In the following two experiments we will compare our proposed multi-level MCMC algorithm to that of \cite{dodwell2015hierarchical}, which, by construction, does not satisfy our Assumption \ref{Ass:essinf}. The aim of these experiments is to verify the theoretical results of the previous sections, as well as to provide a setting for which our methods might be better suited than the sub-sampling approach of \cite{dodwell2015hierarchical}. For ease of exposition, we will consider as a quantity of interest $\Q(\te)=\te, \quad \te\sim \mu^y$, and we will assume that the cost of evaluating the posterior density at each level grows as $2^{\gamma \ell}$, with $\gamma=1$. For both experiments, we implement the sub-sampling ML-MCMC algorithm of \cite{dodwell2015hierarchical} with a level-dependent sub-sampling rate $t_\ell:=\min\left\{1+2\sum_{k=0}^{N_\ell}\hat \varrho_k,5\right\}$, where $\hat \varrho_k$ is the so-called \emph{lag-$k$ auto-correlation time}  and  $1+2\sum_{k=0}^{N_\ell}\hat\varrho_k$ is the so-called \emph{integrated auto-correlation time} \cite{brooks2011handbook}.  

\subsection{Nested Gaussians}\label{ss:nested}
We begin with a scenario for which both ML-MCMC methods can be applied. In this case we aim at sampling from the family of posteriors $\mu^y_\ell=\mathcal{N}\left(1,1+2^{-\ell}\right),$ $\ell=0,1,2,\dots, $ which approximate $\mu^y=\mathcal{N}\left(1,1\right)$ as $\ell\to\infty.$  For the ML-MCMC method proposed in the current work, we will use a fixed proposal across all levels given by $Q_\ell=Q=\mathcal{N}(1,3)$. Such  proposal is chosen to guarantee that Assumption \ref{Ass:essinf} is fulfilled.  The family of posteriors and the proposal $Q$ used in our ML-MCMC algorithm are depicted in Figure \ref{fig:postpropnes}. For both algorithms, the proposal distribution at level $\ell=0$ is a random walk Metropolis proposal $Q_0(\te_0^n,\cdot)=\mathcal{N}(\te_0^n,1)$. This proposal is chosen to guarantee an acceptance rate of about 40\%, the value deemed close to optimum for MCMC in one dimension \cite{brooks2011handbook}.  

\begin{figure}[h]
	\centering
	\input{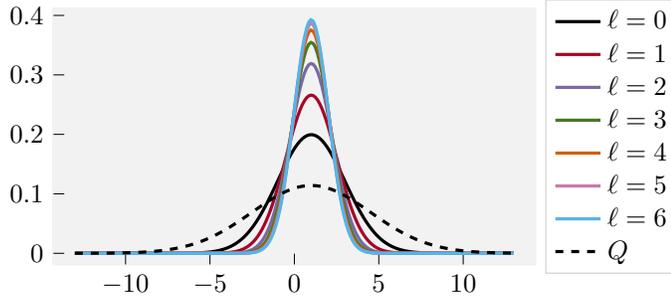}
	\caption{Family of posteriors $\mu^y_\ell$ and fixed proposal distribution $Q$ for the nested Gaussians example.}
	\label{fig:postpropnes}
\end{figure}

 As a \emph{sanity check}, we begin by verifying that both algorithms target the right marginal distribution at different levels. This can be seen in Figure \ref{fig:histsfnes}, where the histograms of samples obtained with a simple ML-MCMC algorithm with proposal $Q$ and prescribed number of levels $\lev=7$ and number of samples $N_\ell=50000$ for $\ell=0,1,\dots,\lev$ (top row) and the algorithm of \cite{dodwell2015hierarchical} (bottom row) are shown for levels $\ell=0,3,6$. The true posterior at level $\ell$ is shown in red. As it can be seen, both methods are able to sample from the right marginal distribution for the family of posteriors considered here-in. 

\begin{figure}[h]
	\centering
	\input{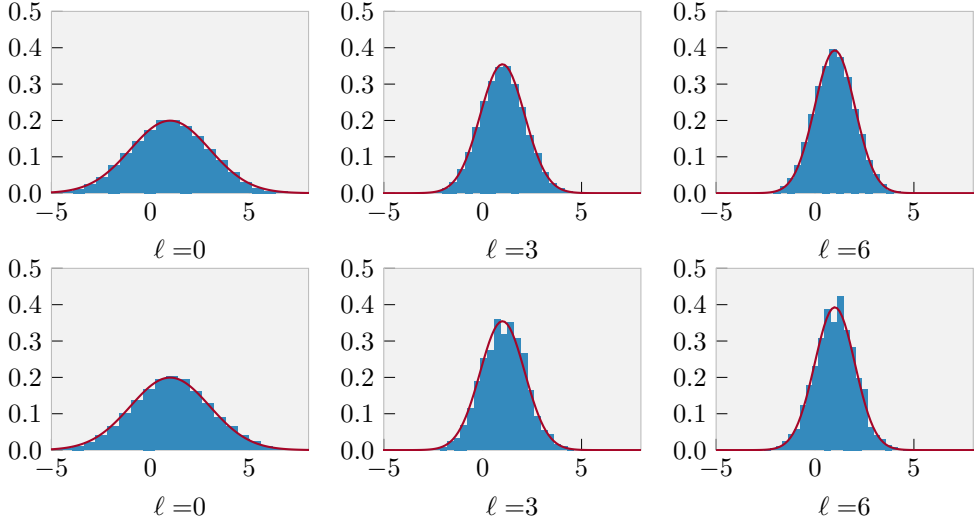}
	\caption{True posterior $\mu^y_\ell$ for different levels $\ell=0,3,6$ and histogram of the samples of $\te_\ell\sim \muy_\ell$ obtained with  the ML-MCMC algorithm described herein with $Q_\ell=\mathcal{N}(1,3)$ (Top row)and  the sub-sampling ML-MCMC algorithm (Bottom row). As we can see, both methods are able to obtain samples from the right posterior distribution.}
	\label{fig:histsfnes}
\end{figure}
We now aim at verifying the rates presented in Theorem \ref{thm:cost_thm_schichl_1}.  To that end, we run the ML-MCMC algorithm 100 independent times. For each independent run, we  obtained 50,000 samples on each level and investigate the behavior of $|\E_{\nul}[Y_\ell]|$ (Figure \ref{fig:edq100bnes} (left) ) and $\V_{\nul}[Y_\ell]$ (Figure \ref{fig:edq100bnes} (right) ) with respect to the level $\ell$. As it can be seen from Figure \ref{fig:edq100bnes}, both $|\E_{\nul}[Y_\ell]|$ and $\V_{\nul}[Y_\ell]$ decay with respect to $\ell$ with nearly the same estimated rate $\approx -1.34$  for the ML-MCMC algorithm discussed in this current work, close to the predicted one in Theorem \ref{thm:cost_thm_schichl_1}. It can be seen, however, from Figure \ref{fig:edq100bnes} (right), that the variance decay of the sub-sampling algorithm is slightly better than the one obtained by the method presented herein. This, in turn, results in a smaller overall sample size at each level for a given particular error tolerance, as it can be seen in Figure \ref{fig:syncnes} (left). We believe that this  difference in rate is due (i) to the slightly higher synchronization rate of the sub-sampling ML-MCMC algorithm (Figure \ref{fig:syncnes} (right)) and (ii) to the fact that the convergence rate of the marginal chain in the sub-sampling algorithm also increases with level, which is not necessarily the case for our method. These results suggest that, for this particular case, it is more cost-efficient to use the sub-sampling ML-MCMC algorithm.

\begin{figure}[h]
	\centering
\setcounter{subfigure}{0}

\begin{tikzpicture}

\definecolor{color0}{rgb}{0.203921568627451,0.541176470588235,0.741176470588235}
\definecolor{color1}{rgb}{0.650980392156863,0.0235294117647059,0.156862745098039}

\begin{groupplot}[group style={group size=2 by 1, horizontal sep=3cm,vertical sep=3.4cm}]

\nextgroupplot[
axis background/.style={fill=white!95!black},
width=5cm,
height=5cm,
axis line style={white!73.7254901960784!black},
legend cell align={left},
legend columns=2,
legend style={fill opacity=0, draw opacity=1, text opacity=1, at={(0.5,-0.4)}, anchor=north, draw=white!80!black, fill=white!93.3333333333333!black},
log basis y={2},
tick pos=left,
title={\(\displaystyle \mathbb{E}_{\nu_\ell}[Y_\ell]\) vs. level },
x grid style={white!69.8039215686274!black},
xlabel={\(\displaystyle \ell\)},
xmajorgrids,
xmin=-0.3, xmax=6.3,
xtick style={color=black},
xtick={-1,0,1,2,3,4,5,6,7},
xtick={0,1,2,3,4,5,6,7},
y grid style={white!69.8039215686274!black},
ylabel={\(\displaystyle \mathbb{E}_{\nu_\ell}[Y_\ell]\)},
ymajorgrids,
ymin=0.000837053736122753, ymax=1.41284188745246,
ymode=log,
ytick style={color=black},
ytick={0.0001220703125,0.00048828125,0.001953125,0.0078125,0.03125,0.125,0.5,2,8},
yticklabels={\(\displaystyle {2^{-13}}\),\(\displaystyle {2^{-11}}\),\(\displaystyle {2^{-9}}\),\(\displaystyle {2^{-7}}\),\(\displaystyle {2^{-5}}\),\(\displaystyle {2^{-3}}\),\(\displaystyle {2^{-1}}\),\(\displaystyle {2^{1}}\),\(\displaystyle {2^{3}}\)}
]

\addplot [thick, color0]
table {%
0 0.999255522000997
1 0.237356203387855
2 0.0705949080530698
3 0.0232602944438356
4 0.0087053812514944
5 0.00379635329732477
6 0.00134787019428512
};
\addlegendentry{Subsampled}
\addplot [thick, color0, dashed, forget plot]
table {%
0 1.00785088005674
1 0.260013115787397
2 0.0801236591313259
3 0.0267421802049089
4 0.00985392979271535
5 0.00413081795326141
6 0.00152232812704974
};
\addplot [thick, color0, dashed, forget plot]
table {%
0 0.990660163945255
1 0.214699290988313
2 0.0610661569748136
3 0.0197784086827623
4 0.00755683271027345
5 0.00346188864138813
6 0.00117341226152049
};
\addplot [thick, color1]
table {%
0 0.995494064724075
1 0.247556503254246
2 0.0770577786996453
3 0.0259985898728209
4 0.00933362643492613
5 0.00389488845740258
6 0.00143793785884097
};
\addlegendentry{Gaussian}
\addplot [thick, color1, dashed, forget plot]
table {%
0 1.00423156364899
1 0.271732636467071
2 0.0872863973610373
3 0.0297735574004615
4 0.0106228713228925
5 0.00424739529516164
6 0.00163109763390657
};
\addplot [thick, color1, dashed, forget plot]
table {%
0 0.986756565799161
1 0.223380370041421
2 0.0668291600382533
3 0.0222236223451803
4 0.00804438154695977
5 0.00354238161964352
6 0.00124477808377537
};

\addplot [thick, black, mark=asterisk, mark size=3, mark options={solid}]
table {%
	0 0.375794261641329
	1 0.147870070921231
	2 0.0581849168711342
	3 0.0228949951143542
	4 0.00900887772078868
	5 0.00354487421302129
	6 0.00139486111096235
};

\addlegendentry{slope -1.34}

\nextgroupplot[
width=5cm,
height=5cm,
axis background/.style={fill=white!95!black},
axis line style={white!73.7254901960784!black},
legend cell align={left},
legend columns=2,
legend style={fill opacity=0, draw opacity=1, text opacity=1, at={(0.5,-0.4)}, anchor=north, draw=white!80!black, fill=white!93.3333333333333!black},
log basis y={2},
tick pos=left,
tick pos=left,
title={\(\displaystyle \mathbb{V}_{\nu_\ell}[Y_{\ell}]\) vs. level },
x grid style={white!69.8039215686274!black},
xlabel={\(\displaystyle \ell\)},
xmajorgrids,
xmin=-0.3, xmax=6.3,
xtick style={color=black},
y grid style={white!69.8039215686274!black},
xtick={-1,0,1,2,3,4,5,6,7},
xtick={0,1,2,3,4,5,6,7},
ylabel={\(\displaystyle \mathbb{V}_{\nu_\ell}[Y_{\ell}]\)},
ymajorgrids,
ymin=0.0288793113963404, ymax=129.545416462658,
ymode=log,
ytick={0.00390625,0.015625,0.0625,0.25,1,4,16,64},
ytick style={color=black}
]
\addplot [thick, color0]
table {%
	0 85.8168286368916
	1 21.3676868652677
	2 5.4376222973837
	3 1.41505499214031
	4 0.38156293170155
	5 0.123788814737785
	6 0.0469811288344201
};
\addlegendentry{Subsampled}
\addplot [thick, color0, dashed, forget plot]
table {%
	0 87.2706638863397
	1 22.0741796302055
	2 5.77563745961365
	3 1.54953411023046
	4 0.430164940098847
	5 0.138924198686729
	6 0.0516389751419348
};
\addplot [thick, color0, dashed, forget plot]
table {%
	0 84.3629933874435
	1 20.6611941003299
	2 5.09960713515375
	3 1.28057587405017
	4 0.332960923304252
	5 0.108653430788841
	6 0.0423232825269055
};
\addplot [thick, color1]
table {%
	0 86.8108362187603
	1 21.3955114300678
	2 5.57058918374557
	3 1.60313013623767
	4 0.565732783831644
	5 0.244266608451167
	6 0.10337395570837
};
\addlegendentry{Gaussian}
\addplot [thick, color1, dashed, forget plot]
table {%
	0 88.3953748061814
	1 22.4238844443854
	2 6.07325493001747
	3 1.78689978368949
	4 0.621088094418137
	5 0.25907585165678
	6 0.108218990041471
};
\addplot [thick, color1, dashed, forget plot]
table {%
	0 85.2262976313393
	1 20.3671384157502
	2 5.06792343747367
	3 1.41936048878585
	4 0.510377473245151
	5 0.229457365245554
	6 0.0985289213752702
};
\addplot [thick, black, dash pattern=on 1pt off 3pt on 3pt off 3pt]
table {%
	0 38.965617156085
	1 12.5353978115352
	2 4.03268855370616
	3 1.29733233964285
	4 0.417357099877298
	5 0.134265479627173
	6 0.0431937518849319
};
\addlegendentry{slope -1.64}
\addplot [thick, black, mark=asterisk, mark size=3, mark options={solid}]
table {%
	0 22.9801531349097
	1 9.0776351647627
	2 3.58585339709312
	3 1.4164861609946
	4 0.559541292434246
	5 0.221030368358241
	6 0.0873115610896949
};
\addlegendentry{slope -1.34}

\end{groupplot}

\end{tikzpicture}
	\caption{(Left) $|\E_{\nul}[Y_{\ell}]|$ Vs. level. (Right) $\V_{\nul}[Y_{\ell}]$ Vs. level. In both figures, the rates were estimated over 100 independent runs, with 50,000 samples per level, on each run. Solid lines indicate the average value, dashed lines indicate 95\% confidence intervals.  }
	\label{fig:edq100bnes}
\end{figure}
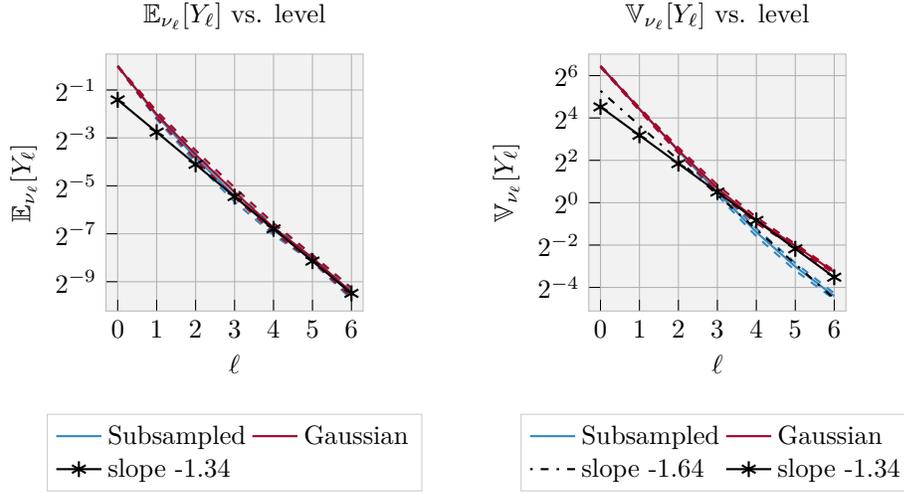

We plot sample size vs level (Figure \ref{fig:syncnes} (left)) and synchronization rate vs. level (Figure \ref{fig:syncnes} (right)). Both figures were obtained from 100 independent runs: solid lines indicate the average value and dashed lines indicate 95\% confidence intervals.  The computation of $N_\ell$ for each level $\ell=0,1,
\dots,\lev$ was done by estimating  $\hat\sigma_\ell$ with 50,000 samples per level and a tolerance $
\tol=0.07$. It can  be seen from Figure \ref{fig:syncnes} (left) that the sub-sampling algorithm requires a smaller number of samples per level.  From Figure \ref{fig:syncnes} (right) we can see that both algorithms tend to a synchronization rate of 1, as expected. It can be seen that the sub-sampling algorithm provides a slightly higher synchronization rate for the problem at hand.

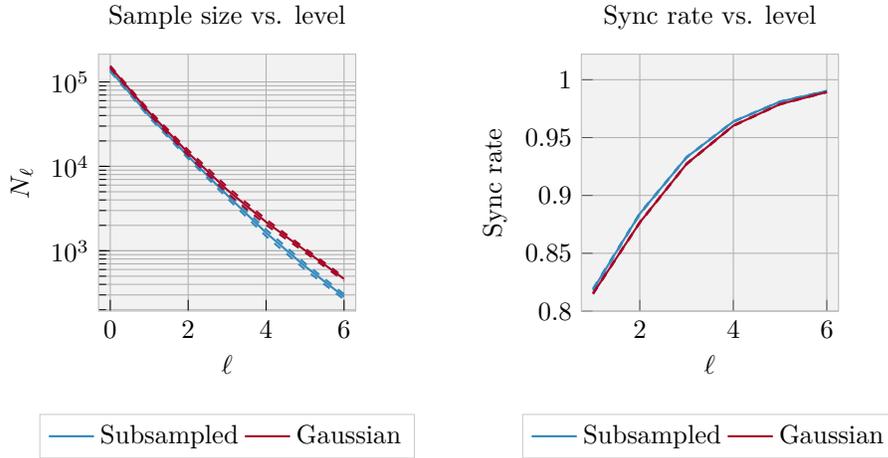
\begin{figure}[h]
	\centering
\setcounter{subfigure}{0}

\begin{tikzpicture}

\definecolor{color0}{rgb}{0.203921568627451,0.541176470588235,0.741176470588235}
\definecolor{color1}{rgb}{0.650980392156863,0.0235294117647059,0.156862745098039}

\begin{groupplot}[group style={group size=2 by 1, horizontal sep=3cm,vertical sep=3.4cm}]

\nextgroupplot[
width=5cm,
height=5cm,
axis background/.style={fill=white!95!black},
axis background/.style={fill=white!95!black},
axis line style={white!73.7254901960784!black},
legend cell align={left},
legend columns=5,
legend style={fill opacity=0, draw opacity=1, text opacity=1, at={(0.5,-0.4)}, anchor=north, draw=white!80!black, fill=white!93.3333333333333!black},
log basis y={10},
tick pos=left,
title={Sample size  vs. level },
x grid style={white!69.8039215686274!black},
xlabel={\(\displaystyle \ell\)},
xmajorgrids,
xmin=-0.3, xmax=6.3,
xtick style={color=black},
y grid style={white!69.8039215686274!black},
ylabel={\(\displaystyle N_\ell\)},
ymajorgrids,
xminorgrids,
yminorgrids,
ymin=192.987726358236, ymax=213192.71794971,
ymode=log,
ytick style={color=black}
]
\addplot [thick, color0]
table {%
	0 137084.92
	1 39735.2
	2 13133.97
	3 4556.35
	4 1633.14
	5 655.33
	6 285.64
};
\addlegendentry{Subsampled}
\addplot [thick, color0, dashed, forget plot]
table {%
	0 141018.660037936
	1 41425.2780031007
	2 13895.1660132199
	3 4895.32994768875
	4 1778.97140222765
	5 710.028811120576
	6 305.905675200239
};
\addplot [thick, color0, dashed, forget plot]
table {%
	0 133151.179962064
	1 38045.1219968993
	2 12372.7739867801
	3 4217.37005231125
	4 1487.30859777235
	5 600.631188879424
	6 265.374324799761
};
\addplot [thick, color1]
table {%
	0 150704.04
	1 43534.52
	2 14469.78
	3 5274.21
	4 2190.98
	5 1010.27
	6 464.68
};
\addlegendentry{Gaussian}
\addplot [thick, color1, dashed, forget plot]
table {%
	0 155039.783687786
	1 45725.908188639
	2 15521.9005241733
	3 5719.21925316504
	4 2353.70541547232
	5 1064.44557782681
	6 487.123278252305
};
\addplot [thick, color1, dashed, forget plot]
table {%
	0 146368.296312214
	1 41343.131811361
	2 13417.6594758267
	3 4829.20074683496
	4 2028.25458452768
	5 956.094422173193
	6 442.236721747695
};

\nextgroupplot[
width=5cm,
height=5cm,
axis background/.style={fill=white!95!black},
axis line style={white!73.7254901960784!black},
legend cell align={left},
legend columns=5,
legend style={fill opacity=0, draw opacity=1, text opacity=1, at={(0.5,-0.4)}, anchor=north, draw=white!80!black, fill=white!93.3333333333333!black},
tick pos=left,
title={Sync rate vs. level},
x grid style={white!69.8039215686274!black},
xlabel={\(\displaystyle \ell\)},
xmajorgrids,
xmin=0.75, xmax=6.25,
xtick style={color=white!33.3333333333333!black},
y grid style={white!69.8039215686274!black},
ylabel={Sync rate},
ymajorgrids,
ymin=0.8, ymax=1.02209382859157,
ytick style={color=white!33.3333333333333!black}
]
\addplot [thick, color0]
table {%
	1 0.8186514
	2 0.8841564
	3 0.9329356
	4 0.9638496
	5 0.9810664
	6 0.9903052
};
\addlegendentry{Subsampled}
\addplot [thick, color0, dashed, forget plot]
table {%
	1 0.819533597241928
	2 0.884739743780946
	3 0.93331930796605
	4 0.964104622727516
	5 0.981210557838776
	6 0.990401873303944
};
\addplot [thick, color0, dashed, forget plot]
table {%
	1 0.817769202758073
	2 0.883573056219054
	3 0.93255189203395
	4 0.963594577272485
	5 0.980922242161224
	6 0.990208526696057
};
\addplot [thick, color1]
table {%
	1 0.8152108
	2 0.876609
	3 0.9272376
	4 0.9602594
	5 0.9788738
	6 0.989338
};
\addlegendentry{Gaussian}
\addplot [thick, color1, dashed, forget plot]
table {%
	1 0.815631763622093
	2 0.877030937107939
	3 0.9276469273235
	4 0.960588179098013
	5 0.979101586538586
	6 0.98950044368378
};
\addplot [thick, color1, dashed, forget plot]
table {%
	1 0.814789836377907
	2 0.876187062892061
	3 0.9268282726765
	4 0.959930620901987
	5 0.978646013461414
	6 0.98917555631622
};

\end{groupplot}

\end{tikzpicture}
	\caption{(Left) Number of samples, Vs. level for both algorithms.(Right) Synchronization rate vs level for both algorithms.  }
	\label{fig:syncnes}
\end{figure}

Lastly, we perform some robustness  experiments for our C-ML-MCMC algorithm.  To that end, we run Algorithm \ref{algo:c_MLMH} using the same level independent proposals $Q_\ell=Q=\mathcal{N}(1,3)$ for three different prescribed tolerances $\tol=\{0.025,0.05,0.1\}$. The algorithm is run for a total of 100 independent times. At each run $k$, we compute the total squared error of the  ML estimator obtained from the $k^\mathrm{th}$ run of the C-ML-MCMC algorithm given by\begin{align}\mathsf{er}_k^2:=\left(\widehat\Q_{\mathsf{L},\{N_{\ell}\}_{\ell=0}^\mathsf{L}}^{(k)}-\my(\Q)\right)^2\label{eq:erk}\end{align} and plot it in Figure \ref{fig:reliabilitynested}. As we can see, we obtain estimators whose mean square error  is less than the prescribed tolerance, as desired. This evidences the robustness of Algorithm \ref{algo:c_MLMH} when computing quantities of interest for a given tolerance.

\begin{figure}
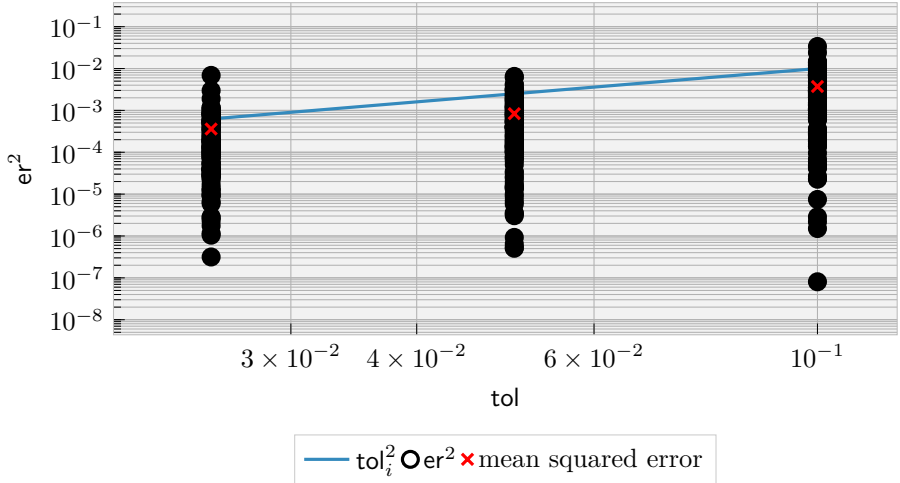

	\centering
	\include{figures/reliability_nested100}
	\caption{Total squared error $\mathsf{er}^2$ vs $\tol$ for the nested Gaussians example. Here, we used 100 independent runs of the full C-ML-MCMC algorithm for 3 different tolerances; $\tol=0.025,0.05,0.1$ (black circles). The red cross denotes the estimated MSE over those 100 runs.  }
	\label{fig:reliabilitynested}
\end{figure}
%
%
\subsection{Shifting Gaussians}\label{ss:moving}
We now move to a slightly more challenging problem, which is better suited for our proposed method. In this case, we aim at sampling from the family of posteriors $\mu^y_\ell=\mathcal{N}\left(2^{-\ell+2},1\right),$ $\ell=0,1,2,\dots, \lev,$ which approximate $\mu^y=\mathcal{N}\left(0,1\right)$ as $\ell\to\infty.$  Once again, for the ML-MCMC method proposed in the current work, we will use a fixed proposal across all levels given by $Q_\ell=Q=\mathcal{N}(2,3)$. Such  a proposal is chosen to guarantee that Assumption \ref{Ass:essinf} is fulfilled. The posterior and proposal densities are shown in Figure \ref{fig:postprop}. Just as in experiment \ref{ss:nested} the proposal distribution at level $\ell=0$ for both algorithms is a random walk Metropolis proposal $Q_0(\te_0^n,\cdot)=\mathcal{N}(\te_0^n,1)$. This proposal is chosen to guarantee an acceptance rate of about 40\%.  
\begin{figure}[h]
	\centering
	\scalebox{1}{\input{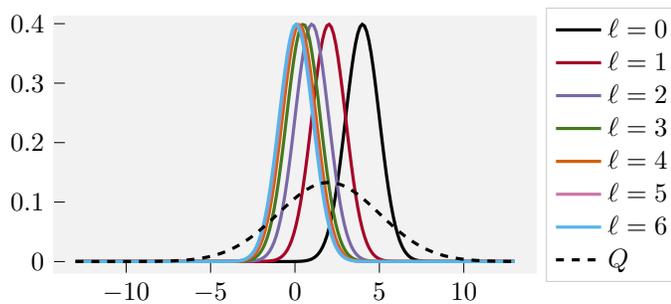}}  

	\caption{Illustration of the posterior densities $\pi_\ell$ and the proposal $Q$ for the moving Gaussians example. }
	\label{fig:postprop}
\end{figure}
Once again, we begin by investigating the correctness of the corresponding marginals. To that end we run both algorithms for $\lev=6$, obtaining $50,000$ samples per level  and plot the resulting histograms of $
\muy_\ell$ for levels $
\ell=0,3,6$.  Such results are presented in Figure \ref{fig:hists}. As it can be seen, the ML-MCMC presented herein (Figure \ref{fig:hists}, top row) is able to  sample from the correct marginals. On the contrary, the sub-sampling ML-MCMC algorithm is not able to produce samples from the correct distributions, at least for the number of samples considered, as it can be seen in Figure \ref{fig:hists} (Bottom row). We believe that this is due to  Assumption \ref{Ass:essinf} not being satisfied since there is a very small \textit{overlap} between the posterior at level 0 and the posteriors at higher levels.  Clearly, sampling from the wrong marginal distribution will in turn result in biased estimators when using the sub-sampling method \cite{dodwell2015hierarchical}.
\begin{figure}[h]
	\begin{center}
		\input{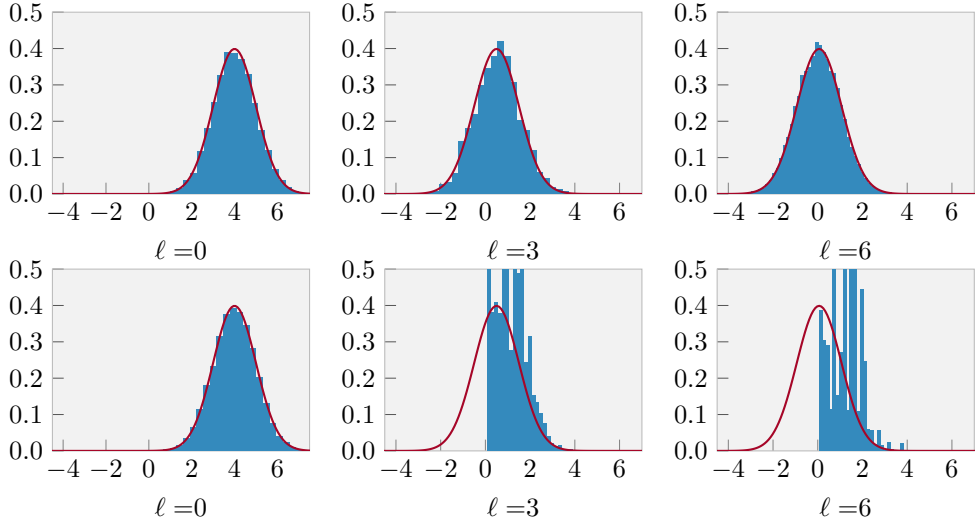}
		\caption{Sample histograms for one ML-MCMC run at levels $\ell=0,3,6$ (Top row): Fixed  Gaussian proposal . (Bottom row): Sub-sampling approach. As it can be seen, the sub-sampling approach is not able to properly sample from the posterior at higher levels. }
		\label{fig:hists}	
	\end{center}
\end{figure}

We now proceed to verify the converge rates stated in Theorem \ref{thm:cost_thm_schichl_1}. Notice that for this particular setting we have $|\E_{\mu^y_\ell}[\Q_\ell]-\E_{\muy}[\Q]|=2^{-\ell+1}$. 
 We  run Algorithm \ref{algo:MLMH} 100 independent times, obtaining 50,000 samples on each level for every run.  The accuracy of the theoretical rates in Theorem \ref{thm:cost_thm_schichl_1} is numerically verified in Figure \ref{fig:edq100b}.  However, as it can be seen in Figure \ref{fig:edq100b} (Top left) the sample mean of $\Q_\ell$ obtained with the sub-sampling algorithm does not decay as $2^{-\ell},$ which confirms the bias of the sub-sampling ML-MCMC algorithm.   The decay rates $\alpha_w$ and $\beta$, corresponding to the decay in weak and strong error, respectively, for the ML-MCMC algorithm with fixed proposals are verified to be 1, as theoretically expected, in Figure \ref{fig:edq100b}	 (Top right and Bottom left). The optimal number of samples per level is presented in Figure \ref{fig:edq100b} (Bottom left). Once again, the sub-sampling ML-MCMC provides a smaller number of samples and variances than the method presented herein, however, at the price of a biased estimator. Furthermore, it can be seen from Figure \ref{fig:sync} that the synchronization rate of both methods tends to 1 with $\ell$, as expected.

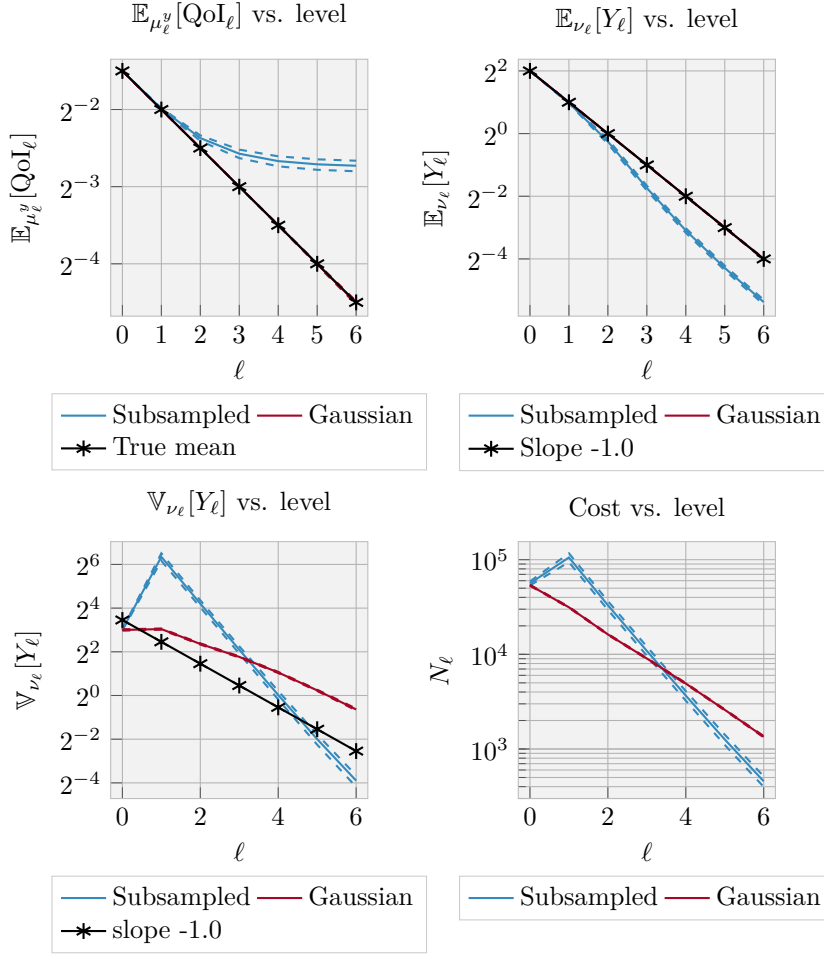
\begin{figure}[h]
\begin{center}
\begin{tikzpicture}

\definecolor{color0}{rgb}{0.203921568627451,0.541176470588235,0.741176470588235}
\definecolor{color1}{rgb}{0.650980392156863,0.0235294117647059,0.156862745098039}

\begin{groupplot}[group style={group size=2 by 2, horizontal sep=2cm,vertical sep=3cm}]
\nextgroupplot[
width=5cm,
height=5cm,
axis background/.style={fill=white!95!black},
axis line style={white!73.7254901960784!black},
legend cell align={left},
legend columns=2,
legend style={fill opacity=0, draw opacity=1, text opacity=1, at={(0.5,-0.3)}, anchor=north, draw=white!80!black, fill=white!93.3333333333333!black},
log basis y={2},
tick pos=left,
title={\(\displaystyle \mathbb{E}_{\mu^y_\ell}[\mathrm{QoI}_\ell]\) vs. level },
x grid style={white!69.8039215686274!black},
xlabel={\(\displaystyle \ell\)},
xmajorgrids,
xmin=-0.3, xmax=6.3,
xtick style={color=white!33.3333333333333!black},
xtick={-1,0,1,2,3,4,5,6,7},
xtick={0,1,2,3,4,5,6,7},
y grid style={white!69.8039215686274!black},
ylabel={\(\displaystyle \mathbb{E}_{\mu^y_\ell}[\mathrm{QoI}_\ell]\)},
ymajorgrids,
ymin=0.0488006134672991, ymax=4.9365553039539,
ymode=log,
ytick style={color=white!33.3333333333333!black},
ytick={0.0078125,0.03125,0.125,0.5,2,8,32},
yticklabels={\(\displaystyle {2^{-6}}\),\(\displaystyle {2^{-5}}\),\(\displaystyle {2^{-4}}\),\(\displaystyle {2^{-3}}\),\(\displaystyle {2^{-2}}\),\(\displaystyle {2^{-1}}\),\(\displaystyle {2^{0}}\)}
]
\addplot [thick, color0]
table {%
0 3.99726262916213
1 2.02693069679442
2 1.19692280986385
3 0.903819193973004
4 0.791416741184772
5 0.746990681685072
6 0.728078833695481
};
\addlegendentry{Subsampled}
\addplot [thick, color0, dashed, forget plot]
table {%
0 3.99987781133807
1 2.05390186790914
2 1.25489685115167
3 0.972576642524485
4 0.86252190225686
5 0.817115261493545
6 0.796992100695823
};
\addplot [thick, color0, dashed, forget plot]
table {%
0 3.99464744698619
1 1.9999595256797
2 1.13894876857604
3 0.835061745421524
4 0.720311580112685
5 0.676866101876599
6 0.659165566695139
};
\addplot [thick, color1]
table {%
0 3.999525130044
1 1.99812495551977
2 1.00072773282464
3 0.499570039897998
4 0.250134311096506
5 0.12437487937849
6 0.0620472680472171
};
\addlegendentry{Gaussian}
\addplot [thick, color1, dashed, forget plot]
table {%
0 4.00209783406948
1 1.99945358924208
2 1.00236607373799
3 0.501178510710681
4 0.251716196628502
5 0.12628259696392
6 0.0638993741477718
};
\addplot [thick, color1, dashed, forget plot]
table {%
0 3.99695242601852
1 1.99679632179746
2 0.999089391911293
3 0.497961569085316
4 0.248552425564511
5 0.12246716179306
6 0.0601951619466624
};
\addplot [thick, black, mark=asterisk, mark size=3, mark options={solid}]
table {%
0 4
1 2
2 1
3 0.5
4 0.25
5 0.125
6 0.0625
};
\addlegendentry{True mean}

\nextgroupplot[
width=5cm,
height=5cm,
axis background/.style={fill=white!95!black},
axis line style={white!73.7254901960784!black},
legend cell align={left},
legend columns=2,
legend style={fill opacity=0, draw opacity=1, text opacity=1, at={(0.5,-0.3)}, anchor=north, draw=white!80!black, fill=white!93.3333333333333!black},
log basis y={2},
tick pos=left,
title={\(\displaystyle \mathbb{E}_{\nu_\ell}[Y_\ell]\) vs. level },
x grid style={white!69.8039215686274!black},
xlabel={\(\displaystyle \ell\)},
xmajorgrids,
xmin=-0.3, xmax=6.3,
xtick style={color=white!33.3333333333333!black},
xtick={-1,0,1,2,3,4,5,6,7},
xtick={0,1,2,3,4,5,6,7},
y grid style={white!69.8039215686274!black},
ylabel={\(\displaystyle \mathbb{E}_{\nu_\ell}[Y_\ell]\)},
ymajorgrids,
ymin=0.0176167140777568, ymax=5.18197666787274,
ymode=log,
ytick style={color=white!33.3333333333333!black},
ytick={0.00390625,0.015625,0.0625,0.25,1,4,16,64},
yticklabels={\(\displaystyle {2^{-8}}\),\(\displaystyle {2^{-6}}\),\(\displaystyle {2^{-4}}\),\(\displaystyle {2^{-2}}\),\(\displaystyle {2^{0}}\),\(\displaystyle {2^{2}}\),\(\displaystyle {2^{4}}\),\(\displaystyle {2^{6}}\)}
]
\addplot [thick, color0]
table {%
0 3.99726262916213
1 1.97040201278992
2 0.838850464841833
3 0.299827989468321
4 0.117285416147902
5 0.0511879493201642
6 0.0240448273737668
};
\addlegendentry{Subsampled}
\addplot [thick, color0, dashed, forget plot]
table {%
0 3.99987781133807
1 1.99704245635464
2 0.880439503744404
3 0.315869496889439
4 0.123333747966578
5 0.0537634163680341
6 0.0252792675205191
};
\addplot [thick, color0, dashed, forget plot]
table {%
0 3.99464744698619
1 1.94376156922521
2 0.797261425939262
3 0.283786482047202
4 0.111237084329226
5 0.0486124822722943
6 0.0228103872270145
};
\addplot [thick, color1]
table {%
0 3.999525130044
1 2.0043323660813
2 0.998167309726464
3 0.500280299318091
4 0.249649564349765
5 0.12493035535732
6 0.0624779418352572
};
\addlegendentry{Gaussian}
\addplot [thick, color1, dashed, forget plot]
table {%
0 4.00209783406948
1 2.0066421217347
2 1.00028277003769
3 0.501932460116828
4 0.250972863843833
5 0.125888004829804
6 0.0632625837814278
};
\addplot [thick, color1, dashed, forget plot]
table {%
0 3.99695242601852
1 2.0020226104279
2 0.996051849415238
3 0.498628138519354
4 0.248326264855697
5 0.123972705884836
6 0.0616932998890865
};
\addplot [thick, black, mark=asterisk, mark size=3, mark options={solid}]
table {%
0 4
1 2
2 1
3 0.5
4 0.25
5 0.125
6 0.0625
};
\addlegendentry{Slope -1.0}

\nextgroupplot[
width=5cm,
height=5cm,
axis background/.style={fill=white!95!black},
axis line style={white!73.7254901960784!black},
legend cell align={left},
legend columns=2,
legend style={fill opacity=0, draw opacity=1, text opacity=1, at={(0.5,-0.3)}, anchor=north, draw=white!80!black, fill=white!93.3333333333333!black},
log basis y={2},
tick pos=left,
title={\(\displaystyle \mathbb{V}_{\nu_\ell}[Y_\ell]\) vs. level },
x grid style={white!69.8039215686274!black},
xlabel={\(\displaystyle \ell\)},
xmajorgrids,
xmin=-0.3, xmax=6.3,
xtick style={color=white!33.3333333333333!black},
y grid style={white!69.8039215686274!black},
ylabel={\(\displaystyle \mathbb{V}_{\nu_\ell}[Y_\ell]\)},
ymajorgrids,
ytick={0.00390625,0.015625,0.0625,0.25,1,4,16,64},
ymin=0.0379876241056124, ymax=131.858884474269,
ymode=log,
ytick style={color=white!33.3333333333333!black}
]
\addplot [thick, color0]
table {%
	0 8.060919544837
	1 81.4329600836222
	2 18.4456302878288
	3 4.27681612349253
	4 1.01645946192774
	5 0.250205809664223
	6 0.0670999437332436
};
\addlegendentry{Subsampled}
\addplot [thick, color0, dashed, forget plot]
table {%
	0 8.22888704100087
	1 91.0288522807099
	2 20.4759738037451
	3 4.77080232254551
	4 1.15282347642689
	5 0.290473841243459
	6 0.0791733149799155
};
\addplot [thick, color0, dashed, forget plot]
table {%
	0 7.89295204867313
	1 71.8370678865345
	2 16.4152867719126
	3 3.78282992443955
	4 0.880095447428587
	5 0.209937778084988
	6 0.0550265724865717
};
\addplot [thick, color1]
table {%
	0 8.00099071124739
	1 8.22402763356956
	2 5.13385711090815
	3 3.40280618789659
	4 2.07681291231557
	5 1.17711778773915
	6 0.642121176593501
};
\addlegendentry{Gaussian}
\addplot [thick, color1, dashed, forget plot]
table {%
	0 8.16305336888074
	1 8.37708907655799
	2 5.23458078202489
	3 3.46574510073006
	4 2.11493482418455
	5 1.20073861677295
	6 0.659508473560959
};
\addplot [thick, color1, dashed, forget plot]
table {%
	0 7.83892805361403
	1 8.07096619058113
	2 5.03313343979141
	3 3.33986727506312
	4 2.03869100044658
	5 1.15349695870536
	6 0.624733879626043
};
\addplot [thick, black, mark=asterisk, mark size=3, mark options={solid}]
table {%
	0 10.9736603832761
	1 5.48683019163804
	2 2.74341509581902
	3 1.37170754790951
	4 0.685853773954755
	5 0.342926886977378
	6 0.171463443488689
};
\addlegendentry{slope -1.0}

\nextgroupplot[
width=5cm,
height=5cm,
axis background/.style={fill=white!95!black},
axis line style={white!73.7254901960784!black},
legend cell align={left},
legend columns=5,
legend style={fill opacity=0, draw opacity=1, text opacity=1, at={(0.5,-0.3)}, anchor=north, draw=white!80!black, fill=white!93.3333333333333!black},
log basis y={10},
tick pos=left,
title={Cost vs. level },
x grid style={white!69.8039215686274!black},
xlabel={\(\displaystyle \ell\)},
xmajorgrids,
xmin=-0.3, xmax=6.3,
xtick style={color=white!33.3333333333333!black},
y grid style={white!69.8039215686274!black},
ylabel={\(\displaystyle N_\ell\)},
ymajorgrids,
yminorgrids,
xminorgrids,
ymin=299.429033155489, ymax=155211.812297445,
ymode=log,
ytick style={color=white!33.3333333333333!black}
]
\addplot [thick, color0]
table {%
	0 56216.24
	1 105529.67
	2 33063.8
	3 10875.2
	4 3665.07
	5 1261.39
	6 455.71
};
\addlegendentry{Subsampled}
\addplot [thick, color0, dashed, forget plot]
table {%
	0 58932.7160832329
	1 116824.103774503
	2 36585.4157449298
	3 12066.887427426
	4 4090.96534918045
	5 1418.36933967014
	6 513.600360140041
};
\addplot [thick, color0, dashed, forget plot]
table {%
	0 53499.7639167671
	1 94235.236225497
	2 29542.1842550702
	3 9683.51257257403
	4 3239.17465081955
	5 1104.41066032987
	6 397.819639859959
};
\addplot [thick, color1]
table {%
	0 53669.29
	1 31418.69
	2 16249.75
	3 9041.61
	4 4913.34
	5 2594.15
	6 1349.16
};
\addlegendentry{Gaussian}
\addplot [thick, color1, dashed, forget plot]
table {%
	0 54285.9187658403
	1 31751.036384556
	2 16431.432304214
	3 9147.655690567
	4 4970.46736283374
	5 2624.34929564526
	6 1371.22005641828
};
\addplot [thick, color1, dashed, forget plot]
table {%
	0 53052.6612341597
	1 31086.343615444
	2 16068.067695786
	3 8935.564309433
	4 4856.21263716626
	5 2563.95070435474
	6 1327.09994358172
};

\end{groupplot}

\end{tikzpicture}
\caption{(Top left) Estimated  expected value of $\Q_\ell$ for both ML-MCMC algorithms and the true mean of $\Q_\ell$ for different values of $\ell$. (Top right) Expected value of $Y_\ell=\Q_\ell-\Q_{\ell-1}$ obtained with both algorithms for different values of $\ell$. (Bottom left): Variance of $Y_\ell$ obtained with both algorithms for different values of $\ell$. (Bottom right): Number of samples per level for each method with $\tol=0.07.$ On all plots, dashed lines represent a 95\% confidence interval estimated over 100 independent runs of each algorithm.  }
\label{fig:edq100b}	
\end{center}
\end{figure}

\begin{figure}[h]
	\begin{center}
\begin{tikzpicture}

\definecolor{color0}{rgb}{0.203921568627451,0.541176470588235,0.741176470588235}
\definecolor{color1}{rgb}{0.650980392156863,0.0235294117647059,0.156862745098039}

\begin{axis}[
width=8cm,
height=6cm,
axis background/.style={fill=white!95!black},
axis line style={white!73.7254901960784!black},
legend style={fill opacity=0, draw opacity=1, text opacity=1, at={(0.9,-0.3)}, draw=white!80!black, fill=white!93.3333333333333!black},
tick pos=left,
title={Sync rate vs. level},
x grid style={white!69.8039215686274!black},
xlabel={\(\displaystyle \ell\)},
xmajorgrids,
xmin=0.75, xmax=6.25,
xtick style={color=white!33.3333333333333!black},
y grid style={white!69.8039215686274!black},
ylabel={Sync rate },
ymajorgrids,
ymin=0.113440332156888, ymax=1.02209382859157,
ytick style={color=white!33.3333333333333!black}
]
\addplot [thick, color0]
table {%
1 0.156949
2 0.492164
3 0.784014
4 0.908954
5 0.95854
6 0.98
};
\addlegendentry{Subsampled}
\addplot [thick, color0, dashed, forget plot]
table {%
1 0.15915523618699
2 0.505630325232635
3 0.787563725605855
4 0.911208026792188
5 0.95997753032953
6 0.98079139693545
};
\addplot [thick, color0, dashed, forget plot]
table {%
1 0.15474276381301
2 0.478697674767365
3 0.780464274394145
4 0.906699973207811
5 0.95710246967047
6 0.97920860306455
};
\addplot [thick, color1]
table {%
1 0.223574
2 0.521021
3 0.731118
4 0.855322
5 0.92525
6 0.961363
};
\addlegendentry{Gaussian}
\addplot [thick, color1, dashed, forget plot]
table {%
1 0.223819279258446
2 0.521383235839839
3 0.731669789815868
4 0.855709793615543
5 0.925643850770927
6 0.961704067089729
};
\addplot [thick, color1, dashed, forget plot]
table {%
1 0.223328720741554
2 0.520658764160161
3 0.730566210184132
4 0.854934206384457
5 0.924856149229073
6 0.961021932910271
};
\end{axis}

\end{tikzpicture}
		\caption{Synchronization rate for both algorithms. Dashed lines represent a 95\% confidence interval.  As expected, the chains become more and more synchronized as the number of levels increases. }
		\label{fig:sync}	
	\end{center}
\end{figure}
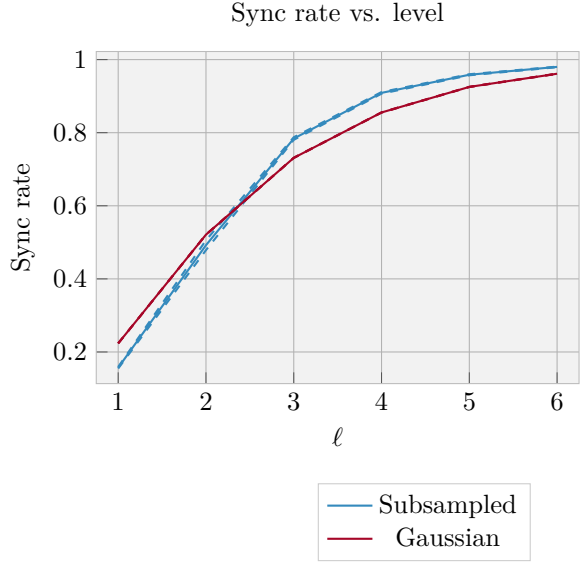

Lastly, we once again perform some robustness experiments for our C-ML-MCMC algorithm.  To that end, we run Algorithm \ref{algo:c_MLMH} using the same level independent proposals $Q_\ell=Q=\mathcal{N}(2,3)$ for three different prescribed tolerances $\tol=\{0.1,0.07,0.06\}$ for a total of 100 independent runs. Similar as in the previous example, for each independent run $k$ of the C-ML-MCMC algorithm, we compute  $\mathsf{er}_k^2$ as in \eqref{eq:erk} and plot it in Figure \ref{fig:reliabilitymovingerrortot}. Once again, we obtain estimators whose mean square error is close to the prescribed tolerance, as desired. This further evidences the robustness of Algorithm \ref{algo:c_MLMH}. 

\begin{figure}
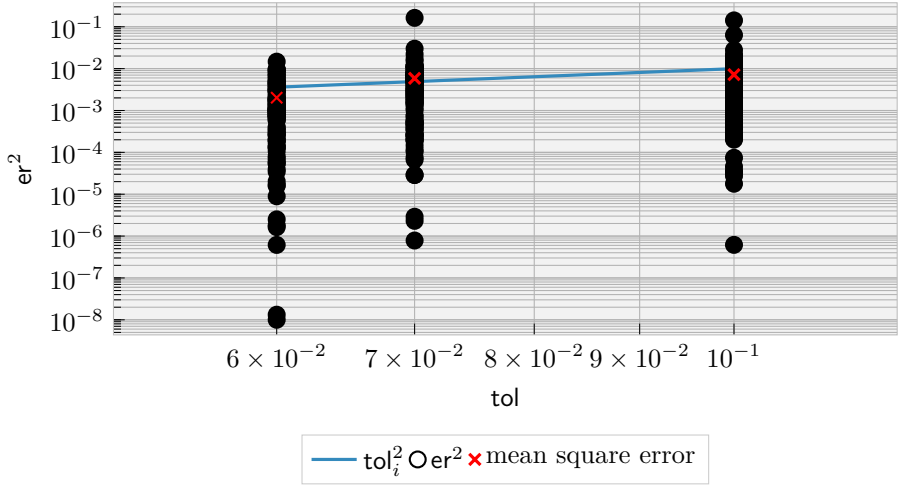

	\centering
	\include{figures/reliability_moving_error_tot}
	\caption{Total squared error $\mathsf{er}^2$ vs $\tol$ for the moving Gaussian example. }
	\label{fig:reliabilitymovingerrortot}
\end{figure}

\subsection{Subsurface flow} 
We consider a slightly more challenging problem for which we try to recover the probability distribution of the stochastic permeability field in Darcy's subsurface flow equation \eqref{eq:darcy}, based on some noise-polluted measured data. In particular, let $\bar{D}=[0,1]^2$, $\Te=\R^4,$ $(x_1,x_2)=:x\in 
\bar D$, $\partial D=\Gamma_N\cup \Gamma_D$, with $\mathring{\Gamma}_N\cap \mathring{\Gamma}_D=\emptyset, $ where $\Gamma_D:=\{(x_1,x_2)\in \partial D, \text{ s.t. } x_1=\{0,1\}\}$, and $\Gamma_N=\partial D\backslash\Gamma_D$. Darcy's subsurface  equation is given by \begin{align}\label{eq:darcy}
\begin{cases}
-\nabla_x \cdot (\kappa(x,\te)\nabla_x u(x,\te))=1,&x\in D,\  \te\in \Te,\\
u(x,\te)=0 &x\in\Gamma_D,\ \te\in \Te,\\
\partial_n u(x,\te)=0 &x\in\Gamma_N,\ \te\in\Te,\\
\end{cases}
\end{align}
where $u$ represents the pressure (or hydraulic head), and  we model the stochastic permeability $\kappa(x,\te)$  for $(\te_1,\te_2,\te_3,\te_4)=:\te\in\Te,$ as \begin{align}
\kappa(x,\te)=\exp\left(\te_1\cos(\pi x)+\frac{\te_2}{2}\sin(\pi x)+\frac{\te_3}{3}\cos(2\pi x)+ \frac{\te_4}{4}\sin(2\pi x)\right), \ 
\end{align}
with $\te_i\sim\mathcal{N}(0,1),\ i=1,2,3,4.$ Data $y$ is modeled by the solution of Equation \eqref{eq:darcy} observed at a grid of $9\times 9$ equally-spaced points in $D$ (hence $\mathsf{Y}=\R^{9\times 9}$) and polluted by a normally-distributed noise  $\eta\sim\mathcal{N}(0,\sigma_\mathrm{noise}^2I_{81\times81})$, with $\sigma_\mathrm{noise}=0.004$, which corresponds to approximately 1\% noise and $I_{81\times 81}$ is the 81-dimensional identity matrix. At each discretization level $\ell\geq 0$, the solution to Equation \eqref{eq:darcy} is numerically approximated using the finite element method on a triangular mesh of $2^\ell \cdot16 \times 2^\ell\cdot 16$ elements, which is computationally implemented using the FEniCS library \cite{Logg2012automated}. Such a library includes optimal solvers for the forward model, for which $\gamma$ can be reasonably taken equal to 1. Thus, the map $\te\mapsto \eff_\ell(\te)$ is to be understood as the numerical solution of Equation \eqref{eq:darcy} at a discretization level $\ell$, observed at a grid of $9\times9$ equally spaced points, for a particular value of $\te\in \X$. This, in turn, induces a level dependent potential \begin{align}
\potm:=\frac{1}{2\sigma^2_\mathrm{noise}}\lno y-\mathcal{F}_\ell (\te)\rno^2,
\end{align}
and prior $\mup=\mathcal{N}(0,I_{4\times 4})$. In the above expressions, $\lno \cdot \rno $ denotes the Frobenius norm on $\R^{9\times 9}$. Given that we are on a finite-dimensional setting, $\mup$ has a density with respect to the Lebesgue measure, and as such, we can define the un-normalized posterior density $\tilde{\pi}^y_\ell:\Te\mapsto\R_+$ w.r.t the Lebesgue measure  given by \begin{align}
\tilde{\pi}^y_\ell(\te)=\exp\left(-\potm -\frac{1}{2}\te^T\te\right).
\end{align}

As a quantity of interest we consider the average pressure over the physical domain, that is, $\Q(\te)=\int_D u(x,\te)\diff x$. We implement our ML-MCMC algorithm to approximate $\E_{\my}[\Q]$. In particular, we use RWM at level 0 with Gaussian proposals $\mathcal{N}(0,\sigma_\mathrm{rwm}^2 I_{4\times4})$ with step-size $\sigma_\mathrm{rwm}=0.05$, which produces an acceptance rate of  about 24\%. For the proposal $Q_\ell$ at higher levels $\ell\geq 1$, we use a mixture between the prior and a KDE obtained from the samples obtained at the previous level $\ell-1$.  This choice of mixture is made so that Assumption \ref{Ass:positivity} holds.

We begin by numerically verifying the converge rates stated in Theorem \ref{thm:cost_thm_schichl_1}. To that end, we  run Algorithm \ref{algo:MLMH} 20 independent times, obtaining 10,000 samples per run at each level $\ell=0,1,2,3$. We plot the obtained rates in Figure \ref{fig:rateselliptic}. As we can see, we numerically verify that $\alpha_w=\beta(=2.0)$, as predicted by our theory.

Lastly, we once again perform some robustness experiments for our C-ML-MCMC algorithm, with $\lev_\mathrm{max}=3$.  To that end, we first estimate  $\E_{\my}[\qoi]\approx{\mu}^y_4(\qoi_4)$ by performing 50 independent runs of a single-level MCMC algorithm at a discretization level $\ell=4$, obtaining 2000 samples on each simulation. In particular, each independent run implements a RWM sampler, using  proposals given by  $\mathcal{N}(0,\sigma_\mathrm{rwm}^2 I_{4\times4})$ with step-size $\sigma_\mathrm{rwm}=0.05$, which produces an acceptance rate of  about 21\%. We run Algorithm \ref{algo:c_MLMH} using the same mixture of independent proposals as before for different tolerance levels $\tol=\{1.1\times 10^{-4},2.0\times 10^{-4},3.0\times10^{-4}\}$. The C-ML-MCMC  algorithm is run 20 independent times for each tolerance $\tol_i$.  For each independent run $k=1,2,\dots,20$, let $\widehat{\qoi}_{\lev^{(k)}(\tol_i),\{N_\ell\}_{\ell=0}^{\lev^{(k)}},\tol_i}^{(k)},$ with  $\lev(\tol_i)\leq 3$, denote the ML estimator obtained from the $k^\mathrm{th}$ run at tolerance $\tol_i$. We compute the (approximate) total error squared $\tilde{\mathsf{er}}^2_{i,k}$ at the $k^\mathrm{th}$ run with a tolerance $\tol_i$ as \begin{align}
\tilde{\mathsf{er}}^2_{i,k}=\left(\widehat{\qoi}_{\lev^{(k)}(\tol_i),\{N_\ell\}_{\ell=0}^{\lev^{(k)}},\tol_i}^{(k)}-\tilde{\mu}^y_4(\qoi_4)\right)^2\label{eq:true_error},
\end{align} and plot it vs a given tolerance in Figure \ref{fig:reliabilityell}. As expected, the MSE of the obtained estimators is less than the prescribed tolerance.  This further evidences the robustness of Algorithm \ref{algo:c_MLMH}.

\begin{figure}
	\centering
\begin{tikzpicture}

\definecolor{color0}{rgb}{0.203921568627451,0.541176470588235,0.741176470588235}
\definecolor{color1}{rgb}{0.75,0,0.75}

\begin{axis}[
width=8cm,
height=6cm,
axis background/.style={fill=white!93.3333333333333!black},
axis line style={white!95!black},
legend cell align={left},
legend columns=2,
legend style={fill opacity=0, draw opacity=1, text opacity=1, at={(0.5,-0.4)}, anchor=north, draw=white!80!black},
log basis y={10},
tick pos=left,
x grid style={white!69.8039215686274!black},
xmajorgrids,
xmin=-0.15, xmax=3.15,
xlabel={\(\displaystyle \ell\)},
xtick style={color=black},
y grid style={white!69.8039215686274!black},
ymajorgrids,
ymin=1.07448669969657e-07, ymax=0.00279965667608173,
ymode=log,
ytick style={color=black}
]
\addplot [thick, color0]
table {%
1 0.000426078353619276
2 0.000116366074932034
3 2.64435230291633e-05
};
\addlegendentry{$\E_{\nul}[Y_\ell]$}
\addplot [thick, color0, dashed, forget plot]
table {%
1 0.000447058673852746
2 0.00012346083072775
3 2.96235003759226e-05
};
\addplot [thick, color0, dashed, forget plot]
table {%
1 0.000405098033385806
2 0.000109271319136319
3 2.3263545682404e-05
};
\addplot [thick, black, dash pattern=on 1pt off 3pt on 3pt off 3pt]
table {%
0 0.00176352605231291
1 0.000439336010362632
2 0.000109448981345191
3 2.72663274463034e-05
};
\addlegendentry{Slope -2.0}
\addplot [thick, red]
table {%
1 4.57968977946817e-06
2 8.03393963158888e-07
3 2.8319368919162e-07
};
\addlegendentry{$\V_{\nul}[Y_\ell]$}
\addplot [thick, red, dashed, forget plot]
table {%
1 5.69636547367441e-06
2 9.41696234938823e-07
3 3.95809015890597e-07
};
\addplot [thick, red, dashed, forget plot]
table {%
1 3.46301408526193e-06
2 6.65091691378953e-07
3 1.70578362492644e-07
};
\addplot [thick, color1, dash pattern=on 1pt off 3pt on 3pt off 3pt]
table {%
0 1.63946401279725e-05
1 4.07685854847251e-06
2 1.01379325770592e-06
3 2.52100178887751e-07
};
\addlegendentry{Slope -2.0}
\end{axis}

\end{tikzpicture}
	\caption{Decays of $\E_{\nul}[Y_\ell]$ and $\V_{\nul}[Y_\ell]$ vs level $\ell$.As we can see, both quantities decay with the same rate, as predicted by the theory.  }
	\label{fig:rateselliptic}
\end{figure}
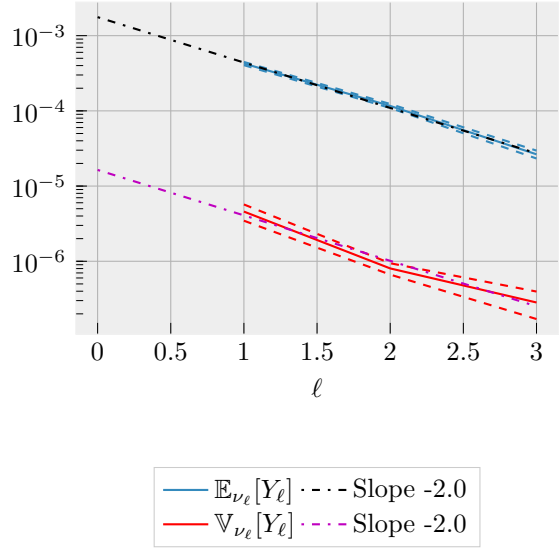

\begin{figure}
	\centering
\begin{tikzpicture}

\definecolor{color0}{rgb}{0.203921568627451,0.541176470588235,0.741176470588235}
\definecolor{color1}{rgb}{0.650980392156863,0.0235294117647059,0.156862745098039}

\begin{axis}[
width=12cm,
height=6cm,
legend cell align={left},
legend columns=6,
legend style={fill opacity=0, draw opacity=1, text opacity=1, at={(0.5,-0.3)}, anchor=north, draw=white!80!black, fill=white!93.3333333333333!black},
axis background/.style={fill=white!95!black},
axis line style={white!73.7254901960784!black},
tick pos=left,
x grid style={white!69.8039215686274!black},
xlabel={\(\displaystyle \mathsf{tol}\)},
xmajorgrids,
xmin=0.000098, xmax=0.0004051,
xminorgrids,
xmode=log,
xtick style={color=black},
xtick={0.0001,0.0002,0.0003},
xticklabels={\(\displaystyle 1 \times{10^{-4}}\),\(\displaystyle 2 \times{10^{-4}}\),\(\displaystyle 3 \times{10^{-4}}\)},
y grid style={white!69.8039215686274!black},
ylabel={\(\displaystyle \tilde{\mathsf{er}}^2\)},
ymajorgrids,
ymin=0.8e-11, ymax=1.1e-5,
yminorgrids,
ymode=log,
ytick style={color=black},
ytick={1e-12,1e-11,1e-10,1e-09,1e-08,1e-07,1e-06,1e-05},
]\addplot [very thick, color0]
table {%
	0.0003 9e-08
	0.0002 4e-08
	0.00011 1.21e-08
};
\addlegendentry{$\mathsf{tol}_i^2$}
\addplot [very thick, black, mark=*, mark size=3, mark options={solid}, only marks]
table {%
	0.0003 9.46922670067477e-09
	0.0003 1.05967989914282e-09
	0.0003 4.19725613352761e-08
	0.0003 1.18016153999048e-08
	0.0003 6.77394290278104e-08
	0.0003 2.84648183737865e-08
	0.0003 1.09906798224605e-10
	0.0003 8.79287442845093e-08
	0.0003 1.8862987614321e-07
	0.0003 9.31938032149984e-07
	0.0003 1.61530020370322e-07
	0.0003 2.22028491529928e-08
	0.0003 1.89190938999794e-08
	0.0003 1.86144872231239e-08
	0.0003 8.53499890493266e-09
	0.0003 1.705333913797e-07
	0.0003 3.75650527079845e-08
	0.0003 1.83381287093795e-08
};
\addlegendentry{$\tilde{\mathsf{er}}^2$}
\addplot [very thick, red, mark=x, mark size=3, mark options={solid}, only marks]
table {%
	0.0003 9.12741575299685e-08
};
\addlegendentry{mean square error}
\addplot [very thick, black, mark=*, mark size=3, mark options={solid}, only marks, forget plot]
table {%
	0.0002 1.75521211238324e-07
	0.0002 2.44200210116066e-09
	0.0002 1.41478249778103e-09
	0.0002 7.90644230423393e-10
	0.0002 2.54846403714752e-09
	0.0002 3.03914971823585e-08
	0.0002 3.2468403578352e-08
	0.0002 2.64946499479065e-08
	0.0002 1.42756134150276e-08
	0.0002 2.06909994147866e-07
	0.0002 1.36627954773557e-07
	0.0002 2.79270200442405e-09
};
\addplot [very thick, red, mark=x, mark size=3, mark options={solid}, only marks, forget plot]
table {%
	0.0002 3.16338959577164e-08
};
\addplot [very thick, black, mark=*, mark size=3, mark options={solid}, only marks, forget plot]
table {%
	0.00011 2.50222261609816e-10
	0.00011 1.81506221562821e-10
	0.00011 2.83994940034313e-10
	0.00011 1.35127822728221e-08
	0.00011 1.35127822728221e-08
	0.00011 3.1932303810578e-10
	0.00011 1.05972463417951e-09
	0.00011 2.83226276448502e-08
	0.00011 1.21071163749061e-09
	0.00011 6.53352424638937e-10
	0.00011 2.85513280217712e-11
	0.00011 2.85513280217712e-11
	0.00011 1.12232051716249e-09
	0.00011 7.81526337570117e-09
	0.00011 3.55299713143508e-08
	0.00011 5.09842306615693e-10
	0.00011 1.34289496276454e-09
	0.00011 7.69859115221088e-10
	0.00011 2.0127949708717e-09
	0.00011 3.70031977261368e-09
};
\addplot [very thick, red, mark=x, mark size=3, mark options={solid}, only marks, forget plot]
table {%
	0.00011 5.60836981697304e-09
};
\end{axis}

\end{tikzpicture}
	\caption{Computed squared error $\mathsf{er}^2$ (using Equation \ref{eq:true_error}) vs $\tol$ for the elliptic PDE example.}
	\label{fig:reliabilityell}
\end{figure}
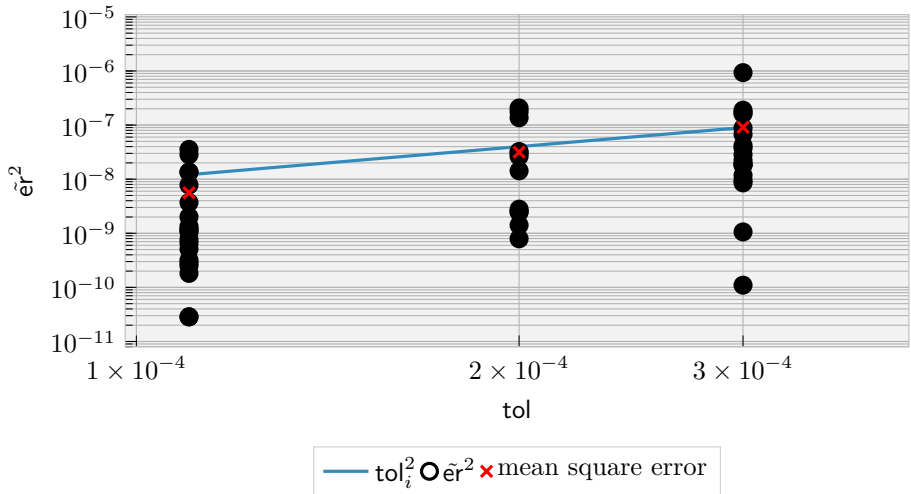

\section{Conclusions and future outlook}
\label{sec:conclussions}
In the current work we have presented a general ML-MCMC framework based on independent Metropolis-Hastings proposals, which can be understood as an extension  of currently existing methods (\cite{dodwell2015hierarchical}). Furthermore, we have provided what is, to the best of the authors knowledge, the first thorough study on the theory of ML-MCMC algorithms based on independent proposals; we have investigated necessary conditions for the existence and convergence to a unique invariant probability measure for the methods discussed herein, as well as presented sufficient conditions under which the cost-tolerance result presented in \cite{dodwell2015hierarchical} can be generalized to our method. Furthermore, we have verified our theoretical results on three academic examples. Lastly, we have presented a non-asymptotic  bound on the statistical MSE in the spirit of \cite{rudolf2011explicit} under minimal assumptions on the Markov transition operator. We remark that such a bound can be applied outside the context of ML-MCMC.

We intend to carry out a number of future extensions of the work presented herein, both from a theoretical and computational perspective. On the one hand, from a  theoretical perspective, there are still some unanswered questions regarding the class of methods at hand.  In particular, it is not clear for us if the multi-level Markov transition operator $\Pisl$ satisfies Assumptions \ref{R1} and \ref{R2} in Theorem \ref{thm:rudolf}, which would in turn produce a sharper bound on the statistical component of the MSE.  In addition, given the \textit{unattainable} nature of the invariant joint probability measure of the ML-MCMC algorithm, it is also unclear at this point how such a measure depends on the choice of proposal. 

 On the other hand, from a computational perspective,  one future research direction  stems from the construction of the level $\ell$ proposal $Q_\ell$. Intuitively, such a proposal should \textit{resemble} the posterior $\muy_{\ell-1}$ or $\muy_\ell$ (or a convex combination of both) as much as possible, while at the same time satisfying Assumption \ref{Ass:essinf}. Some preliminary experiments suggest that density approximation techniques, such as kernel density approximation (KDE) \cite{silverman1986density} or, more recently, flow-based generative models \cite{papamakarios2019normalizing}, can be used to construct efficient ML-MCMC proposals. Finally, we would also be interested in extending these results to a multi-index setting as in  \cite{haji2016multi}. 

\section*{Acknowledgments}
The authors would like to thank Dr. Sebastian Krumscheid, Dr. Panagiotis Tsilifis and Dr. Anamika Pandey for fruitful discussions during the very early stages of this work. 
\appendix
\section{Bounding the statistical MSE}\label{sec:mse_bound}
 We now present a non-asymptotic bound of the mean square error of the ergodic estimator $\hat{f}_{N,n_b}=\frac{1}{N}\sum_{i=1}^N f(\te^{n+n_b})$ obtained with a non-reversible Markov chain $\{\te^n\}$. We remark that the bound presented  in this work has  utility beyond the multi-level construction.  Throughout this section we will let $(\mathsf{X},\lno \cdot \rno_{\mathsf{X}})$ be a separable Banach space with associated Borel $\sigma$-algebra $\mathcal{B}(\mathsf{X})$. We will also let $P:L_2(\mathsf{X},\nu)\mapsto L_2(\mathsf{X},\nu)$ be a $\nu$-invariant Markov operator for some non-trivial probability measure $\nu$ on $(\mathsf{X},\mathcal{B}(\mathsf{X})).$ Lastly,  we define the \emph{averaging operator} as $\nuop f:=\nu(f)=\int_{\mathsf{X}}f(\te)\nu(\diff\te).$  

\begin{theorem}[Non-asymptotic bound on the mean square error]\label{thm:namsegen} 
	 Let $f\in  L_2(\mathsf{X},\nu)$, be a $\nu$-square integrable function and write $g(\te)=f(\te)-\int_{\mathsf{X}}f(\te)\nu(\diff \te)$. Suppose the Markov operator $P$ is uniformly ergodic and non-necessarily reversible. In addition, assume the chain generated by $P$ starts from an initial probability measure $\nu^0\ll\nu$, with $\frac{\diff \nu^0}{\diff \nu}\in L_\infty(\mathsf{X},\nu)$.  Then,
	\begin{align}\label{eq:mse_bound}
		\mathrm{MSE}(\hat{f}_{N,n_{b }};\nu^0)=\mathsf{E}_{\nu^0, P}\left|\frac{1}{N}\sum_{n=1}^{N}g(\te^{n+n_{b }})\right|^2\leq\frac{\V_{\nu}[f]}{N}(C_{\mathrm{inv} }+ C_{\mathrm{ns} }),
	\end{align}	
	where $C_{\mathrm{inv} }=\left(1+\frac{4}{\gamma_\mathrm{ps}[P]}\right),$ $C_{\mathrm{ns} }=\left(2\lno\frac{\diff \nu^0}{\diff \nu}-1 \rno_{L_\infty}\left(1+\frac{4}{\gamma_{\mathrm{ps}}[P]}\right)\right)$, where $\gamma_{\mathrm{ps}}[ P]$  is the pseudo-spectral gap of $P,$ defined in \eqref{eq:psg}. 
\end{theorem}

The proof of Theorem \ref{thm:namsegen} is decomposed into a series of auxiliary results. 
We now present a first bound of the form \eqref{eq:mse_bound} for chains which are started at stationarity, i.e., whenever $\nu^0=\nu$. Although this is usually not the case, the following Lemma will be useful in the proof of Theorem \ref{thm:namsegen}. 
\begin{lemma}[MSE bound starting at stationarity]\label{lemma:mse_stationarity}
Under the same assumptions as in Theorem \ref{thm:namsegen} and with $\nu^0=\nu,$  it holds
\begin{align}
	\mathrm{MSE}(\hat{f}_{N,n_{b }=0};\nu):=\mathsf{E}_{\nu, P}\left|\frac{1}{N}\sum_{n=1}^{N}g(\te^{n+n_{b }})\right|^2\leq \frac{\V_{\nu}[f]}{N}\left(1+\frac{4}{\gamma_\mathrm{ps}[P]}\right).
\end{align}

\begin{proof} We follow a similar approach to those presented in \cite[Theorem 3.2]{paulin2015concentration} and \cite[Section 3]{rudolf2009explicit}. In order to ease notation, for the remainder of this proof we write $ L_q= L_q(\mathsf{X},\nu), \ q\in[1,\infty]$. We can write the MSE of a Markov chain generated by $P$ starting at $\nu$ as 
\begin{align}
\mathsf{E}_{\nu, P}\left|\frac{1}{N}\sum_{n=1}^{N}g(\te^n)\right|^2=\frac{1}{N^2}\sum_{n=1}^{N}\mathsf{E}_{\nu, P}[g(\te^n)^2]+\frac{2}{N^2}\sum_{j=1}^{N-1}\sum_{i=j+1}^{N}\mathsf{E}_{\nu, P}[g(\te^i)g(\te^j)]. \label{eq:mse_g}
\end{align}
Working on the expectation of the second term on the right hand side we get from the Cauchy-Schwarz inequality  that
\begin{align}
\mathsf{E}_{\nu, P}[g(\te^i)g(\te^j)]&=\langle g,P^{i-j}g\rangle_{\nu}=\langle g,(  P -\nuop)^{i-j}g\rangle_{\nu}\\
&\leq \lno g\rno_{ L_2}^2 \lno (  P -\nuop)^{i-j}\rno_{ L_2\mapsto  L_2}.
\end{align}	
Notice that for any $k\geq 1$,  we have 
\begin{align}
\lno (  P -\nuop)^{i-j}\rno_{ L_2\mapsto  L_2}&\leq\lno (  P -\nuop)^k\rno_{ L_2\mapsto  L_2}^{\lfloor\frac{i-j}{k}\rfloor}\\
&=\lno(P^*-\nuop)^k (  P -\nuop)^k\rno_{ L_2\mapsto  L_2}^{\frac{1}{2}\lfloor\frac{i-j}{k}\rfloor} \label{eq:bound_P}
\end{align}
where $\lfloor\cdot\rfloor$ is the floor function. Now, let $k_\mathrm{ps}$ be the smallest integer such that 
\begin{align}\label{eq:k_ps}
k_\mathrm{ps}\gamma_\mathrm{ps}[P]=\gamma[( P^*)^{k_\mathrm{ps}}  P^{k_\mathrm{ps}}]=1-\lno(P^*-\nuop)^{k_\mathrm{ps}} (  P -\nuop)^{k_\mathrm{ps}}\rno_{ L_2\mapsto  L_2},
\end{align}
which, is strictly positive for uniformly ergodic chains (see \cite[Section 3.3]{roberts2004general}).  Then, from \eqref{eq:mse_g}, \eqref{eq:bound_P}, and \eqref{eq:k_ps}, we obtain: 
\begin{align}
&\frac{2}{N^2}\sum_{j=1}^{N-1}\sum_{i=j+1}^{N}\mathsf{E}_{\nu, P}[g(\te^i)g(\te^j)]\leq \frac{ 2 }{N^2}\sum_{j=1}^{N-1}\sum_{i=j+1}^{N}\lno g\rno_{ L_2}^2 \left(1-k_\mathrm{ps}\gamma_\mathrm{ps}[P]\right)^{\frac{1}{2}\lfloor\frac{i-j}{k_\mathrm{ps}}\rfloor}.
\end{align}
For notational simplicity we write $\varrho=(1-k_\mathrm{ps}\gamma_\mathrm{ps}[P])$. We then have that 
\begin{align}
\frac{ 2 }{N^2}\sum_{j=1}^{N-1}\sum_{i=j+1}^{N}\lno g\rno_{ L_2}^2 \varrho^{\frac{1}{2}\lfloor\frac{i-j}{k_\mathrm{ps}}\rfloor}&
\leq \frac{2\lno g\rno^2_{L_2}}{N}\sum_{m=0}^{\infty}\varrho^{\frac{1}{2}\lfloor \frac{m}{k_\mathrm{ps}}\rfloor}\leq \frac{2\lno g\rno^2_{L_2}k_\mathrm{ps}}{N}\sum_{m=0}^{\infty}\varrho^{\frac{1}{2}m}\\
&=\frac{2\lno g\rno^2_{L_2}k_\mathrm{ps}}{N}\frac{1}{1-\varrho^{1/2}}=\frac{2\lno g\rno^2_{L_2}k_\mathrm{ps}}{N} \frac{1+\varrho^{\frac{1}{2}}}{1-\varrho}\\
&\leq \frac{4\lno g\rno^2_{L_2}}{N\gamma_\mathrm{ps}[P]},
\end{align}
where the second inequality comes from the definition of the  floor function $\lfloor\cdot\rfloor$.
%
%
%

We now shift our attention to the first term in \eqref{eq:mse_g}. Using H\"older's inequality with $q=\infty,\ q'=1$, together with the fact that $P$ is a weak contraction in $L_q(\mathsf{X},\nu)$, for any $q\in[1,\infty]$, we obtain 
\begin{align}
\frac{1}{N^2}\sum_{n=1}^{N}\mathsf{E}_{\nu, P}[g(\te^n)^2]=\frac{1}{N^2}\sum_{n=1}^{N}\langle 1,  P ^n g^2\rangle_{\nu}\leq\frac{1}{N^2}\sum_{n=1}^{N}\lno  P ^n g^2 \rno_{L_1}\leq \frac{ \lno g^2\rno_{ L_1}}{N}.
\end{align}
Lastly, notice that $$\lno g^2\rno_{ L_1}=\int_{\mathsf{X}}|g^2(\te)|\nu(\diff \te)=\int_{\mathsf{X}}g^2(\te)\nu(\diff \te)=\lno g\rno^2_{L_2}.$$ Hence, we obtain the bound \begin{align}
\frac{1}{N^2}\sum_{n=1}^{N}\mathsf{E}_{\nu, P}[g(\te^n)^2]\leq \frac{ \lno g\rno_{ L_2}^2}{N}. \label{eq:bound2}
\end{align}  Thus, from \eqref{eq:mse_g} and \eqref{eq:bound2}, together with the observation that $\lno g\rno_{L_2}^2=\int_{\mathsf{X}}(f(\te)-\nuop(f))^2\nu(\diff \te)=\V_{\nu}[f]$, we finally obtain,
\begin{align}
\mathsf{E}_{\nu, P}\left|\frac{1}{N}\sum_{n=1}^{N}g(\te^n)\right|^2&\leq \frac{\lno g\rno_{ L_2}^2}{N} \left(1+\frac{4}{\gamma_\mathrm{ps}[P]}\right)\\
&=\frac{\V_{\nu}[f]}{N}\left(1+\frac{4}{\gamma_\mathrm{ps}[P]}\right).
\end{align}
\end{proof}
\end{lemma}
The previous result should be compared to \cite[Theorem 3.2]{paulin2015concentration} and \cite[Theorem 5]{rudolf2009explicit}. We now consider the more general case where the chain is not started from stationarity, i.e., when $\te^0\sim\nu^0$, where $\nu^0\ll\nu$ is a probability measure on $(\mathsf{X},\mathcal{B}(\mathsf{X}))$. We recall the following result from \cite{rudolf2011explicit}. 
\begin{lemma}\label{lemma:aux_mse}
Denote by $n_{b }\in\mathbb{N}$ the burn-in period and let $\{ \te^n\}_{n\in\mathbb{N}}$ be a Markov chain generated by $P$ starting from an initial measure $\nu^0$ and invariant probability measure $ \nu$, with $\nu^0\ll\nu$. Under the same assumptions as in Theorem \ref{thm:namsegen}, it holds that:
\begin{align}
\mathsf{E}_{\nu^0, P}\left|\frac{1}{N}\sum_{n=1}^{N}g(\te^{n+n_{b }})\right|^2&=\mathsf{E}_{\nu, P}\left|\frac{1}{N}\sum_{n=1}^{N}g(\te^n)\right|^2+\frac{1}{N^2}\sum_{j=1}^{N}\mathcal{H}^{j+{n_b}}(g^2)\\&+\frac{2}{N^2}\sum_{j=1}^{N-1}\sum_{k=j+1}^{N}\mathcal{H}^{j+{n_b}}(gP^{k-j} g), \label{eq:mse_ns}
\end{align} 
where \begin{align}
\mathcal{H}^{i}(h)=\left \langle ( P^i-\nuop)h,\left(\frac{\diff \nu^0}{\diff \nu}-1\right)\right \rangle_{\nu}, \quad i\in \mathbb{N}, \ h\in L_2(\mathsf{X},\nu),  \label{eq:hcal}
\end{align}
	\begin{proof}
		See \cite[Proposition 3.29]{rudolf2011explicit}.
		\end{proof}
\end{lemma}
We can now prove Theorem \ref{thm:namsegen}. 
\begin{proof}[Proof of Theorem \ref{thm:namsegen}] Once again, for the remainder of this proof we write $ L_q= L_q(\mathsf{X},\nu), \ q\in[1,\infty].$ From Lemma \ref{lemma:aux_mse}  we get 
\begin{align}
\mathcal{H}^{j+{n_b}}(g^2) &=\left \langle ( P^{j+{n_b}}-\nuop)g^2,\left(\frac{\diff \nu^0}{\diff \nu}-1\right)\right \rangle_{\nu}, \label{eq:hg2}\\
\mathcal{H}^{j+{n_b}}(gP^{k-j} g)&=\left \langle ( P^{j+{n_b}}-\nuop)(gP^{k-j} g),\left(\frac{\diff \nu^0}{\diff \nu}-1\right)\right \rangle_{\nu}.\label{eq:hgpg}
\end{align} 	
 Using H\"older's inequality with $q'=\infty$, $q=1$ on the right hand side of \eqref{eq:hg2} gives
 \begin{align}
 \mathcal{H}^{j+{n_b}}(g^2) &\leq  \lno\frac{\diff \nu^0}{\diff \nu}-1 \rno_{L_\infty} \lno ( P^{j+{n_b}}-\nuop)g^2 \rno_{L_1} \\
 &\leq  \lno\frac{\diff \nu^0}{\diff \nu}-1 \rno_{L_\infty} \lno ( P^{j+{n_b}}-\nuop)\rno_{L_1\mapsto L_1}\lno g^2 \rno_{L_1} \label{eq:hg22},
 \end{align} 
 where the last inequality comes from the definition of operator norm. Moreover, since the Markov operators are weak contractions, we have that $\lno ( P^{j+{n_b}}-\nuop)\rno_{L_1\mapsto L_1}\leq 2, \forall j\in\mathbb{N}$,
 which gives the bound 
 \begin{align}
 \mathcal{H}^{j+{n_b}}(g^2)\leq  2\lno\frac{\diff \nu^0}{\diff \nu}-1 \rno_{L_\infty} \lno g \rno_{L_2}^2.
 \end{align}
 Summing over $j$ gives 
\begin{align}
\frac{1}{N^2}\sum_{j=1}^{N} \mathcal{H}^{j+{n_b}}(g^2) \leq \frac{2\lno g \rno_{L_2}^2}{N}\lno\frac{\diff \nu^0}{\diff \nu}-1 \rno_{L_\infty}. \label{eq:hg23}
\end{align}
Following similar procedure for \eqref{eq:hgpg} we obtain
\begin{align}
 \mathcal{H}^{j+{n_b}}(g P^{k-j}g)\leq 2 \lno\frac{\diff \nu^0}{\diff \nu}-1 \rno_{L_\infty} \lno g( P^{k-j})g \rno_{L_1}. \label{eq:Hgpg2}
\end{align}
Furthermore, from H\"older's inequality (with $q'=q=2)$ and the fact that $\nuop(g)=0$ we get
 \begin{align}
\lno g( P^{k-j})g \rno_{L_1}&\leq \lno g \rno_{L_2}\lno  P^{k-j}g\rno_{L_2} 
= \lno g \rno_{L_2}\lno ( P-\nuop)^{k-j}g\rno_{L_2}\\&\leq \lno g \rno_{L_2}^2\lno ( P-\nuop)^{k-j}\rno_{L_2\mapsto L_2}\leq\lno g \rno_{L_2}^2\left(1-k_\mathrm{ps}\gamma_\mathrm{ps}[P]\right)^{\frac{1}{2}\lfloor \frac{k-j}{k_\mathrm{ps}}\rfloor},
\end{align}
where the last inequality follows from the same pseudo-spectral gap argument used in the proof of Lemma \ref{lemma:mse_stationarity}. Adding over $j$ and $k$ gives 
\begin{align}
\frac{2}{N^2}\sum_{j=1}^{N-1}\sum_{k=j+1}^{N}\mathcal{H}^{j+{n_b}}(gP^{k-j} g)\leq \frac{8\lno g\rno_{L_2}^2}{N \gamma_{\mathrm{ps}}}\lno\frac{\diff \nu^0}{\diff \nu}-1 \rno_{L_\infty} \label{eq:Hgpg3}.
\end{align}
Notice then that Equations \eqref{eq:hg23} and \eqref{eq:Hgpg3}, provide a bound on the second and third term in Lemma \ref{lemma:aux_mse}. Lastly, combining these results with Lemma \ref{lemma:mse_stationarity} and once again observing that $\lno g\rno_{L_2}^2=\V_{\nu}[f]$, gives  the desired result
\begin{align}
&\eqref{eq:mse_ns}\leq \frac{\V_{\nu}[f]}{N}\left(1+\frac{2}{\gamma_\mathrm{ps}[P]}\right) +\frac{\V_{\nu}[f]}{N}\left(2\lno\frac{\diff \nu^0}{\diff \nu}-1 \rno_{L_\infty}\left(1+\frac{4}{\gamma_{\mathrm{ps}}[ P]}\right)\right).
\end{align}
		\end{proof}

\bibliography{bib}
\bibliographystyle{plain}
\newpage

\end{document}